\documentclass[a4paper,11pt]{amsart}

\usepackage[top=30truemm,bottom=25truemm,left=35truemm,right=35truemm]{geometry}

\usepackage{amsmath}
\usepackage{amssymb}
\usepackage{amsthm}
\usepackage[abbrev,alphabetic]{amsrefs}
\usepackage{amscd}
\usepackage[dvipdfmx]{graphicx}
\usepackage{ulem}

\usepackage[all,cmtip]{xy}
\usepackage{xypic}

\theoremstyle{plain}
\newtheorem{thm}{Theorem}[section]

\newtheorem{lem}[thm]{Lemma}
\newtheorem{cor}[thm]{Corollary}
\newtheorem{conj}[thm]{Conjecture}

\theoremstyle{definition} 

\newtheorem{defi}[thm]{Definition}

\newtheorem{rmk}[thm]{Remark}

\newtheorem*{ack}{Acknowledgments}

\begin{document}



\title[Generalized mld]{On generalized minimal log discrepancy}
\author{Weichung Chen}
\author{Yoshinori Gongyo}
\author{Yusuke Nakamura}

\address{Graduate School of Mathematical Sciences, The University of Tokyo, 3-8-1 
Komaba, Meguro-ku, 
Tokyo, 153-8914, Japan.}
\email{weichung@g.ecc.u-tokyo.ac.jp}

\address{Graduate School of Mathematical Sciences, The University of Tokyo, 3-8-1 
Komaba, Meguro-ku, 
Tokyo, 153-8914, Japan.}
\email{gongyo@ms.u-tokyo.ac.jp}

\address{Graduate School of Mathematical Sciences, The University of Tokyo, 3-8-1 
Komaba, Meguro-ku, 
Tokyo, 153-8914, Japan.}
\email{nakamura@ms.u-tokyo.ac.jp}

\begin{abstract}
We discuss the ACC conjecture and the LSC conjecture for minimal log discrepancies of generalized pairs. 
We prove that some known results on these two conjectures for usual pairs are still valid for generalized pairs. 
We also discuss the theory of complements for generalized pairs. 
\end{abstract}

\maketitle

\tableofcontents

\section{Introduction}
Generalized pairs are introduced by Birkar and Zhang \cite{BZ16}, and they are generalizations of the usual log pairs. 
Generalized pairs play an essential role in Birkar's papers \cite{Bir19a, Bir16b}. 
In these papers, Birkar proves the BAB conjecture, which states that a certain class of singular Fano varieties is bounded. 
Generalized pairs also have many applications to important topics in birational geometry
(for example, the Iitaka conjecture \cite{BZ16}, and the termination problem and the existence of minimal models \cite{Mor, HMo, Mor2}).
We refer the reader to his expository article \cite{Bir20} for the motivation and related problems. 

In this article, we focus on the singularity aspect of generalized pairs, 
in particular, their minimal log discrepancies. 
The minimal log discrepancy is an invariant of singularities introduced by Shokurov (for usual log pairs) 
in order to attack the termination of flips. 
Shokurov proved that two conjectures on the minimal log discrepancies, 
the LSC (lower semi-continuity) conjecture and 
the ACC (ascending chain condition) conjecture, 
imply the conjecture of termination of flips \cite{Sho04}. 

The ACC conjecture predicts that 
the set of minimal log discrepancies of fixed dimension 
with the suitable restrictions on the coefficients of the boundary divisors satisfies the ascending chain condition. 
\begin{conj}[{ACC conjecture, \cite{Sho96}}]\label{conj:ACCu}
Let $d \in \mathbb{Z} _{>0}$ and let $I \subset [0,1]$ be a subset that satisfies the DCC. 
Then the set $A(d,I)$ defined by
\[
A(d,I) := 
\left\{ \operatorname{mld}_x (X, \Delta) \ \middle | 
\begin{array}{l}
\text{$(X, \Delta)$ is a log pair with}\\
\text{$\dim X = d$, $\Delta \in I$ and $x \in |X|_{0}$.}
\end{array}
\right \}
\]
satisfies the ACC, where $|X|_{0}$ denotes the set of all closed points of $X$. 
\end{conj}

\noindent
The ACC conjecture is proved when $d \le 2$ by Alexeev \cite{Ale93} and Shokurov \cite{Sho}. 
In \cite{Kaw15, Kaw14}, Kawakita proves the ACC conjecture on the interval $[1,3]$ for three-dimensional smooth varieties, 
and the discreteness of the minimal log discrepancies for a fixed variety. 
In \cite{Nak16b}, the third author generalizes the discreteness result by Kawakita to the varieties with fixed Gorenstein index and 
proves the ACC conjecture for canonical three-folds when $I$ is a finite set. 
We refer the reader to \cite{Amb06, MN18, Kaw, HLS, HanLuo, Mor21, Mor3} for other developments related to the ACC conjecture.

The LSC conjecture predicts that the minimal log discrepancies satisfy the lower semi-continuity. 
\begin{conj}[{LSC conjecture, \cite{Amb99}}]\label{conj:LSCu}
Let $(X, \Delta)$ be a log pair. 
Then the function 
\[
m: |X|_{0} \to \mathbb{R} \cup \{ - \infty \}; \quad x \mapsto \operatorname{mld}_{x}(X, \Delta)
\]
is lower semi-continuous, where $|X|_{0}$ denotes the set of all closed points of $X$ with the Zariski topology.
\end{conj}

\noindent
The LSC conjecture is proposed by Ambro and proved when $d \le 3$ \cite{Amb99}. 
Ein, Musta{\c{t}}{\v{a}} and Yasuda in \cite{EMY03} prove the conjecture when $X$ is smooth using the jet scheme theory, and 
Ein and Musta{\c{t}}{\v{a}} in \cite{EM04} generalize the argument to the case where $X$ is a locally complete intersection variety. 
In \cite{Nak16a}, the third author proves the conjecture when $X$ has quotient singularities, 
more generally when $X$ has a crepant resolution in the category of the Deligne-Mumford stacks. 
In \cite{NS, NS2}, the third author and Shibata prove the conjecture when $X$ has hyperquotient singularities.

A usual log pair $(X, \Delta)$ consists of a normal variety $X$ and an $\mathbb{R}$-divisor $\Delta$. 
A generalized pair $(X, \Delta + M)$ additionally has a b-divisor $M$ over $X$. 
More precisely, $M$ is the push-forward of a nef Cartier $\mathbb{R}$-divisor on some birational model $X'$ (see Definition \ref{defi:gp}). 
Then we can extend the definition of the minimal log discrepancy to generalized pairs, and 
we can also expect the ACC conjecture and the LSC conjecture in this setting. 

\begin{conj}[ACC conjecture for generalized pair]\label{conj:ACCintro}
Let $d \in \mathbb{Z} _{>0}$ and let $I \subset [0, + \infty)$ be a subset that satisfies the DCC. 
Then the set $A_{\textup{gen}}(d,I)$ defined by
\[
A_{\textup{gen}}(d,I) := 
\left\{ \operatorname{mld}_x (X, B + M) \ \middle | 
\begin{array}{l}
\text{$(X, B +M)/Z$ is a generalized pair with}\\
\text{$\dim X = d$, $B \in I$, $M \in I$ and $x \in |X|_{0}$.}
\end{array}
\right \}
\]
satisfies the ACC. See Definition \ref{defi:defs} (1) for the meaning of $M \in I$
\end{conj}

\begin{conj}[LSC conjecture for generalized pair]\label{conj:LSCintro}
Let $(X, B+M)/Z$ be a generalized pair. 
Then the function 
\[
m: |X|_{0} \to \mathbb{R} \cup \{ - \infty \}; \quad x \mapsto \operatorname{mld}_{x}(X, B+M)
\]
is lower semi-continuous.
\end{conj}

The main purpose of this article is to confirm that known results on these two conjectures for usual pairs are still valid for generalized pairs. 

First, we follow the argument in \cite{Sho} to prove Conjecture \ref{conj:ACCintro} in dimension two (see \cite{Mor} for a similar related result). 
Furthermore, we generalize the argument in \cite{Nak16b} to prove Conjecture \ref{conj:ACCintro} for varieties with fixed Gorenstein index, 
and prove partial results on generalized canonical pairs of dimension three. 

\begin{thm}[$={}$Theorem \ref{thm:ACC_2dim}]\label{thm:introACC_2dim}
For any DCC subset $I \subset [0,+ \infty)$, the set $A_{\textup{gen}}(2,I)$ satisfies the ACC. 
\end{thm}

\begin{thm}[$={}$Theorem \ref{thm:ACCfixed}]\label{thm:introACCfixed}
Let $d,r \in \mathbb{Z} _{>0}$ and let $I \subset [0, + \infty)$ be a finite set. 
Let $P(d,r,I)$ be the set of all generalized log canonical pairs $(X, B + M)/Z$ with the following conditions: 
\begin{itemize}
\item $\dim X = d$, 
\item $r K_X$ is Cartier, 
\item $B = \sum _{i} b_i B_i$ for some $b_i \in I$ and effective Cartier divisors $B_i$. 
\item $M = \sum _i m_i M_i$ for some $m_i \in I$ and Cartier divisors $M_i$ such that $M_i = f_* M_i '$ for some 
projective birational morphism $f:X' \to X$ and nef$/Z$ Cartier divisors $M'_i$. 
\end{itemize}
Then the following set
\[
A' _{\textup{gen}}(d,r,I) := \left\{ \operatorname{mld}_x (X, B + M) \ \middle | 
\begin{array}{l}
\text{$(X,B+M)/Z \in P(d,r,I)$, $x \in |X|_{0}$.}
\end{array}
\hspace{-1.5mm}
\right\}
\] 
is a discrete subset of $[0,+\infty)$. 
\end{thm}

\begin{thm}[$={}$Corollary \ref{cor:ACC3dim}]\label{thm:introACC3dim}
Let $I \subset [0, + \infty)$ be a finite subset. 
Then the following set 
\[
A_{\textup{gen.can}}(3,I) := \left\{ \operatorname{mld}_x (X, B + M) \ \middle | 
\begin{array}{l}
\text{$(X,B+M)/Z$ is a generalized} \\
\text{canonical pair of $\dim X = 3$}\\
\text{with $B \in I$, $M \in I$, $x \in |X|_{0}$.}
\end{array}
\right \}
\]
satisfies the ACC. Furthermore, $1$ is the only accumulation point of this set. 
\end{thm}

For the LSC conjecture, we prove that Conjecture \ref{conj:LSCintro} can be reduced to Conjecture \ref{conj:LSCu}. 

\begin{thm}[$={}$Theorem \ref{thm:LSC_reduction}]\label{conj:LSCreductionintro}
Conjecture \ref{conj:LSCu} implies Conjecture \ref{conj:LSCintro}. 
\end{thm}

\noindent
By this theorem, we can confirm that Conjecture \ref{conj:LSCintro} is still true in dimension three or when $X$ is smooth for example. 
In the proof of Theorem \ref{conj:LSCreductionintro}, two lemmas play important roles. 
The first one is the constructibility of the minimal log discrepancy proved by Ambro (Lemma \ref{lem:const}). 
The second one is the limit lemma (Lemma \ref{lem:mld_limit}), which allows us to approximate the minimal log discrepancy of a generalized pair 
by those of usual pairs. 

As we mentioned, Shokurov proved that Conjecture \ref{conj:ACCu} and Conjecture \ref{conj:LSCu} 
imply the conjecture of termination of flips \cite{Sho04}. 
His argument can be easily extended to the generalized setting. 
\begin{thm}[$={}$Theorem \ref{thm:termination}]\label{thm:terminationintro}
Conjecture \ref{conj:LSCintro} and Conjecture \ref{conj:ACCintro} imply the termination 
of flips for generalized log canonical projective pairs. 
\end{thm}

In the remaining part of this paper, we discuss the theory of complements for generalized pairs.
Theorem \ref{thm:perturbmld} states that it is possible to perturb irrational coefficients preserving 
$\epsilon$-log canonicity in some sense. 
This gives a slight generalization of the result by G.\ Chen and J.\ Han \cite[Lemma 3.4]{CH21}. 
As an application, we prove Corollaries \ref{cor1} and \ref{cor2}, 
which remove the rationality condition on the coefficient set from the results by Filipazzi and Moraga \cite{FM}. 
We note that Corollaries \ref{cor1} and \ref{cor2} overlap with \cite{C20} and \cite{CX20} 
(see the beginning of Subsection \ref{subsection:compl} and Remark \ref{rmk:overlap} for more details).

The paper is organized as follows. 
In Section \ref{section:pre}, we review some definitions and facts on generalized pairs. 
In Section \ref{section:LSC}, we prove some basic properties on the minimal log discrepancies 
for generalized pair and prove Theorem \ref{conj:LSCreductionintro} ($={}$Theorem \ref{thm:LSC_reduction}). 
In Section \ref{section:termination}, we prove Theorem \ref{thm:terminationintro} ($={}$Theorem \ref{thm:termination}). 
In Section \ref{section:ACC2dim}, we prove the ACC conjecture in dimension two (Theorem \ref{thm:introACC_2dim}${}={}$Theorem \ref{thm:ACC_2dim}). 
In Section \ref{section:ACCfixed}, we study the minimal log discrepancies on varieties of fixed Gorenstein index, 
and prove Theorem \ref{thm:introACCfixed} ($={}$Theorem \ref{thm:ACCfixed}) 
and Theorem \ref{thm:introACC3dim} ($={}$Corollary \ref{cor:ACC3dim}). 
In Subsections \ref{subsectionperturb} and \ref{subsection:compl}, we discuss the theory of complements for generalized pairs. 
Theorems \ref{thm:perturbmld} and \ref{thm:bdd-e-comple}, Corollaries \ref{cor1} and \ref{cor2} are the main results of these subsections. 

\begin{ack}The second and third authors started this work during their stay at Johns Hopkins University in 2019.
We would like to thank Professors V.V.\ Shokurov and J.\ Han for their hospitality. 
Moreover, we thank Professor J.\ Han for his comments on the earlier version of this preprint. 
We also thank Professor K.\ Shibata for his valuable comments on the draft. 
The first author was supported by Japan-Taiwan exchange association as a student of the University of Tokyo, by the University of Tokyo as a project researcher, and by
the Ministry of Science and Higher Education of the Russian Federation (agreement no. 075-15-2019-1614) when he was working at the Steklov International Mathematical Center.
The second author was partially supported by JSPS KAKENHI No.\ 15H03611, 16H02141, 17H02831, 18H01108, and JPJSBP120219935. 
He brushed up this paper when he stayed at Johns Hopkins University, National Taiwan University, and the University of Utah. 
He is grateful for their hospitality. 
The third author is partially supported by JSPS KAKENHI No.\ 18K13384, 16H02141, 17H02831, 18H01108, and JPJSBP120219935. 
\end{ack}

\section{Preliminary}\label{section:pre}

\subsection{Notation and Convention}
Throughout this paper, we work over the field $\mathbb{C}$ of complex numbers. 
We follow the notation and terminology in \cite{KM98} and \cite{Kol13}. 

\subsection{Generalized pairs}
We recall some notation and properties on generalized pairs. 
For more detail, we refer the reader to \cite{Bir19a} and \cite{Bir16b}. 
We also refer the reader to \cite{KM98} and \cite{Kol13} for the notation in the minimal model program. 

\begin{defi}\label{defi:gp}
A \textit{generalized pair} $(X,B+M)/Z$ consists of
\begin{itemize}
\item a normal variety $X$ and a quasi-projective variety $Z$ with a projective morphism $X \to Z$, 
\item an effective $\mathbb{R}$-divisor $B$ on $X$, and 
\item a b-$\mathbb{R}$-Cartier b-divisor over $X$ represented by some projective birational morphism $\varphi: X' \to X$ and 
a nef$/Z$ $\mathbb{R}$-Cartier divisor $M'$ on $X'$
\end{itemize}
such that $M = \varphi_* M'$ and $K_{X}+B+M$ is $\mathbb{R}$-Cartier. 

We say that $B$ is the \textit{boundary part} and $M$ is the \textit{nef part} of the generalized pair $(X,B+M)$. 
\end{defi}

\begin{defi}\label{defi:defs}
Let $(X,B+M)/Z$ be a generalized pair. 
\begin{enumerate}
\item Let $I \subset [0, + \infty)$ be a subset. 
Then we write $B \in I$ when all nonzero coefficients of $B$ belong to $I$. 
By abuse of notation, we also write $M \in I$ when $M' = \sum _{1 \le i \le \ell} r_i M'_{i}$ holds for some 
positive real numbers $r_1, \ldots, r_{\ell} \in I$ and nef$/Z$ Cartier divisors $M' _1, \ldots, M' _{\ell}$ possibly replacing $X'$. 
This definition is natural when considering an NQC generalized pair (see \cite{HL22}*{Section 2.6} for more detail). 

\item For a prime divisor $E$ over $X$, we define the \textit{generalized log discrepancy} $a_E (X, B + M)$ as follows. 
Possibly replacing $X'$ with a higher model, we may assume that $\varphi$ is a log resolution of $(X, B)$ and that $E$ is a divisor on $X'$. 
We define an $\mathbb{R}$-divisor $B'$ on $X'$ by 
\[
K_{X'} + B' + M' = \varphi ^* (K_X + B +M). 
\]
Then we define $a_E (X, B + M) := 1 - \operatorname{coeff}_E B'$. 
The image $\varphi (E)$ is called the \textit{center} of $E$ on $X$ and we denote it by $c_X(E)$. 

\item 
Let $W \subset X$ be a closed subset. Then we define the \textit{generalized minimal log discrepancy} $\operatorname{mld}_W (X, B +M)$ along $W$ as 
\[
\operatorname{mld}_W (X, B +M) := \inf _{c_X(E) \subset W} a_E(X, B+M)
\]
if $\dim X \ge 2$, where the infimum is taken over all prime divisors $E$ over $X$ with center $c_X(E) \subset W$. 
When $\dim X = 1$, we define $\operatorname{mld}_W (X, B +M) := \inf _{c_X(E) \subset W} a_E(X, B+M)$ if the infimum is non-negative and 
$\operatorname{mld}_W (X, B +M) := - \infty$ otherwise. 
For simplicity of notation, we write $\operatorname{mld} (X, B +M)$ instead of $\operatorname{mld}_X (X, B +M)$.

\item \label{item:mld}
Let $\eta$ be a scheme-theoretic point of $X$ of $\operatorname{codim} \operatorname{\eta} \ge 1$. Then we define the \textit{generalized minimal log discrepancy} 
$\operatorname{mld}_{\eta} (X, B +M)$ at $\eta$ as 
\[
\operatorname{mld}_{\eta} (X, B +M) := \inf _{c_X(E) = \overline{\{ \eta \}}} a_E(X, B+M)
\]
if $\operatorname{codim} \operatorname{\eta} \ge 2$, 
where the infimum is taken over all prime divisors $E$ over $X$ with center $c_X(E) = \overline{\{ \eta \}}$. 
When $\operatorname{codim} \operatorname{\eta} = 1$, we define 
$\operatorname{mld}_{\eta} (X, B +M) := \inf _{c_X(E) = \overline{\{ \eta \}}} a_E(X, B+M)$ if the infimum is non-negative and 
$\operatorname{mld}_{\eta} (X, B +M) := - \infty$ otherwise. 
We denote $\operatorname{mld}_{\eta} (X) = \operatorname{mld}_{\eta} (X, B +M)$ if $B=M'=0$. 

\item When $\operatorname{mld}_{\eta} (X, B +M) \ge 0$ in (\ref{item:mld}), 
we say that a divisor $E$ over $X$ \textit{computes} $\operatorname{mld}_{\eta} (X, B +M)$ 
if $c_X (E) = \overline{\{ \eta \}}$ and $\operatorname{mld}_{\eta} (X, B +M) = a_E (X, B+M)$ hold. 
When $\operatorname{mld}_{\eta} (X, B +M) = - \infty$, we say that $E$ \textit{computes} $\operatorname{mld}_{\eta} (X, B +M)$ 
if $c_X (E) = \overline{\{ \eta \}}$ and $a_E (X, B+M) < 0$ hold.

\item We say that $(X,B+M)$ is \textit{generalized log canonical} (\textit{generalized lc} for short) 
if $\operatorname{mld} (X, B + M) \ge 0$ holds. 
We say that $(X,B+M)$ is \textit{generalized Kawamata log terminal} (\textit{generalized klt} for short) 
if $\operatorname{mld} (X, B + M) > 0$ holds. 
Furthermore, for $\epsilon \in \mathbb{R}_{\ge 0}$ we say that $(X,B+M)$ is \textit{generalized $\epsilon$-lc} 
if $\operatorname{mld} (X, B + M) \ge \epsilon$ holds.

\item A \textit{generalized non-klt center} of $(X, B+M)$ is the center $c_X (E)$ of a divisor $E$ over $X$ with 
$a_E (X, B+M) \le 0$. 

\item We say that $(X,B+M)$ is \textit{generalized dlt} when the following two conditions hold: 
\begin{itemize}
\item $(X,B+M)$ is generalized lc. 
\item If $\eta$ is the generic point of a generalized non-klt center of $(X, B+M)$, 
then  $(X, B)$ is log smooth at $\eta$ and $M'= \varphi ^* M$ holds over a neighborhood of $\eta$. 
\end{itemize}

\item \label{item:dltmodel}
Let $f: Y \to X$ be a projective birational morphism. Possibly replacing $\varphi$, we may assume that $\varphi$ factors through $f$. 
Then we define $B_Y$ by 
\[
K_Y + B_Y + M_Y = f^* (K_X + B + M), 
\]
where $M_Y$ is the push-forward of $M'$ on $Y$. 
We say that $(Y, B_Y + M_Y)$ is a \textit{$\mathbb{Q}$-factorial generalized dlt model} of $(X, B + M)/Z$ if the following three conditions hold. 
\begin{itemize}
\item $Y$ is $\mathbb{Q}$-factorial. 
\item $(Y, B_Y + M_Y)$ is generalized dlt. 
\item $a_E (X, B+M) = 0$ holds for any $f$-exceptional divisor $E$. 
\end{itemize}
By \cite[Proposition 3.10]{HL22}, a $\mathbb{Q}$-factorial generalized dlt model always exists for every generalized lc pair $(X,B+M)/Z$. 

\item \label{item:lct}
Suppose that $(X, B+M)$ is generalized lc. 
Let $D$ be an effective $\mathbb{R}$-divisor on $X$, and let $N'$ be a nef$/Z$ $\mathbb{R}$-Cartier divisor on $X'$. 
We assume that $D + N$ is $\mathbb{R}$-Cartier, where $N:= \varphi _* N'$. 
Then $\bigl( X, (B+tD) + (M+tN) \bigr)/Z$ is a generalized pair for each $t \in \mathbb{R}_{\ge 0}$. 
We define the \textit{generalized lc threshold} of $D+N$ with respect to $(X, B+M)$ as 
\[
\sup \bigl\{ t \in \mathbb{R}_{\ge 0} \ \big| \  \text{$\bigl( X, (B+tD) + (M+tN) \bigr)$ is generalized lc} \bigr \}. 
\]
Here, we regard $B+tD$ as the boundary part and $M+tN$ as the nef part of $\bigl( X, (B+tD) + (M+tN) \bigr)$. 

\item \label{item:ord}
Let $E$ be a divisor over $X$. Suppose that $X$ is $\mathbb{Q}$-Gorenstein at the generic point of $c_X(E)$. 
Then we define the \textit{order} $\operatorname{ord}_E (B + M)$ at $E$ as follows. 
Possibly replacing $\varphi$, we may assume that $E$ is a divisor on $X'$. We define an $\mathbb{R}$-divisor $C$ on $X'$ by
\[
C + M' = \varphi ^* (B + M). 
\]
Then we define $\operatorname{ord}_E (B + M) := \operatorname{coeff}_E C$. 
Since $B \ge 0$ and $M'$ is nef$/Z$, we have $\operatorname{ord}_E (B + M) \ge 0$. 

\item \label{item:mult}
Let $\eta$ be a scheme-theoretic point of $X$ of $\operatorname{codim} \eta \ge 1$. 
Suppose that $X$ is smooth at $\eta$. 
Then we define the \textit{multiplicity} $\operatorname{mult}_{\eta} (B + M)$ as follows. 
If $\operatorname{codim} \eta =1$, then 
$\operatorname{mult}_{\eta} (B + M) := \operatorname{coeff} _E B$, where $E$ is the corresponding prime divisor to $\eta$. 
We assume $\operatorname{codim} \eta \ge 2$ in what follows. 
Let $Y \to X$ be the blow-up along $\overline{\{ \eta \}}$ and $E$ the divisor on $Y$ that dominates $\overline{\{ \eta \}}$. 
Then we define $\operatorname{mult}_{\eta} (B + M) := \operatorname{ord}_E (B + M)$. 
\end{enumerate}
\end{defi}

\begin{rmk}
\begin{enumerate}
\item
For a generalized pair $(X, B+M)/Z$ with $M' = 0$, the invariants defined in Definition \ref{defi:defs} coincide 
with those for a usual pair $(X, B)$. 
\item 
A referee kindly pointed out that the condition ``$M' = \sum _{1 \le i \le \ell} r_i M'_{i}$'' in Definition \ref{defi:defs}(1) can be 
weakened to ``$M' \equiv \sum _{1 \le i \le \ell} r_i M'_{i}$'' without loss of generality because the invariants defined in Definition \ref{defi:defs} 
remains unchanged even if $M'$ is replaced with a numerically equivalent one. 
\end{enumerate}
\end{rmk}

\noindent
We list basic facts on generalized minimal log discrepancies. 

\begin{rmk}\label{rmk:rmks}
Let $(X, B+M)/Z$ be a generalized pair and let $\eta$ be a scheme-theoretic point of $X$ of $\operatorname{codim} \eta \ge 1$. 

\begin{enumerate}
\item Let $\varphi: X' \to X$ be the birational morphism in Definition \ref{defi:gp}. 
We define an $\mathbb{R}$-divisor $B'$ by 
\[
K_{X'} + B' + M' = \varphi ^* (K_X + B + M). 
\]
Then we have
\[
a_E (X, B+M) = a_E (X', B')
\]
for any divisor $E$ over $X$. 
Hence we have 
\[
\operatorname{mld}_{\eta} (X, B +M) = \inf _{c_X(E) = \overline{\{ \eta \}}} a_E(X', B'), 
\]
where the infimum is taken over all prime divisors $E$ over $X$ with center $c_X(E) = \overline{\{ \eta \}}$. 
This observation allows us to prove the remarks below, which are well-known for usual pairs (that is, the case when $M' = 0$).

\item \label{item:resol1} 
Possibly replacing $\varphi$, we may assume that $\varphi$ is a log resolution of $(X, B)$. 
Then $\operatorname{mld}_{\eta} (X, B + M) = - \infty$ holds if and only if 
$a_E(X,B+M)<0$ and $\eta \in c_X(E)$ hold for some prime divisosr $E$ on $X'$. 
Moreover, such $E$ can be found in the divisors contained in the $\varphi$-exceptional locus or the strict transform of $B$. 

\item For any $\eta \in X$, possibly replacing $\varphi$, we may assume that $\varphi$ is a log resolution of 
the triple $(X, B, \overline{\{ \eta \}})$. 
If $\operatorname{mld}_{\eta} (X, B + M) \ge 0$, then some divisor on $X'$ computes $\operatorname{mld}_{\eta} (X, B + M)$. 

\item \label{item:-inftycl}
The set $\{ \eta \in X \mid  \operatorname{mld}_{\eta} (X, B + M) = - \infty\}$ is a closed subset of $X$. 
This follows from (\ref{item:resol1}). 

\item If $\operatorname{mld}_{\eta} (X, B + M) < 0$
, then $\operatorname{mld}_{\eta} (X, B + M) = - \infty$. 

\item If $\operatorname{mld}_{\eta} (X, B + M) \ge 0$, then the infimum in Definition \ref{defi:defs}(\ref{item:mld}) is in fact the minimum. 

\end{enumerate}
We note that the similar statements as above also hold for $\operatorname{mld} (X, B + M)$. 
\end{rmk}

\subsection{Generalized minimal model program}

In this subsection, we collect some facts related to the minimal model program for generalized pairs. 

We will start from the basics of the minimal model program for generalized pairs. 

\begin{defi}[Generalized minimal model program]\label{LMMP}
Let $\pi:X \to S$ be a projective morphism of quasi-projective varieties and 
$(X,B+M)/Z$ a $\mathbb{Q}$-factorial generalized lc pair. 
Suppose that $X \to Z$ in Definition \ref{defi:gp} factors through $\pi : X \to S$. 
Assume that there exists an extremal ray $R_1\subset \overline{NE}(X/S)$ such that $(K_X+B+M) \cdot R_1 <0$ 
and we have an extremal contraction with respect to $R_1$. 
Note that when $(X,B+M)$ is generalized klt, 
we always have such an extremal contraction. 
If the contraction is a divisorial contraction or a flipping contraction,
let
\[
(X,B+M) \dashrightarrow (X_1,B_1+M_1)
\]
be the divisorial contraction or its flip, respectively. 
Here $B_1+M_1$ is the strict transform of $B+M$ on $X_1$.
Next, if we find an extremal ray $R_2 \subset \overline{NE}(X_1/S)$ with the same conditions as above, 
we repeat this process. We call this process a \textit{generalized minimal model program} over $S$; 
\[
(X,B+M) = (X_0,B_0+M_0) \dashrightarrow (X_1, B_1+M_1) \dashrightarrow \cdots \dashrightarrow (X_i, B_i+M_i) \dashrightarrow \cdots. 
\]
Furthermore, we call it a {\it sequence of flips} when  every map 
$(X_{i-1}, B_{i-1}+M_{i-1}) \dashrightarrow (X_i, B_i+M_i)$ is a flip and 
we say that it {\it terminates} if the sequence of flips is not an infinite sequence.
\end{defi}

Next, we consider special minimal model programs with some extra data:

\begin{defi}[Generalized minimal model program with scaling]\label{LMMPS}
Let $\pi:X \to S$ and $(X, B + M)/Z$ be as in Definition \ref{LMMP}. 
Suppose that $(X, B + M)$ is generalized klt. 
Let $H$ be an effective $\mathbb{R}$-divisor such that 
$K_X + B + H +M$ is $\pi$-nef and $(X, B + H +M)$ is still generalized lc. 
We put
\[
\lambda_1 := \inf\{\alpha \in \mathbb{R}_{\geq0} \mid K_X + B + \alpha H +M\ \text{is\ $\pi$-nef} \}.
\]
If $K_X + B +M$ is not $\pi$-nef, then $\lambda_1>0$
and there exists an extremal ray $R_1 \subset \overline{NE}(X/S)$ such that $(K_X + B+M) \cdot R_1 <0$ 
and $(K_X + B + \lambda_1 H +M) \cdot R_1=0$ (cf.\ \cite[Section 4]{BZ16}). 
We consider the extremal contraction with respect to this $R_1$. 
If it is a divisorial contraction or a flipping contraction, let 
\[
(X, B+M) \dashrightarrow (X_1, B_1+M_1)
\]
 be the divisorial contraction or its flip. 
Note that $K_{X_1}+ B_1 + \lambda_1H_1 +M_1$ is $\pi$-nef, where $H_1$ is the strict transform of $H$ on $X_1$. 
We put 
\[
\lambda_2 := \inf\{\alpha \in \mathbb{R}_{\geq0} \mid K_{X_1}+ B_1+ \alpha H_1 +M_1 \ \text{is\ $\pi$-nef} \}. 
\]
If $K_{X_1} + B_1+M_1$ is not $\pi$-nef, 
we can find an extremal ray $R_2$ by the same way as above. 
We repeat this process and get a sequence 
\[
(X,B+M) = (X_0, B_0+M_0) \dashrightarrow (X_1, B_1+M_1) \dashrightarrow \cdots \dashrightarrow (X_i, B_i+M_i) \dashrightarrow \cdots, 
\]
with $\lambda_1 \geq \lambda_2 \geq \lambda_3 \ge \cdots$, where 
\[
\lambda_i := \inf\{\alpha \in \mathbb{R}_{\geq0} \mid K_{X_{i-1}} + B_{i-1} + \alpha H_{i-1} +M_{i-1} \ \text{is\ $\pi$-nef} \}
\] 
and $H_{i-1}$ is the strict transform of $H$ on $X_{i-1}$.
We call this process a \textit{generalized minimal model program for $(X, B+M)$ with scaling of} $H$ over $S$. 
\end{defi}

\noindent
In Definition \ref{LMMPS}, we have assumed that $(X, B+M)$ is generalized klt. 
We note that the existence of $R_1$ is a problem when the pair is generalized lc. 

\begin{thm}[{\cite[Lemma 4.4(1)]{BZ16}}]\label{thm:MFS}
Let $(X, B+M)/Z$ be a $\mathbb{Q}$-factorial generalized lc pair with $X$ klt.
Suppose that $K_X + B + M$ is not pseudo-effective over $Z$. 
Then, for a general ample $\mathbb{R}$-divisor $H$, 
a generalized minimal model program with scaling of $H$ for $(X, B+M)$ exists and terminates with a Mori fiber space.
\end{thm}

\begin{thm}\label{thm:extraction}
Let $(X, B+M)/Z$ be a generalized klt pair and let $E$ be an exceptional divisor over $X$ with $a_E(X, B+M ) \leq 1$. 
Then there exists a projective birational morphism $f:W \to X$ such that $E$ is the only $f$-exceptional divisor. 
\end{thm}

\begin{proof}
Take a log resolution $\psi: Y \to X$ of $(X, B)$ such that $E$ is on $Y$. 
We may assume that $\varphi: X' \to X$ factors through $\psi:Y \to X$ and $M_Y$ is nef$/Z$, 
where $M_Y$ is the push-forward of $M'$ on $Y$. 
Write $\psi^*(K_X+ B+M)+F=K_Y+B_Y+M_Y$, where $F$ and $B_Y$ are effective and with no common components. 
Let $G$ be the sum of $\psi$-exceptional divisors except for $E$. 
Take a sufficiently small $\epsilon >0$ such that $(Y, B_Y+\epsilon G+M_Y)$ is generalized klt. 
Then we run a generalized MMP for $(Y, B_Y+\epsilon G+M_Y)/Z$ over $X$ with scaling of general ample divisor $H$. 
This MMP exists and terminates with a minimal model $W$ by \cite[Lemma 4.4(2)]{BZ16}.
By applying the negativity lemma to the push-forward of $\epsilon G + F$ on $W$, the minimal model $f: W \to X$ turns out to have only exceptional divisor $E$. 
\end{proof}

\subsection{Divisorial adjunction for generalized pairs}
We recall the construction of divisorial adjunction for generalized pairs. 
In the minimal model theory, the divisorial adjunction is a significant tool in arguments by induction on dimension. 

\begin{defi}[{cf.\ \cite[3.1]{Bir19a}}]\label{def:gadj}
Let $(X,B+M)/Z$ be a generalized pair and let $\varphi:X'\rightarrow X$ be the birational morphism in Definition \ref{defi:gp}. 
Replacing $X'$, we may assume that $\varphi$ is a log resolution of $(X,B)$.
Let $S$ be the normalization of a component of $B$ with coefficient $1$, 
and let $S'$ be the strict transform of $S$ on $X'$.
We define an $\mathbb{R}$-divisor $B'$ on $X'$ by 
\[
K_{X'}+B'+M' = \varphi^* (K_X+B+M). 
\]
Then we set $B_{S'}=(B'-S')|_{S'}$ and pick $M_{S'} \sim _{\mathbb{R}}M'|_{S'}$. 
By the adjunction, we have 
\[
K_{S'}+B_{S'}+M_{S'} \sim _{\mathbb{R}}(K_{X'}+B'+M')|_{S'}.
\]
We denote by $\varphi _S : S' \to S$ the induced morphism. 
We set $B_S = \varphi _{S*} B_{S'}$ and $M_S = \varphi _{S*}M_{S'}$.
Then, $(S, B_S + M_S)/Z$ is a generalized pair satisfying
\begin{itemize}
\item $K_S+B_S+M_S \sim _{\mathbb{R}}(K_{X}+B+M)|_{S}$ and 
\item $K_{S'}+B_{S'}+M_{S'}=\varphi_S ^*(K_S+B_S+M_S)$.
\end{itemize}
We call $(S, B_S + M_S)$ or $K_S+B_S+M_S$ the \textit{divisorial adjunction} for the generalized pair $(X,B+M)$ and $S$. 
We remark that if $(X,B+M)$ is generalized lc, 
then $(S, B_S + M_S)$ is also generalized lc by the construction (\cite[Remark 4.8]{BZ16}).
\end{defi}

\begin{rmk}\label{gadjm}
In Definition \ref{def:gadj}, the  divisor  $B_S$ is uniquely determined by $B$ and $M'$ as a divisor, 
while $M_{S'}$ is defined only up to $\mathbb{R}$-linear equivalence.
We can observe that if $M'=\sum m_iM'_i$ with $m_i>0$ and $M'_i$ nef$/Z$ Cartier,
then we can pick some Cartier divisors $M_{i,S'}\sim M'_i|_{S'}$ and take $M_{S'}=\sum m_iM_{i,S'}$. 
Note here that some of $M_{i,S'}$'s might be numerically trivial. 
\end{rmk}

The following lemma is from \cite[Lemma 3.3]{Bir19a}. 
We shall write down the coefficients of $M'$ more explicitly. 
\begin{lem}\label{lem:adjcoef}
Under the settings and the notation in Definition \ref{def:gadj}, assume that
\begin{itemize}
\item $(X, B + M)/Z$ is generalized lc,
\item $B=\sum b_iB_i$, where $B_i$'s are prime divisors on $X$, and
\item $M'=\sum m_iM'_i$ as in Remark \ref{gadjm}.
\end{itemize}
Then the coefficients of $B_S$ are of the form
\[
1-\frac{1}{\ell}+\frac{\sum \alpha_ib_i+\sum \beta_jm_j}{\ell},
\] where $\ell \in \mathbb{Z}_{> 0}$, $\alpha_i, \beta_j \in \mathbb{Z}_{\geq 0}$,
and $M_{S'}$ can be chosen to be of the form
\[
\sum m_iM_{i,S'},
\]
where $M_{i,S'}$ are nef$/Z$ Cartier divisors on $S'$.
\end{lem}
\begin{proof}
We take over the notations in Definition \ref{def:gadj}. 
We have already seen the assertion for $M_{S'}$ in Remark \ref{gadjm}. 
In what follows, we shall show the assertion for $B_S$. 
By taking general hypersurface sections (see \cite[Remark 4.8]{BZ16}), we may assume that $\dim X=2$. 
Note that, in particular, $X$ is $\mathbb{Q}$-factorial (cf.\ \cite[Proposition 4.11(1)]{KM98}). 
For any fixed closed point $v \in S$, we shall show that $\mu_v B_S$ is of the form in the statement. 
Since $(S,B_S+M_S)$ is generalized lc, we have $\mu_v B_S\leq 1$.
We may assume that $\mu_v B_S< 1$.
We define $\widetilde{B}_S$ by the divisorial adjunction for usual pairs: 
\[
K_S + \widetilde{B}_S = (K_X+B)|_S.
\]
Then by the relative nefness of $M'$ and the negativity lemma, we have $\widetilde{B}_S \leq B_S$.
Therefore, by the inversion of adjunction for surfaces (cf.\ \cite[Corollary 3.12]{Sho93}), 
$(X,B)$ is plt at the image $v' \in X$ of $v$.

By the adjunction for usual pairs (cf.\ \cite[Proposition 3.9 and Corollary 3.10]{Sho93}), 
there exists $\ell \in \mathbb{Z}_{> 0}$ such that $\ell D$ is Cartier at $v'$ for any Weil divisor $D$ on $X$, and we have 
\[
\mu_v \widetilde{B}_S=1-\frac{1}{\ell}+\sum\frac{\alpha_ib_i}{\ell}
\]
for some $\alpha_i\in\mathbb{Z}_{\geq 0}$. 
We write $\varphi^* M_i = M'_i+E_i$. Then $E_i\geq 0$ holds by the relative nefness of $M'_i$ and the negativity lemma.
Since $M' _i$ is Cartier, $M_i$ is integral and hence $\ell M_i$ is Cartier at $v'$. 
Therefore $\ell E_i$ is also Cartier over $v'$.
By construction, it follows that $B_S = \widetilde{B}_S + \sum m_iE_{iS}$,
where $E_{iS} := \varphi_{S *}(E_i|_{S'})$.
Since $\ell E_{iS} = \varphi_{S*}(\ell E_i|_{S'})$ is integral, we have
\[
\mu_v B_S=1-\frac{1}{\ell}+\sum\frac{\alpha_ib_i}{\ell}+\sum \frac{\beta_jm_j}{\ell}
\]
for some $\beta _j \in \mathbb{Z}_{\geq 0}$, which completes the proof. 
\end{proof}

\section{Lower semi-continuity conjecture for generalized mld's}\label{section:LSC}

\subsection{Basic properties on generalized minimal log discrepancies}
In this subsection, following Ambro's paper \cite{Amb99}, 
we prove the finiteness and the constructibility of the mld function (Lemma \ref{lem:const}), 
which is weaker than the lower semi-continuity conjecture (Conjecture \ref{conj:LSC}). 
We also prove that the minimal log discrepancy of a generalized pair can be approximated 
by those of usual pairs (Lemma \ref{lem:mld_limit}). 
Both Lemmas \ref{lem:const} and \ref{lem:mld_limit} will be important in the next subsection, 
where we reduce the lower semi-continuity conjecture for generalized pairs to that for usual pairs. 

The following lemma is a generalization of \cite[Proposition 2.1]{Amb99} to the generalized setting. 
The same proof in \cite[Proposition 2.1]{Amb99} also works in this setting. 

\begin{lem}\label{lem:ambro}
Let $(X, B+M)/Z$ be a generalized pair, and 
let $\eta$ be a scheme-theoretic point of $X$ with $\operatorname{codim} \eta \ge 1$. 
Then there exists an open set $U \subset |X|_{0}$ with $\overline{\{ \eta \}} \cap U \not = \emptyset$ such that 
\[
\operatorname{mld}_x(X, B+M) = \operatorname{mld}_{\eta} (X, B+M) + \dim \eta
\]
holds for any $x \in \overline{\{ \eta \}} \cap U$.
\end{lem}

\begin{proof}
We may assume that $(X, B+M)$ is generalized lc at $\eta$ (cf.\ Remark \ref{rmk:rmks}(\ref{item:-inftycl})). 
Let $\varphi: X' \to X$ be the projective birational morphism in Definition \ref{defi:gp}. 
Possibly replacing $\varphi$, we may assume that $\varphi$ is a log resolution of the triple $(X, B, \overline{\{ \eta \}})$. 

We define an $\mathbb{R}$-divisor $B'$ on $X'$ by 
\[
K_{X'} + B' + M' = \varphi ^* (K_{X} + B + M). 
\]
Let $\{ D_i \}_{i \in I}$ be the set of the $\varphi$-exceptional divisors, the components of 
$\varphi ^{-1}_* B$, and the computing divisors of $\operatorname{mld}_{\eta} (X, B+M)$ on $X'$. 
We may write 
\[
B' = \sum _{i \in I} e_i D_i
\]
with $e_i \le 1$. 
Let $I_{\textup{dom}} \subset I$ be the set of the indices $i$ satisfying $\varphi (D_i) = \overline{\{ \eta \}}$. 
Then we have 
\[
\operatorname{mld}_{\eta} (X, B+M) = \min _{i \in I_{\textup{dom}}} (1 - e_i). 
\]
Since $(X', B')$ is log smooth, we have 
\[
\operatorname{mld}_{\xi}(X', B'+M') = \operatorname{mld}_{\xi}(X', B') = \operatorname{codim} \xi - \sum _{\xi \in D_i} e_i
\]
for any scheme-theoretic point $\xi$ of $X'$ (cf.\ \cite[Corollary 2.32]{KM98}). 

By generic smoothness, we can take an open set $U \subset |X|_{0}$ 
such that $\overline{\{ \eta \}} \cap U \not = \emptyset$ with the following two conditions: 
\begin{itemize}
\item[(a)] 
For any $x \in \overline{\{ \eta \}} \cap U$ and any stratum $C$ of 
$\operatorname{Supp} \bigl( \sum _{i \in I} D_i \bigr)$ that satisfies $\varphi (C) = \overline{\{ \eta \}}$, 
it follows that $\dim F =  \dim C  - \dim \eta$ for any irreducible component $F$ of $C \cap \varphi^{-1}(x)$. 

\item[(b)]
For any stratum $C$ of $\operatorname{Supp} \bigl( \sum _{i \in I} D_i \bigr)$, 
if $\varphi (C) \subset \overline{\{ \eta \}}$ and $\varphi (C) \cap U \not = \emptyset$, then $\varphi (C) = \overline{\{ \eta \}}$ holds. 
\end{itemize}
In what follows, we shall prove that this $U$ satisfies the assertion. 

Let $x \in \overline{\{ \eta \}} \cap U$ and let $\xi \in X'$ be a scheme-theoretic point with $\varphi (\xi) = x$. 
Let $C$ be the minimum stratum of $\operatorname{Supp} \bigl( \sum _{i \in I} D_i \bigr)$ that contains $\xi$. 
Since $\varphi$ is a log resolution of $\overline{\{ \eta \}}$, 
$\xi \in D_j$ holds for some $j \in I_{\rm dom}$. 
Therefore, it follows that $\varphi (C) \subset \overline{\{ \eta \}}$, and hence $\varphi (C) = \overline{\{ \eta \}}$ by (b). 
Therefore, by (a), we have
\begin{itemize}
\item $\operatorname{codim} \xi - \operatorname{codim} C \ge \dim \eta$. 
\end{itemize}
Furthermore, since $j \in I_{\rm dom}$, we have 
\begin{itemize}
\item 
$1 - e_j \ge \min _{i \in I_{\textup{dom}}} (1 - e_i) = \operatorname{mld}_{\eta} (X, B +M)$. 
\end{itemize}
Therefore, we have 
\begin{align*}
\operatorname{mld}_{\xi}(X', B') 
& = \operatorname{codim} \xi - \sum _{\xi \in D_i} e_i \\
& = (1 - e_j) + (\operatorname{codim} \xi - 1) - \sum _{\xi \in D_i, \ i \not = j} e_i \\
& \ge \operatorname{mld}_{\eta} (X, B+M) +  (\operatorname{codim} \xi - 1) - (\operatorname{codim} C - 1) \\
& \ge \operatorname{mld}_{\eta} (X, B+M) + \dim \eta. 
\end{align*}
Hence we have
\[
\operatorname{mld}_x(X, B+M) \ge \operatorname{mld}_{\eta} (X, B+M) + \dim \eta. 
\]

We shall prove the opposite inequality. 
Let $D_i$ be a computing divisor on $X'$ of $\operatorname{mld}_{\eta} (X, B+M)$ (cf.\ Remark \ref{rmk:rmks}(3)). 
Let $\xi '$ be the scheme-theoretic point of an irreducible component of $D_i \cap \varphi ^{-1}(x)$. 
Then by (a) and (b), the $D_i$ is the unique divisor that contains $\xi'$. 
Hence we have 
\[
\operatorname{mld}_{\xi'}(X', B') = \operatorname{codim} \xi' - e_i = \operatorname{mld}_{\eta} (X, B +M) + \dim \eta, 
\]
which proves 
\[
\operatorname{mld}_x(X,B +M) \le \operatorname{mld}_{\xi'}(X',B') = \operatorname{mld}_{\eta} (X, B+M) + \dim \eta. 
\]
It completes the proof. 
\end{proof}

As a corollary of Lemma \ref{lem:ambro}, we can prove the finiteness and the constructibility of the mld function. 

\begin{lem}\label{lem:const}
Let $(X, B + M)/Z$ be a generalized pair. 
We denote by $m$ the function defined by the minimal log discrepancy: 
\[
m: |X|_{0} \to \mathbb{R} \cup \{ - \infty \}; \quad x \mapsto \operatorname{mld}_x (X,B + M), 
\]
where $|X|_{0}$ is the set of all closed points of $X$ with the Zariski topology. 
Then the following hold. 
\begin{itemize}
\item[(1)] The function $m$ takes finitely many values. 

\item[(2)] The function $m$ is constructible, that is, any fiber of $m$ is a constructible subset of $|X|_{0}$. 
\end{itemize}
\end{lem} 

\begin{proof}
Let $\varphi: X' \to X$ be the birational morphism in Definition \ref{defi:gp}. 
We define an $\mathbb{R}$-divisor $B '$ on $X'$ by 
\[
K_{X'} + B' + M' = \varphi ^* (K_X + B + M). 
\]
Then for $x \in X_{\textup{sm}} \setminus \varphi (\operatorname{Supp} B')$, we have
\[
\operatorname{mld}_x (X, B + M) = \operatorname{mld}_x (X) = \dim X. 
\]
Hence the function $m$ is constant on a dense open subset of $|X|_{0}$. 
Then the assertion follows from Lemma \ref{lem:ambro} by the Noetherian induction. 
\end{proof}

The following lemma will be used in Lemma \ref{lem:mult}. 

\begin{lem}\label{lem:ord_limit}
Let $(X, B+M)/Z$ be a $\mathbb{Q}$-Gorenstein generalized pair and let $E$ be a divisor over $X$. 
Then there exists a sequence of effective divisors $D_m$ on $X$ such that  $K_X+B+D_m$ is $\mathbb{R}$-Cartier and the following two conditions hold. 
\begin{itemize}
\item[(1)] $\operatorname{ord}_{F} (B + M) \le \operatorname{ord}_{F} (B + D_m)$ holds for any divisor $F$ over $X$. 
\item[(2)] $\operatorname{ord}_{E} (B + M) = \lim _{m \to \infty} \operatorname{ord}_{E}(B + D_m)$ holds. 
\end{itemize}
Here $\operatorname{ord}_{{\rm -}} (B + M)$ on the left-hand side is 
regarded as the order for a generalized pair defined in Definition \ref{defi:defs}\,(\ref{item:ord}). 
The order $\operatorname{ord}_{{\rm -}} (B + D_m)$ on the right-hand side is regarded as the order in the usual sense (for a usual pair). 
\end{lem}

\begin{proof}
Let $\varphi:X' \to X$ be the birational morphism as in Definition \ref{defi:gp}. 
We may assume that $E$ is a divisor on $X'$. 

Let $A$ be an ample$/Z$ divisor on $X$. We may write
\[
\varphi^*A \sim_{\mathbb{R}} A'+C 
\]
for some ample$/Z$ $\mathbb{R}$-divisor $A'$ and some effective $\varphi$-exceptional divisor $C$ on $X'$. 
Let $A'_m \in \big| M'+\frac{1}{m}A' \big|_{\mathbb{R}}$ be a general effective divisor. 
We define $D_m := \varphi_* A'_m$. 
Note that $A'_m = \varphi ^{-1}_* D_m$ holds by the generality of $A'_m$. 
Then we have 
\begin{align*}
\varphi ^* (B + M) &= \varphi ^{-1}_* B + M' + \sum _G \operatorname{ord}_G (B + M) G, \\
\varphi ^* (B + D_m) &= \varphi ^{-1}_* B + A'_m + \sum _G \operatorname{ord}_G (B + D_m) G,
\end{align*}
where the sums are taken over all $\varphi$-exceptional divisors $G$. 
Since $M' \sim _{\mathbb{R}, \varphi} A'_m + \frac{1}{m}C$, it follows that 
\[
\operatorname{ord}_G (B + M) =  \operatorname{ord}_G (B + D_m) - \frac{1}{m} \operatorname{coeff}_G C.  
\]
Therefore we have 
\[
\operatorname{ord}_{G} (B + M) = \lim _{m \to \infty} \operatorname{ord}_{G}(B + D_m)
\]
for any $\varphi$-exceptional divisor $G$, in particular for $G = E$. We have proved (2). 

Let $F$ be a divisor over $X$. Let $f: Y \to X$ be a birational morphism such that $F$ is a divisor on $Y$. 
We may assume that $f$ factors through $\varphi$. 
We define divisors $G_1$ and $G_2$ on $X'$ by
\[
G_1 = \sum _G \operatorname{ord}_G (B + M) G, \quad G_2 = \sum _G \operatorname{ord}_G (B + D_m) G, 
\]
where the sums are taken over all $\varphi$-exceptional divisors $G$.
Then it follows that 
\[
\operatorname{ord}_F (B + M) = \operatorname{ord}_F (\varphi ^{-1}_* B + G_1), \quad 
\operatorname{ord}_{F} (B + D_m) = \operatorname{ord}_F (\varphi ^{-1}_* B + A'_m + G_2). 
\]
Since $G_2 - G_1$ and $A' _m$ are effective, we have 
$\operatorname{ord}_{F} (B + M) \le \operatorname{ord}_{F} (B + D_m)$, which proves (1). 
\end{proof}

\begin{lem}\label{lem:mld_limit}
Let $(X, B+M)/Z$ be a generalized pair and let $\eta \in X$ be a scheme-theoretic point 
of $\operatorname{codim} \eta \ge 1$. Suppose that $\operatorname{mld}_{\eta } (X, B + M) > 0$. 
Then there exists a sequence of effective divisors $D_m$ on $X$ such that the following two conditions hold. 
\begin{itemize}
\item[(1)] $\operatorname{mld}_{\xi} (X, B+M) \geq \operatorname{mld}_{\xi}(X, B+D_m)$ holds for any scheme-theoretic point $\xi \in X$ 
of $\operatorname{codim} \xi \ge 1$. 
\item[(2)] $\operatorname{mld}_{\eta } (X, B + M) = \lim _{m \to \infty} \operatorname{mld}_{\eta }(X,B+D_m)$ holds. 
\end{itemize}
Here, we regard $\operatorname{mld}_{{\rm -}}(X, B+D_m)$ on the right-hand side as the minimal log discrepancy for a usual pair, not as a generalized pair. 
\end{lem}

\begin{proof}
Let $\varphi:X' \to X$ be the birational morphism as in Definition \ref{defi:gp}. 
We may assume that $\varphi$ is a log resolution of the triple $(X, B, \overline{\{ \eta \}})$. 
Furthermore, we may assume that there exists a computing divisor of $\operatorname{mld}_{\eta } (X, B+M)$ on $X'$. 

We define $A$, $A'$, $C$, $A_m'$ and $D_m$ as in the proof of Lemma \ref{lem:ord_limit}. 
Then we have 
\begin{align*}
\varphi ^* (K_X + B + M) &= K_{X'} + \varphi ^{-1}_* B + M' + \sum _G \bigl( 1 - a_G (X, B + M) \bigr) G, \\
\varphi ^* (K_X + B + D_m) &= K_{X'} + \varphi ^{-1}_* B + A'_m + \sum _G \bigl(1 - a_G (X, B + D_m) \bigr) G,
\end{align*}
where the sums are taken over all $\varphi$-exceptional divisors $G$. By the same way as in the proof of Lemma \ref{lem:ord_limit}, we have 
\[
a_G (X, B + M) =  a_G (X, B + D_m) + \frac{1}{m} \operatorname{coeff}_G C. 
\]
Hence by passing to limit, we have
\[
\lim _{m \to \infty} a_F (X, B + D_m) = a_F (X, B + M)
\]
for each $\varphi$-exceptional divisor $F$. 
Note that $\varphi$ is also a log resolution of $(X, B + D_m)$ by the generality of $A'_m$. 
If $\operatorname{codim} \eta = 1$, then the equality in (2) is clear 
since we may take $A'_m$ such that its coefficients approach zero as $m$ goes to infinity. 
Suppose $\operatorname{codim} \eta > 1$. 
Then, for each $m$, some $\varphi$-exceptional divisor computes $\operatorname{mld}_{\eta }(X, B+D_m)$. 
Therefore we have
\[
\lim _{m \to \infty} \operatorname{mld}_{\eta }(X, B+D_m) = \lim _{m \to \infty} \min _{F} a_F (X, B + D_m) 
= \min _{F} a_F (X, B + M), 
\]
where the minimum is taken over all divisors $F$ on $X'$
 with center $c_X (F) = \overline{\{ \eta \}}$. 
Hence, we have the equality in (2). 

By the same argument as in the proof of Lemma \ref{lem:ord_limit}(1), we have
\[
a_F (X, B + M) \ge  a_F (X, B + D_m)
\]
for any divisor $F$ over $X$. Therefore we have the inequality in (1). 
\end{proof}

\subsection{Lower semi-continuity conjecture for generalized mld's}
In this subsection, we reduce the lower semi-continuity of minimal log discrepancies 
for generalized pairs to that for usual log pairs (Theorem \ref{thm:LSC_reduction}). First, we recall the lower semi-continuity conjecture. 
\begin{conj}[LSC conjecture]\label{conj:LSC}
Let $(X, B+M)/Z$ be a generalized pair. 
Then the function 
\[
|X|_{0} \to \mathbb{R} \cup \{ - \infty \}; \quad x \mapsto \operatorname{mld}_{x}(X, B+M)
\] 
is lower semi-continuous, where $|X|_{0}$ denotes the set of all closed points of $X$ with the Zariski topology.
\end{conj}

\noindent
We introduce a different but equivalent formulation of this conjecture.  
\begin{conj}\label{conj:LSC'}
Let $(X, B+M)/Z$ be a generalized pair. 
For any two scheme-theoretic points $\xi$ and $\eta$ of $X$ with $\eta \in \overline{\{\xi \}}$ and $\operatorname{codim} \xi \ge 1$, 
it follows that
\[
\operatorname{mld}_{\eta}(X,B+M) \leq 
\operatorname{mld}_{\xi}(X,B+M) + 
\dim \xi - \dim \eta. 
\]
\end{conj}
  
\begin{thm}[{cf.\ \cite[Lemma 2.6]{Amb99}}]\label{thm:equivalence1}
The conjectures Conjecture \ref{conj:LSC} and Conjecture \ref{conj:LSC'} are equivalent.
\end{thm}

\begin{proof}
First, we assume Conjecture \ref{conj:LSC}. 
Let $\xi, \eta \in X$ be two scheme-theoretic points with $\eta \in \overline{\{\xi \}}$ and $\operatorname{codim} \xi \ge 1$. 
By Lemma \ref{lem:ambro}, there exist open subsets $U, V \subset |X|_{0}$ such that 
$\overline{\{ \eta \}} \cap U \not = \emptyset$, $\overline{\{ \xi \}} \cap V \not = \emptyset$, and 
\begin{align*}
\operatorname{mld}_{\eta}(X, B+M) &= \operatorname{mld}_{x} (X, B+M) - \dim \eta,\\
\operatorname{mld}_{\xi}(X, B+M) &= \operatorname{mld}_{y} (X, B+M) - \dim \xi 
\end{align*}
hold for any $x \in \overline{\{ \eta \}} \cap U$ and $y \in \overline{\{ \xi \}} \cap V$. 
We fix $x \in \overline{\{ \eta \}} \cap U$. 
Then by Conjecture \ref{conj:LSC}, possibly replacing $U$ with a smaller open subset that still contains $x$, we may assume 
\[
\operatorname{mld}_{y} (X, B+M) \ge \operatorname{mld}_{x}(X, B+M)
\]
for any $y \in U$. Since $\overline{\{ \xi \}} \cap U \cap V \not = \emptyset$, we have the inequality in Conjecture \ref{conj:LSC'}.

Next, we assume Conjecture \ref{conj:LSC'}. Let $a \in \mathbb{R} \cup \{ - \infty \}$. 
We define $F \subset |X|_{0}$ as 
\[
F := \bigl\{ x \in |X|_{0} \ \big| \ \operatorname{mld}_x (X, B + M) \le a  \bigr \}. 
\]
Let $\overline{F}$ be its closure. It is sufficient to show that $\overline{F} = F$. 
Take $x \in \overline{F}$. 
Let $\overline{F}'$ be an irreducible component of $\overline{F}$ containing $x$, 
and let $\xi$ be the scheme-theoretic point of $X$ corresponding to $\overline{F}'$. 
Then by Conjecture \ref{conj:LSC'}, we have 
\[
\operatorname{mld}_{x}(X, B+M) \le \operatorname{mld}_{\xi} (X, B+M) - \dim \xi. 
\]
On the other hand, by Lemma \ref{lem:ambro}, there exists an open subset $U \subset |X|_{0}$ with
$\overline{\{ \xi \}} \cap U \not = \emptyset$ such that
\[
\operatorname{mld}_{y}(X, B+M) = \operatorname{mld}_{\xi} (X, B+M) - \dim \xi
\]
holds for any $y \in \overline{\{ \xi \}} \cap U$. 
Since $\overline{\{ \xi \}} \cap U \cap F = \overline{F}' \cap U \cap F \not = \emptyset$,  
we have $\operatorname{mld}_{x} (X, B+M) \le a$, which implies $x \in F$. 
The proof is complete. 
\end{proof}

In order to prove Conjecture \ref{conj:LSC}, it is sufficient to prove it when $M'=0$. 

\begin{thm}\label{thm:LSC_reduction}
Conjecture \ref{conj:LSC} for $M'=0$ implies Conjecture \ref{conj:LSC}.
\end{thm}
\begin{proof}
Suppose that Conjecture \ref{conj:LSC} is true for $M'=0$. 
Suppose that Conjecture \ref{conj:LSC} does not hold for a generalized pair $(X, B+M)/Z$. 
Then there exists $a \in \mathbb{R} \cup \{ - \infty \}$ such that the set 
\[
F := \bigl \{ y \in |X|_{0} \ \big| \ \operatorname{mld}_y (X, B + M) \le a  \bigr \}.
\]
is not closed. 
We may assume that $a \ge 0$, otherwise $F$ is always closed (cf.\ Remark \ref{rmk:rmks}(\ref{item:-inftycl})). 
By the constructibility of $F$ which follows from Lemma \ref{lem:const}, and by the valuative criterion 
(cf.\ \cite{GW10}*{Corollary 15.10}) of properness, we can take an irreducible closed curve $C \subset \overline{F}$ such that 
$C \cap F$ is an open dense proper subset of $C$. 
We fix a closed point $x \in C \setminus F$. 

By Lemma \ref{lem:mld_limit}, there exists a sequence of effective divisors $D_m$ such that 
\begin{itemize}
\item[(1)] $\operatorname{mld}_{y} (X, B+M) \ge \operatorname{mld}_{y} (X, B+D_m)$ for each $m$ and $y \in |X|_{0}$, and 
\item[(2)] $\operatorname{mld}_{x} (X, B+M) = \lim _{m \to \infty} \operatorname{mld}_{x}(X, B+D_m)$ holds. 
\end{itemize}
We define 
\[
F_m := \bigl \{ y \in |X|_{0} \ \big| \ \operatorname{mld}_y (X, B + D_m) \le a  \bigr \}, 
\]
which is a closed subset by the assumption. 
Then $F \subset F_m$ holds for each $m$ by (1). Since $F_m$ is a closed subset of $X$ 
that contains an open dense subset of  $C$, 
it follows that $x \in F_m$, and hence $\operatorname{mld}_{x}(X,B+D_m) \le a$. 
By (2), we have $\operatorname{mld}_{x} (X, B+M) \le a$, which contradicts $x \not \in F$. 
\end{proof}

\begin{cor}\label{cor:3dimLSC}
Conjecture \ref{conj:LSC} is true for the following cases. 
\begin{itemize}
\item[(1)] $\dim X \le 3$. 
\item[(2)] $X$ is smooth. 
\item[(3)] More generally, $X$ has locally complete intersection singularities. 
\item[(4)] $X$ has quotient singularities. 
\end{itemize}
\end{cor}
\begin{proof}
Conjecture \ref{conj:LSC} is known when $M' = 0$ for each cases (1)-(4) by \cite{Amb99}, \cite{EMY03}, \cite{EM04}, and \cite{Nak16a}, respectively. 
Hence the general case also holds by Theorem \ref{thm:LSC_reduction}. 
\end{proof}

\section{Termination of generalized MMP's due to Shokurov}\label{section:termination}

In this section, we discuss the relation to the conjecture of termination of flips. 

\begin{conj}[Termination of flips]\label{conj:termination}
There is no infinite sequence of flips for generalized lc pairs. 
\end{conj}

\noindent
Shokurov proved that this conjecture for usual pairs is true if we assume the ACC conjecture and the LSC conjecture \cite{Sho04}. 
The purpose of this section is to explain his proof in the generalized setting. 
First, we prove some lemmas which will be used in the proof. 

The generalized minimal log discrepancies satisfy the monotonicity as in the case of usual pairs. 

\begin{lem}\label{lem:mld_mono}
Let $(X,B+M) \dashrightarrow (X^+, B^+ +M^+)$ be the flip of a generalized pair $(X, B +M)/Z$. 
Then the following hold. 
\begin{enumerate}
\item $a_E(X, B+M) \leq a_E(X^{+}, B^+ +M^+)$ holds for any divisor $E$ over $X$.

\item $a_E(X, B +M) < a_E (X^{+}, B^+ +M^+)$ holds 
if $E$ is a divisor over $X$ whose center $c_X (E) $ is contained in the flipping locus. 
\end{enumerate}
\end{lem}
\begin{proof}
The assertion follows from the negativity lemma. 
We refer the reader to \cite{KM98}*{Lemma 3.38} for the detailed proof for usual pairs. 
The same proof works for generalized pairs. 
\end{proof}

For our purpose, we generalize the ACC conjecture and the LSC conjecture to those for scheme-theoretic points. 

\begin{lem}\label{lem:ACC_T}
Assume Conjecture \ref{conj:ACCintro}. 
Let $n \in \mathbb{Z} _{>0}$ and let $I \subset [0, + \infty)$ be a DCC subset. 
Then the following set
\[
\left\{ \operatorname{mld}_{\eta} (X, B + M) \ \middle | 
\begin{array}{l}
\text{$(X, B+M)/Z$ is a generalized pair with $\dim X = n$, }\\
\text{$B \in I$, $M \in I$ and $\eta \in X$ with $\operatorname{codim} \eta \ge 1$.}
\end{array}
\right \}
\]
satisfies the ACC, where $\eta$ is a scheme-theoretic point of $X$. 
\end{lem}
\begin{proof}
The assertion follows from Conjecture \ref{conj:ACCintro} and Lemma \ref{lem:ambro}. 
\end{proof}

\begin{lem}\label{lem:LSC_T}
Assume Conjecture \ref{conj:LSC}. 
Let $(X, B +M)/Z$ be a generalized pair, and $d$ a non-negative integer such that $d \le \dim X -1$. 
Then the function 
\[
m: |X|_{d} \to \mathbb{R} \cup \{ - \infty \}; \quad \eta \mapsto \operatorname{mld}_{\eta}(X, B +M)
\]
is lower semi-continuous, where $|X|_{d}$ is the set of all $d$-dimensional scheme-theoretic points of $X$ with the Zariski topology. 
\end{lem}
\begin{proof}
The assertion follows from Conjecture \ref{conj:LSC} and Lemma \ref{lem:ambro}. 
\end{proof}

\begin{lem}\label{lem:const_T}
Let $(X, B +M)/Z$ be a generalized pair. Then the following set 
\[
\bigl \{ \operatorname{mld}_{\eta} (X, B + M) \ \big| \ \text{$\eta \in X$ with $\operatorname{codim} \eta \ge 1$.} \bigr \}
\]
is a finite set. 
\end{lem}
\begin{proof}
The assertion follows from Lemmas \ref{lem:const} and \ref{lem:ambro}. 
\end{proof}

In the proof of Theorem \ref{thm:termination}, 
the finiteness of the dimension of $N_k (X)_{\mathbb{Q}}$ plays an important role. 

\begin{defi}[{cf.\ \cite[Examples 19.1.3--19.1.6]{Ful98}, \cite[Ch.\ II.\ (4.1.5)]{Kol96}}]\label{defi:cycle}
Let $X$ be a reduced projective scheme and let $k$ be a non-negative integer. 
We denote by $Z_k (X) _{\mathbb{Q}}$ the group of $k$-dimensional algebraic cycles on $X$ with rational coefficients. 
All cycles which are numerically equivalent to zero form a subgroup of $Z_k(X) _{\mathbb{Q}}$, 
and we denote by $N_k (X)_{\mathbb{Q}}$ the quotient group. Then $N_k (X)_{\mathbb{Q}}$ is a finite-dimensional $\mathbb{Q}$-vector space 
(cf.\ \cite[Example 19.1.4]{Ful98}). 
\end{defi}

\begin{lem}\label{lem:cycle}
Let $f: X \dashrightarrow Y$ be a dominant rational map of reduced projective schemes. 
Suppose that $f$ induces a birational map on each irreducible component of $X$ and $Y$. 
Let $k$ be a positive integer. 
Suppose that $f^{-1}$ does not contract any $k$-dimensional subvariety of $Y$. 
Then the following hold. 
\begin{itemize}
\item[(1)] $\dim N_k(X)_{\mathbb{Q}} \ge \dim N_k(Y)_{\mathbb{Q}}$ holds. 
\item[(2)] $\dim N_k(X)_{\mathbb{Q}} > \dim N_k(Y)_{\mathbb{Q}}$ holds if $f$ contracts some $k$-dimensional subvariety of $X$.
\end{itemize}
\end{lem}

\begin{proof}
We have (1) since $f_*: N_{k}(X)_{\mathbb{Q}} \to N_{k}(Y)_{\mathbb{Q}}$ is surjective. 
Suppose that a subvariety $W \subset X$ of dimension $d$ is contracted by $f$. 
Then the cycle $[W]$ satisfies $[W] \not \equiv 0$ in $Z_k(X)_{\mathbb{Q}}$ and $f_* [W] \equiv 0$ in $Z_k(Y)_{\mathbb{Q}}$. 
Hence we have (2). 
\end{proof}

Shokurov proved that the ACC conjecture and the LSC conjecture imply the conjecture of termination of flips for usual pairs in \cite{Sho04}. 
His proof can be easily extended to the generalized setting. We shall explain his argument below. 

\begin{thm}\label{thm:termination} 
Conjectures \ref{conj:ACCintro} and \ref{conj:LSC} imply Conjecture \ref{conj:termination} 
for gerenalized pairs $(X, B + M)/Z$ with $Z$ projective. 
\end{thm}

\begin{proof}
Let 
\[
X = X_0 \overset{f_0}{\dashrightarrow} X_1 \overset{f_1}{\dashrightarrow} \  \cdots \ 
\overset{f_{i-1}}{\dashrightarrow} X_i \overset{f_{i}}{\dashrightarrow} \ \cdots, 
\]
be an infinite sequence of $(K_X + B + M)$-flips. 
Let $B_i$ and $M_i$ be the strict transforms of $B$ and $M$ on $X_i$, respectively. 
Then each $(X_i, B_i + M_i)/Z$ is a generalized pair and each $f_i$ is a $(K_{X_i} + B_i + M_i)$-flip. 
We denote by $g_i: X_i \to Y_i$ the corresponding flipping contraction and by $g_i ^+: X_{i+1} \to Y_i$ its flip. 
\[
\xymatrix{
X_i \ar[rd]_{g_i} \ar@{.>}[rr]^{f_i}&&X_{i+1} \ar[ld]^{g_i^+}\\
&Y_i&
}
\]
We define $Z_i \subset X_i$ and $Z_i ^+ \subset X_{i+1}$ as follows: 
\begin{itemize}
\item $Z_i := \operatorname{excep} (g_i)$ is the flipping locus, that is, the exceptional locus of $g_i$, and 
\item $Z_i ^+ := \operatorname{excep} (g^+ _i)$ is the flipped locus, that is, the exceptional locus of $g^+_i$. 
\end{itemize}
For non-negative integers $i$ and $\ell$, we denote
\begin{itemize}
\item $a_i := \operatorname{mld} _{Z_i} (X_i, B_i + M_i)$, 
\item $\alpha _i := \inf \{ a_j \mid j \ge i \}$, and 
\item $\alpha _i ^{\ell} := \min \{ a_j \mid \ell \ge j \ge i \}$. 
\end{itemize}

\noindent
\underline{\textbf{STEP 1:}}\ \ 
This step shows that for any $i$ and $\ell$ with $i \le \ell$, 
there exists a scheme-theoretic point $\eta \in X_i$ such that 
$\alpha _i ^{\ell} = \operatorname{mld} _{\eta} (X_i, B_i + M_i)$.

Let $\ell '$ be the minimum $\ell '$ with $i \le \ell ' \le \ell$ such that $\alpha _i ^{\ell} = a_{\ell '}$. 
We prove by induction that for any $j$ with $i \le j \le \ell'$, there  exists $\eta _j \in X_j$ such that 
$\alpha _i ^{\ell} = \operatorname{mld} _{\eta _j} (X_j, B_j + M_j)$. 

Suppose that $\alpha _i ^{\ell} = \operatorname{mld} _{\eta _j} (X_j, B_j + M_j)$ holds for $i < j \le \ell '$ and some $\eta _j$. 
Then, since 
\[
\operatorname{mld} _{Z_{j-1}} (X_{j-1}, B_{j-1} + M_{j-1}) = a_{j-1} > \alpha _i ^{\ell} = \operatorname{mld} _{\eta _j} (X_j, B_j + M_j), 
\]
it follows that $f_{j-1}$ is isomorphic over $\eta _j$ by Lemma \ref{lem:mld_mono}. 
Hence $\eta _{j-1} := f_{j-1}^{-1} (\eta _j)$ satisfies 
\[
\operatorname{mld} _{\eta _{j-1}} (X_{j-1}, B_{j-1} + M_{j-1}) = \operatorname{mld} _{\eta _j} (X_j, B_j + M_j) = \alpha _i ^{\ell}, 
\]
which proves the claim of STEP 1. 

\vspace{2mm}

\noindent
\underline{\textbf{STEP 2:}}\ \ 
This step shows that we may assume the existence of a non-negative real number $a$ such that 
\begin{itemize}
\item $a_i \ge a$ holds for any $i \ge 0$, and 
\item $a_i = a$ holds for infinitely many $i$. 
\end{itemize}

By STEP 1 and Lemma \ref{lem:const_T}, the set $\{ \alpha _i ^{\ell} \mid \ell \ge i \}$ is a finite set for each $i$. 
Therefore $\alpha _i = \alpha_i ^{\ell}$ holds for some $\ell \ge i$, and hence 
\[
\alpha _i = \operatorname{mld} _{\eta _i} (X_i, B_i + M_i)
\]
for some $\eta_i \in X_i$. 
Since the sequence $\alpha _i$ is non-decreasing, by the ACC conjecture (Lemma \ref{lem:ACC_T}), 
there exists a positive integer $N$ such that $\alpha _i = \alpha _N$ holds for any $i \ge N$. 
Therefore, we may have the desired $a$, possibly passing to a tail of the sequence. 

\vspace{2mm}

\noindent
\underline{\textbf{STEP 3:}}\ \ 
This step shows that we may assume the existence of a non-negative integer $d$ such that
\begin{itemize}
\item For any $i$, any scheme-theoretic point $\eta \in Z_i$ with $\operatorname{mld}_{\eta} (X_i, B_i + M_i) = a$ satisfies $\dim \eta \le d$. 
\item For infinitely many $i$, there exists a $d$-dimensional scheme-theoretic point $\eta \in Z_i$ such that 
$\operatorname{mld}_{\eta} (X_i, B_i + M_i) = a$. 
\end{itemize}

For non-negative integer $i$, we denote 
\begin{align*}
d_i &:= \max \bigl \{ \dim \eta \ \big | \  
\text{$\eta \in Z_i$ such that $\operatorname{mld}_{\eta} (X_i, B _i + M_i) = a$} \bigr \}, \\ 
e_i &:= \max \{ d_j \mid j \ge i \}. 
\end{align*}
Since the sequence $e _i$ is non-increasing, 
there exists a positive integer $N$ such that $e _i = e _N$ holds for any $i \ge N$. 
Therefore, we may have the desired $d$, possibly passing to a tail of the sequence.

\vspace{2mm}

\noindent
\underline{\textbf{STEP 4:}}\ \ 
Let $S_i$ be the set of the $d$-dimensional scheme-theoretic points $\eta \in X_i$ with 
$\operatorname{mld}_{\eta} (X_i, B_i + M_i) \le a$. 
Let $W_i \subset X_i$ be the Zariski closure of $S_i$. 
Then by the LSC conjecture (Lemma \ref{lem:LSC_T}), the following condition holds. 
\begin{itemize}
\item Any $d$-dimensional scheme-theoretic point $\eta \in W_i$ belongs to $S_i$. 
\end{itemize}

\vspace{2mm}

\noindent
\underline{\textbf{STEP 5:}}\ \ 
In this step, we prove that $f_i$ induces 
\begin{itemize}
\item a bijective map $S_i \setminus Z_i \to S_{i+1}$, and 
\item a dominant morphism $f_i' : W_i \setminus Z_i \to W_{i+1}$. 
\end{itemize}

Let $\eta \in S_i \setminus Z_i$. 
Then $f_i(\eta) \in S_{i+1}$ holds because 
\[
\operatorname{mld}_{f_i(\eta)} (X_{i+1}, B_{i+1} + M_{i+1}) = \operatorname{mld}_{\eta} (X_i, B_i + M_i) \le a. 
\]
Let $\eta \in S_{i+1}$. 
Suppose $\eta \in Z_{i}^+$. Then by Lemma \ref{lem:mld_mono}, we have 
\[
\operatorname{mld}_{Z_i}(X_i, B_i + M_i) < \operatorname{mld}_{\eta}(X_{i+1}, B_{i+1} + M_{i+1}) \le a, 
\]
and it contradicts STEP 2. Therefore it follows that $\eta \not \in Z_{i}^+$ and 
it shows that $f_i$ induces a bijective map $S_i \setminus Z_i \to S_{i+1}$. 
The second assertion follows from the first one.

\vspace{2mm}

\noindent
\underline{\textbf{STEP 6:}}\ \ 
By STEP 5, the number of the irreducible components of $W_i$ is non-increasing. 
Hence, passing to a tail of the sequence, we may assume that 
$f_i$ induces a birational map on each irreducible component of $W_i$. 
Further by STEP 5 and STEP 4, $f_i ^{-1}$ does not contract any $d$-dimensional subvariety of $W_{i+1}$. 
On the other hand, by the choice of $d$ in STEP 3, there exist infinitely many $i$'s such that 
$\operatorname{mld}_{\eta _i}(X_i, B_i + M_i) = a$ holds for some $d$-dimensional point $\eta _i \in Z_i$. 
For such $i$ and $\eta _i$, it follows from STEP 5 that $\eta _i \in W_i$ is contracted by $f_i$. 
Hence it contradicts Lemma \ref{lem:cycle}. 
\end{proof}

\begin{rmk}
By the proof of Theorem \ref{thm:termination}, it can be seen that
for proving Conjecture \ref{conj:termination} in dimension $n$, 
it is sufficient to assume Conjectures \ref{conj:ACCintro} and \ref{conj:LSC} in the same dimension. 
\end{rmk}

\section{ACC conjecture for surfaces}\label{section:ACC2dim}

In this section, we prove the ACC conjecture for surfaces (Theorem \ref{thm:ACC_2dim}). 
We generalize the argument in \cite{Sho} to the generalized setting with more detail. 

First, we prove the following convexity property of generalized log discrepancies 
(cf.\ \cite{Sho}, \cite[Proposition 2.37]{Kol13}, \cite[Proposition 4.1]{MN18}). 

\begin{lem}\label{lem:mono} 
Let $(X, B +M)/Z$ be a generalized lc surface and let $f: Y \to X$ be a 
projective birational morphism from a smooth surface $Y$. 
Suppose that $a_E(X, B + M) \le 1$ holds for every $f$-exceptional divisor $E$. 
Then the following hold. 

\begin{enumerate}
\item[(1)] Let $E_1$, $E_2$ and $E_3$ be distinct $f$-exceptional divisors such that 
\begin{itemize} 
\item $E_2$ meets both $E_1$ and $E_3$, and 
\item $E^2_2 \leq -2$. 
\end{itemize}
Then it follows that
\[
a_2 \leq \frac{1}{2}(a_1+a_3),
\]
where $a_i=a_{E_i}(X, B+M)$.

\item[(2)] If $E$ is an $f$-exceptional divisor, then $a_E(X, B+M) \leq \frac{2}{-E^2}$ holds. 
In particular, if $(X, B+M)$ is generalized $\epsilon$-lc, then $-E^2 \leq \frac{2}{\epsilon}$ holds. 

\item[(3)] Let $E_1, E_2$ and $E_3$ be as in (1). Suppose that $(X, B+M)$ is $\epsilon$-lc and $E_2^2\leq -3$, then it follows that
\[
a_3-a_2\geq a_2-a_1 + \epsilon.
\]
\end{enumerate}
\end{lem}

\begin{proof} 
First, we shall prove (1) and (3). Let $\{ E_i \}$ be the set of all $f$-exceptional divisors. 
Let $M_Y$ be the push-forward of $M'$ on $Y$. 
Then we have
\[
f^*(K_X + B+M) = K_Y + \Bigl ( f^{-1}_* B +  \sum_{i} (1-a_i) E_i \Bigr )  + M_Y 
\]
for $a_i = a_{E_i}(X, B+M)$. 
Note that $1 - a_i \ge 0$ holds for every $i$ by assumption. 
We have 
\begin{align*}
0= {} &  f^*(K_X + B+M) \cdot E_2\\
= {} &  (K_Y+E_2) \cdot E_2 - a_2 E^2_2 + (1-a_1) E_1 \cdot E_2 + (1-a_3)E_3 \cdot E_2 \\
{} & + f^{-1}_* B \cdot E_2 + M_Y \cdot E_2 + \sum_{i \not= 1,2,3}(1-a_i)E_i \cdot E_2.
\end{align*}

\noindent
Note that $M_Y \cdot E_2 \ge 0$ holds by the projection formula and the nefness of $M'$. 
Furthermore, it is clear that 
\[
\quad (K_Y+E_2) \cdot E_2 \geq -2, \quad f^{-1}_* B \cdot E_2 \ge 0, \quad 
\]
and the assumptions give
\[
-a_2 E_2 ^2 \ge 2 a_2, \quad (1-a_1) E_1 \cdot E_2 \ge 1 - a_1, \quad (1-a_3)E_3 \cdot E_2 \ge 1 - a_3. 
\]
By combining them, we obtain the desired inequality $2\geq2a_2+(1-a_1)+(1-a_3)$, which proves (1). 

If $E_2 ^2 \le -3$ and $a_2 \ge \epsilon$, we have 
\[
a_3-a_2 \geq (a_2-a_1)+a_2 \geq  (a_2-a_1)+\epsilon, 
\]
which proves (3). 

(2) follows from the same calculation of $f^*(K_X + B+M) \cdot E$. 
\end{proof}

\begin{lem}\label{lem:smooth_surf} 
Let $(X, B+M)/Z$ be a generalized lc surface and let $x \in X$ be a closed point. 
If $\operatorname{mld}_x (X, B+M) >1$, then $X$ is smooth at $x$.
\end{lem}

\begin{proof}
Since $X$ is numerically lc, $X$ is $\mathbb{Q}$-Gorenstein. 
Then
\[
\operatorname{mld}_x(X) \ge \operatorname{mld}_x(X, B+M) >1
\] 
implies that $X$ is terminal at $x$. Hence the surface $X$ is smooth at $x$. 
\end{proof}

If the minimal log discrepancy at a smooth closed point $x \in X$ is at least $\dim X - 1$, 
then the minimal log discrepancy is computed by the exceptional divisor obtained by the blow-up at $x$. 
See Definition \ref{defi:defs}(\ref{item:mult}) for the definition of the multiplicity in the generalized setting. 

\begin{lem}[{cf.\ \cite[Example]{Sho}}]\label{lem:highmld}
Let $(X, B+M)/Z$ be a generalized pair and let $\eta$ be a scheme-theoretic point of $X$ of $\operatorname{codim} \eta \ge 1$. 
Suppose that $X$ is smooth at $\eta$. 
Then $\operatorname{mld}_{\eta}(X, B+M) \ge \operatorname{codim} \eta - 1$ holds if and only if 
\[
\operatorname{mult}_{\eta} (B+M) \leq 1
\]
holds. 
Moreover, in this case, $\operatorname{mld}_{\eta} (X, B+M) = a_E (X, B+M)$ holds 
for the divisor $E$ obtained by the blow-up of $X$ along $\overline{\{ \eta \}}$. 
\end{lem}

\begin{proof}
If $\operatorname{codim} \eta = 1$, then the assertion is trivial. We assume that $\operatorname{codim} \eta \ge 2$. 

Let $f:Y \to X$ be the blow-up of $X$ along $\overline{\{ \eta \}}$, 
and let $E$ be the exceptional divisor dominating $\overline {\{\eta\}}$. 
Then we have 
\[
\operatorname{mld}_{\eta} (X, B+M) \leq a_E(X, B + M) = \operatorname{codim} \eta - \operatorname{mult}_{\eta}(B+M). 
\]
Therefore, the condition $\operatorname{mld}_{\eta} (X, B+M) \ge \operatorname{codim} \eta -1$ 
implies $\operatorname{mult}_{\eta}(B+M) \leq 1$. 

Next, we assume that $\operatorname{mult}_{\eta}(B+M) \leq 1$. 
Let $F$ be a computing divisor of $\operatorname{mld}_{\eta} (X, B+M)$. 
Then by Zariski's lemma (cf.\ \cite[Lemma 2.45]{KM98}), $F$ can be obtained by the sequence of blow-ups along its center. 
Therefore there exists a sequence of blow-ups
\[
X_{\ell} \to X_{\ell -1} \to \cdots \to X_1 \to X_0 = X
\]
with the following conditions: 
\begin{itemize}
\item $X_{i+1} \to X_i$ is the normalization of the blow-up along $C_i := c_{X_{i}}(F)$. 
\item $E_{i+1} \subset X_{i+1}$ is the exceptional divisor of $X_{i+1} \to X_i$ dominating $C_i$. 
\item $E_{\ell} = F$ holds, and $\operatorname{codim} C_i \ge 2$ holds for each $0 \le i \le \ell -1$. 
\end{itemize}
Then, we note that the following hold. 
\begin{itemize}
\item $C_0 = \overline{\{ \eta \}}$. 
\item For each $0 \le i \le \ell -1$, 
$E_{i+1}$ is the only exceptional divisor of $X_{i+1} \to X_{i}$ satisfying $C_{i+1} \subset E_{i+1}$. 
\item For each $0 \le i \le \ell -1$, $X_i$ is smooth at the generic point of $C_i$. 
\end{itemize}
We set $a_i := a_{E_i} (X, B+ M)$ for $1 \le i \le \ell$. 
To prove that $E_1$ is also a computing divisor, it is sufficient to show $a_{\ell} \ge a_1$. 
We shall prove $a_i \ge a_1$ for each $1 \le i \le \ell$ by induction on $i$. 
Let $1 \le c \le \ell - 1$. Suppose that $a_i \ge a_1$ holds for each $1 \le i \le c$. 
Since we have
\[
a_i \ge a_1 = \operatorname{codim} \eta - \operatorname{mult} _{\eta}(B+M) \ge 2-1 = 1
\]
for each $1 \le i \le c$, it follows that 
\[
a_{c+1} = a_{E_{c+1}}(X, B + M) \ge a_{E_{c+1}} \bigl( X_c, B_c - (a_c-1)E_c  + M_c \bigr), 
\]
where $B_c$ is the strict transform of $B$ on $X_c$, and $M_c$ is the push-forward of $M'$ on $X_c$ 
(possilby replacing $X'$ with a higher model). 
Then by Lemma \ref{lem:mult} below, we have 
\begin{align*}
a_{E_{c+1}} \bigl( X_c, B_c - (a_c-1)E_c  + M_c \bigr) 
& = \operatorname{codim} C_c - \operatorname{mult}_{\eta _{C_c}} (B_c + M_c) + (a_c - 1) \\
&\ge \operatorname{codim} C_c - \operatorname{mult}_{\eta} (B + M) + (a_c-1) \\
&\ge 2 - 1 + (a_c - 1) = a_c \ge a_1, 
\end{align*}
where $\eta _{C_c}$ is the generic point of $C_c$. 
The proof is complete. 
\end{proof}

\begin{lem}\label{lem:mult}
Let $(X, B+ M)/Z$ be a generalized pair and let $\eta$ be a scheme-theoretic point of $\operatorname{codim} \eta \ge 2$. 
Suppose that $X$ is smooth at $\eta$. 
Let $f: Y \to X$ be the blow-up of $X$ along $\overline{ \{ \eta \} }$. 
We may assume that $\varphi$ in Definition \ref{defi:gp} factors through $f$. 
Then for every scheme-theoretic point $\xi$ of $Y$ such that $f(\xi) = \eta $, we have 
\[
\operatorname{mult}_{\xi} (\widetilde{B} + M_Y) \le \operatorname{mult}_{\eta} (B + M), 
\]
where $\widetilde{B}$ is the strict transform of $B$ on $Y$, and $M_Y$ is the push-forward of $M'$ on $Y$. 
\end{lem}
\begin{proof}
If $M' = 0$, then the assertion is well-known. We shall prove the general case from the case when $M' = 0$. 

By Lemma \ref{lem:ord_limit}, there exists a sequence of effective divisors $D_m$ such that 
\[
\operatorname{mult}_{\eta} (B + M) = \lim _{m \to \infty} \operatorname{mult}_{\eta} (B + D_m). 
\]
By its proof, we may also assume
\[
\operatorname{mult}_{\xi} (\widetilde{B} + M_Y) = 
\lim _{m \to \infty} \operatorname{mult}_{\xi} (\widetilde{B} + \widetilde{D}_m), 
\]
where $\widetilde{D}_m$ is the strict transform of $D_m$ on $Y$. 
Hence the desired inequality follows from the case when $M' =0$. 
\end{proof}

First, we prove the ACC property on the interval $[1, 2]$. 

\begin{lem}\label{lem:acc>1}
For any DCC subset $I \subset [0, + \infty)$, the set $A_{\textup{gen}}(2,I) \cap [1, + \infty)$ satisfies the ACC. 
\end{lem}
\begin{proof}
Let $(X, B + M)/Z$ be a generalized lc surface with $B \in I$ and $M \in I$. 
Suppose that $\operatorname{mld}_x(X, B + M) > 1$. Then by Lemmas \ref{lem:smooth_surf} and \ref{lem:highmld}, we have 
\[
\operatorname{mld}_x(X, B + M) = 2 - \operatorname{mult}_x (B + M). 
\]
Note that $\operatorname{mult}_x (B + M)$ is contained in the set 
\[
\Bigl \{ \sum _{i=1} ^ {\ell} f_i \ \Big | \ \ell \in \mathbb{Z}_{\ge 0}, f_i \in I \Bigr \} \cap [0,1], 
\]
which satisfies the DCC. 
Therefore, the set $A_{\textup{gen}}(2,I) \cap [1, + \infty)$ satisfies the ACC. 
\end{proof}

\begin{lem}\label{lem:Y}
Let $(X, B + M)/Z$ be a generalized lc surface and $x \in X$ a closed point with 
$\operatorname{mld}_x (X, B + M) < 1$. 
Then there exists a projective birational morphism $f : Y \to X$ from a smooth surface $Y$ with the following conditions: 
\begin{itemize}
\item $a_F(X, B + M) \le 1$ holds for any $f$-exceptional divisor $F$ over $x$, and 
\item $a_E (X, B + M) = \operatorname{mld}_x (X, B + M)$ holds for some $f$-exceptional divisor $E$ over $x$. 
\end{itemize}
Moreover, we can take such $Y$ with either of the following conditions. 
\begin{enumerate}
\item $f$ is the minimal resolution of $X$. 
\item $E$ is the unique computing divisor of $\operatorname{mld}_x (X, B + M)$ on $Y$, 
and $E$ is the unique $(-1)$-curve among the $f$-exceptional divisors over $x$. 
\end{enumerate}
\end{lem}
\begin{proof}
Let $Y_0 \to X$ be the minimal resolution of $X$. 
If there exists a computing divisor of $\operatorname{mld}_x (X, B + M)$ on $Y_0$, 
then $Y := Y_0$ satisfies the conditions. 
Otherwise, we fix a computing divisor $E'$ of $\operatorname{mld}_x (X, B + M)$ on some higher model. 
We take the blow-up $Y_1 \to Y_0$ of $Y_0$ at the center $c_{Y_0}(E')$. 
We continue this process 
\[
Y_{\ell} \to Y_{\ell -1} \to \cdots \to Y_1 \to Y_0
\]
until $E'$ is a divisor on $Y_{\ell}$ (cf.\ \cite{KM98}*{Lemma 2.45}). 
Let $1 \le j \le \ell$ be the minimum $j$ such that there exists a computing divisor $E$ on $Y_j$. 
Then we set $Y := Y_j$ and $f: Y \to X$. 

Then it is sufficient to show that 
\begin{itemize}
\item $a_F(X, B + M) \le 1$ holds for any $f$-exceptional divisor $F$ over $x$. 
\end{itemize}
For each $0 \le i \le j$, we define $B_i$ on $Y_i$ by $(K_X + B + M)|_{Y_i} = K_{Y_i} + B_i + M_{Y_i}$, 
where $M_{Y_i}$ is the push-forward of $M'$ on $Y_i$ (possilby replacing $X'$ with a higher model). 
Then it is sufficient to show $B_{j-1} \ge 0$. 
Note that $B_0 \ge 0$ holds by the negativity lemma. 
Then we obtain $B_i \ge 0$ for $0 \le i \le j-1$ by induction on $i$ since 
$\operatorname{mult}_{y_i} (B_i + M_{Y_i}) > 1$ for $y_i := c_{Y_i}(E')$ and $i \le j-2$ by Lemma \ref{lem:highmld}. 
\end{proof}

Next, we prove the ACC property when the number of the $f$-exceptional divisors in Lemma \ref{lem:Y} is bounded from above. 

\begin{lem}\label{lem:n_bounded}
Let $\epsilon$ be a positive real number, $\ell$ a positive integer, and $I \subset [0,+ \infty)$ a DCC subset. 
Then there exists an ACC subset $J = J(\epsilon, \ell, I) \subset [0,1]$ 
depending only on $\epsilon$, $\ell$ and $I$ such that 
$\operatorname{mld}_x (X, B+M) \in J$ holds for any generalized lc surface 
$(X, B+M)/Z$ and a closed point $x \in X$ with the following three conditions: 
\begin{itemize}
\item $B \in I$, $M \in I$, 
\item $\epsilon \le \operatorname{mld}_x (X, B + M) < 1$, and
\item there exists a birational morphism $f : Y \to X$ as in Lemma \ref{lem:Y} such that 
the number of the $f$-exceptional divisors over $x$ is at most $\ell$. 
\end{itemize}
\end{lem}

\begin{proof}
Let $E$ be a computing divisor on $Y$ of $\operatorname{mld}_x (X, B + M)$. 
Since $a_E(X, B + M) <1$, we can construct the extraction $W \to X$ of $E$ (cf.\ Theorem \ref{thm:extraction}).  
Then we have
\[
\operatorname{deg} (K_{W} + B_W + M_{W} + E)|_{E} = a_E (X, B + M)E^2 \leq 0, 
\]
where $B_W$ is the strict transform of $B$ on $W$, and 
$M_{W}$ is the push-forward of $M'$ on $W$ (possilby replacing $X'$ with a higher model). 

Let $\nu: E^{\nu} \to E$ be the normalization. 
Then by the generalized adjunction (see Definition \ref{def:gadj} and Lemma \ref{lem:adjcoef}), 
there exists a generalized boundary $B_{E^{\nu}}+M_{E^{\nu}}$ on $E^{\nu}$ such that 
\begin{itemize}
\item $K_{E^{\nu}} + B_{E^{\nu}} + M_{E^{\nu}} \sim _{\mathbb{R}} (K_{W} + B_W + M_{W} + E)|_{E^{\nu}}$, 
\item $B_{E^{\nu}} \in I'$ and $M_{E^{\nu}} \in I'$ hold for some DCC subset $I' \subset [0, + \infty)$ which is determined by $I$. 
\end{itemize}
Since $\operatorname{deg} (K_{E^{\nu}}) = -2$ and 
\[
\operatorname{deg} (B_{E^{\nu}} + M_{E^{\nu}}) 
\in \Bigl \{ \sum _{i=1} ^ {k} f_i \ \Big | \ k \in \mathbb{Z}_{\ge 0}, f_i \in I' \Bigr \}, 
\]
the non-negative real number
\[
a_E (X, B + M) \cdot (- E^2) = 2 - \operatorname{deg} (B_{E^{\nu}} + M_{E^{\nu}})
\] 
is contained in an ACC set determined by $I$. 

Thus it is sufficient to show that $- E^2$ has finitely many possibilities depending only on $\epsilon$ and $\ell$. 
By Lemma \ref{lem:mono}(2) and the assumption on $\ell$ and $\epsilon$, 
the weighted dual graph of the $f$-exceptional divisors over $x$ has finitely many possibilities 
depending only on $\epsilon$ and $\ell$. 
Since $W$ is obtained by contracting the $f$-exceptional divisors except for $E$, 
the self-intersection number $-E^2$ on $W$ also has finitely many possibilities depending only on $\epsilon$ and $\ell$. 
The proof is complete. 
\end{proof}

We study minimal log discrepancies which are close to $1$. 
\begin{lem}\label{lem:all-2}
Let $\delta$ be a positive real number. 
Then there are no generalized lc surface $(X, B + M)/Z$ and a closed point $x \in X$ 
with the following three conditions. 
\begin{itemize}
\item $\max \bigl \{\frac{2}{3}, 1- \frac{1}{2} \delta  \bigr \} < \operatorname{mld}_x(X, B + M) < 1$. 
\item $B \ge \delta$ and $M \ge \delta$. 
\item There exists a computing divisor of $\operatorname{mld}_x(X, B + M)$ on the minimal resolution $Y$ of $X$. 
\end{itemize}
\end{lem}
\begin{proof}
Suppose the contrary that there exist such $(X,B+M)/Z$ and $x \in X$.
Let $f: Y \to X$ be the minimal resolution of $X$. 
Since $\operatorname{mld}_x(X, B+ M) > \frac{2}{3}$, it follows from Lemma \ref{lem:mono}(2) that 
any $f$-exceptional divisor $E$ over $x$ is a $(-2)$-curve. 
Then, by adjunction of $(X, B+ M)$ to an  $f$-exceptional divisor $E$ over $x$, 
we have
\[
2 - 2 a_{E}(X, B + M) \ge (f^{-1}_* B + M_Y) \cdot E
\]
(see the proof of Lemma \ref{lem:mono}), where $M_Y$ is the push-forward of $M'$ on $Y$ (possilby replacing $X'$ with a higher model). 
Since 
\[
2 - 2 a_{E}(X, B + M) \le 2 - 2 \operatorname{mld}_x(X, B + M) < \delta, 
\]
we have $(f^{-1}_* B+ M_Y) \cdot E = 0$ for any $f$-exceptional divisor $E$ over $x$, 
which implies $B = 0$ and $M_Y = f^* M$ near $x$. 
Therefore we have 
\[
\operatorname{mld}_x(X, B + M) = \operatorname{mld}_x(X) = 1, 
\]
which contradicts the assumption $\operatorname{mld}_x(X, B + M) < 1$. 
Note that the second equality $\operatorname{mld}_x(X) = 1$ follows from the fact 
that any $f$-exceptional divisor $E$ over $x$ is a $(-2)$-curve. 
\end{proof}

We introduce a notation that will be used in Lemma \ref{lem:I}. 

\begin{defi}
Let $X$ be a klt surface and $x \in X$ a closed point, 
and let $f: Y \to X$ be a projective birational morphism from a smooth surface $Y$. 
Let $\Gamma$ be the dual graph of the $f$-exceptional divisors over $x$. 
For adjacent vertices $E$ and $F$ in $\Gamma$, 
we denote by $\Gamma'_{E,F}$ the subgraph of $\Gamma$ that is obtained by removing the edge connecting $E$ and $F$ 
from $\Gamma$. Since $\Gamma$ is a tree (cf.\ \cite[Theorem 4.7]{KM98}), $\Gamma'_{E,F}$ has exactly two connected components. 
We denote by $\Gamma_{E, F}$ the connected component of $\Gamma' _{E,F}$ that contains $E$ as a vertex. 
\end{defi}

We state a graph-theoretic lemma, which will be used in the proof of Lemma \ref{lem:I}. 
\begin{lem}\label{NM-lem}
Let $\ell$ be a positive integer and $G$ a connected graph of order $n > \frac{1}{2}(3^{\ell}-1)$. 
Let $v$ be its vertex with degree at most three. 
If every vertex of $G$ has degree at most four, 
then the graph $G$ contains a chain of length $\ell$ containing $v$ with degree one.
\end{lem}

\begin{proof}
Let $G'$ be the subgraph of $G$ obtained by removing the vertex $v$ and all edges incident to $v$. 
Then $G'$ has at most three connected components, 
and hence some component $G''$ of $G'$ has order at least $\frac{n-1}{3} > \frac{1}{2}(3^{\ell - 1} -1)$. 
Let $v'$ be a vertex of $G''$ that is incident to $v$ in $G$. 
Then $v'$ has degree at most three in $G''$. 
Therefore the assertion follows from the induction on $\ell$. 
\end{proof}

The following lemma is a key of the proof of Theorem \ref{thm:ACC_2dim} and 
gives combinatorial information on the computing divisors. 

\begin{lem}\label{lem:I}
Let $n$ be an integer larger than one. Set $n' := \log _3 (2n+1) - 1$. 
Let $(X, B +M)/Z$ be a generalized lc surface and let $x \in X$ be a closed point. 
Suppose that 
\begin{itemize}
\item $0 < \operatorname{mld}_x (X, B + M) < 1$, and 
\item there exists a birational morphism $f : Y \to X$ as in Lemma \ref{lem:Y} such that 
the number of the $f$-exceptional divisors over $x$ is at least $n$. 
\end{itemize}
Let $\Gamma$ be the dual graph of $f$-exceptional divisors over $x$. 
Then there exists an edge path 
\[
E_0 {\text \ - \ } E_1 {\text \ - \ } \cdots {\text \ - \ } E_{m}
\]
of length $m \ge \frac{n'}{2}$ with the following three conditions.  
\begin{enumerate}
\item[(1)] $E_0$ is a computing divisor of $\operatorname{mld}_x (X, B+ M)$. 
\item[(2)] Either of the following two conditions holds: 
\begin{enumerate}
\item[(2-1)] $E_1$ is not a computing divisor. 
\item[(2-2)] 
\begin{itemize}
\item $E_1, \cdots , E_m$ are computing divisors, 
\item $\Gamma_{E_0, E_1}$ contains no computing divisor except for $E_0$, and 
\item $\Gamma_{E_0, E_1}$ has no fork. 
\end{itemize}
\end{enumerate}
\item[(3)] $0 \le a_{E_1}(X, B+M) - a_{E_0}(X, B+M) \le \frac{1}{m}$. 
\end{enumerate}
Moreover, if $\epsilon$ and $\delta$ are positive real numbers such that 
\begin{itemize}
\item $\min \{ \epsilon, \delta \} > \frac{16}{n'}$, 
\item $\operatorname{mld}_x (X, B+ M) \ge \epsilon$ and
\item $B \ge \delta$ and $M \ge \delta$, 
\end{itemize}
then the following additional three conditions hold.  
\begin{enumerate}
\item[(4)] $E_i$ is a $(-2)$-curve for each $1 \le i \le \frac{m}{2}$. 
\item[(5)] If $\delta \le \frac{2}{3}$, 
then the number of the vertices of $\Gamma_{E_0, E_1}$ is at most $\frac{16}{\min \{ \epsilon, \delta \}} + 1$. 
\item[(5')] If $E_0$ is a $(-1)$-curve, then $E_0$ is the unique vertex of $\Gamma_{E_0, E_1}$. 
\end{enumerate}
\end{lem}

\begin{proof}
First, we shall find an edge path $E_0, \ldots, E_m$ satisfying (1) and (2). 

Since $X$ is klt at $x$, the dual graph of the minimal resolution of $X$ over $x$ is a tree with at most one fork. 
Furthermore, if it has a fork, then the fork has exactly three branches (cf.\ \cite[Theorem 4.7]{KM98}). 
Therefore, by the construction of $Y$ in Lemma \ref{lem:Y}, the dual graph $\Gamma$ satisfies the following condition. 
\begin{itemize}
\item Each vertex of $\Gamma$ has at most four branches. 
Furthermore, the vertices corresponding to the computing divisors of $\operatorname{mld}_x (X, B + M)$ have at most three branches. 
\end{itemize}

Let $F_0$ be a computing divisor of $\operatorname{mld}_x (X, B + M)$. 
Then by the condition on $\Gamma$ above, we can take an edge path
\[
F_0 {\text \ - \ } F_1 {\text \ - \ } \cdots {\text \ - \ } F_{m'}
\]
of length $m'$ with $m' \ge n'$ by Lemma \ref{NM-lem}. 
Let $j$ be the maximum $j$ such that $F_j$ is a computing divisor. 
If $j < \frac{n'}{2}$, the chain 
\[
F_j {\text \ - \ } F_{j+1} {\text \ - \ } \cdots {\text \ - \ } F_{m'}
\]
satisfies (1) and (2-1), and it is sufficient to set $E_i := F_{j+i}$ for $0 \le i \le m' - j$. 

Suppose $j \ge \frac{n'}{2}$. 
Since $F_0$ and $F_j$ are computing divisors, 
all vertices of the chain 
\[
F_0 {\text \ - \ } F_{1} {\text \ - \ } \cdots {\text \ - \ } F_j
\]
are computing divisors by Lemma \ref{lem:mono}(1). 
Note that $f:Y \to X$ is the minimal resolution of $X$ in this case since there are at least two computing divisors $F_0$ and $F_1$. 
Since $\Gamma$ is a tree with at most one fork, 
one of the following conditions holds: 
\begin{itemize}
\item $\Gamma_{F_0, F_1}$ is a chain and $F_0$ has at most one branch in $\Gamma_{F_0, F_1}$. 
\item $\Gamma_{F_j, F_{j-1}}$ is a chain and $F_j$ has at most one branch in $\Gamma_{F_j, F_{j-1}}$.
\end{itemize}
Hence we may assume the former case, i.e.\ the graph $\Gamma_{F_0, F_1}$ is a chain: 
\[
F_{-m''+1} {\text \ - \ } F_{-m''+2} {\text \ - \ } \cdots {\text \ - \ } F_{0}, 
\]
where $m''$ is the number of the vertices of $\Gamma_{F_0, F_1}$. 
Let $k$ be the minimum $k \le 0$ such that $F_{k}$ is a computing divisor. 
Then $F_{k+1}, \ldots , F_{-1}$ are also computing divisors by Lemma \ref{lem:mono}(1). 
Hence the chain
\[
F_k {\text \ - \ } F_{k+1} {\text \ - \ } \cdots {\text \ - \ } F_j
\]
satisfies (1) and the three conditions in (2-2). 
Therefore, it is sufficient to set $E_i := F_{k+i}$ for $0 \le i \le j - k$. 

We have constructed an edge path $E_0, \ldots, E_m$ satisfying (1) and (2). 
In what follows, we shall prove that this path satisfies the other conditions (3)-(5)'. 

Set $a_i := a_{E_i} (X, B + M)$ for each $0 \le i \le m$. 
Note that 
\[
0 \le a_0 \le a_1 \le \cdots \le a_m
\] 
holds by Lemma \ref{lem:mono}(1). 
Furthermore, we have 
\[
a_m - a_0 \ge m (a_1 - a_0)
\]
by Lemma \ref{lem:mono}(1). 
Then (3) follows from the log canonicity $a_0 \ge 0$ and $a_m \le 1$ (cf.\ Lemma \ref{lem:Y}). 

We shall prove (4). Suppose the contrary that $- E_i^2 \ge 3$ for some $1 \le i \le \frac{m}{2}$. 
Then we have 
\[
a_{i+1} - a_i \ge a_i - a_{i-1} + \epsilon
\]
by Lemma \ref{lem:mono}(3), and hence
\[
a_m - a_i \ge (m -i) (a_i - a_{i-1} + \epsilon)
\]
by Lemma \ref{lem:mono}(1). 
Then it contradicts the following inequalities
\[
1 \ge a_m, \quad a_i \ge 0 , \quad m - i \ge \frac{m}{2} , 
\quad a_i - a_{i-1} + \epsilon \ge \epsilon > \frac{16}{n'} \ge \frac{8}{m}. 
\]

Next, we shall prove (5'). Suppose the contrary that $E_{-1}$ is an adjacent vertex of $E_0$ in $\Gamma_{E_0, E_1}$. 
Then by (4), there exists the following chain
\[
E_{-1} {\text \ - \ } E_0 {\text \ - \ } E_1 {\text \ - \ } \cdots {\text \ - \ } E_{m'}
\]
in $\Gamma$ with $m' > \frac{m}{2} - 1$ such that $E_1, \ldots, E_{m'}$ are $(-2)$-curves. 
Since $Y \to X$ factors through the blow-downs of the curves $E_0, \ldots, E_{m'}$, 
it follows that $- E_{-1}^2 \ge m' + 2$. 
On the other hand, we have 
$- E_{-1}^2 \le \frac{2}{\epsilon}$ by Lemma \ref{lem:mono}(2). It contradicts the following inequalities
\[
\frac{2}{\epsilon} \ge -E_{-1}^2 \ge m' + 2 > \frac{m}{2} \ge \frac{n'}{4} > \frac{4}{ \epsilon}. 
\]

We shall prove (5). By (5'), which has been already proved, we may assume that $- E_0 ^2 \ge 2$. 
Note that $f:Y \to X$ is the minimal resolution of $X$ in this case 
because a computing divisor $E_0$ is not a $(-1)$-curve. 
Therefore,  $\Gamma_{E_0, E_1}$ has at most one fork as discussed in the second paragraph. 
Suppose the contrary that $\ell > \frac{16}{\min \{ \epsilon, \delta \}}$ and $\Gamma_{E_0, E_1}$ consists of $\ell + 1$ vertices. 
Since $\Gamma_{E_0, E_1}$ has at most one fork, there exists a chain
\[
E_0 {\text \ - \ } E_{-1} {\text \ - \ } \cdots {\text \ - \ } E_{-m'}
\]
in $\Gamma_{E_0, E_1}$ with $m' \ge \frac{\ell}{2}$. 
Then the same argument as in (3) gives
\[
0 \le a_{-1} - a_{0} \le \frac{1}{m'}. 
\]
Note that we have the inequality
\begin{itemize}
\item[($\star$)] $(a_{-1} - a_0) + (a_1 - a_0) < \frac{1}{4}\min \{ \epsilon, \delta \}$
\end{itemize}
because 
\[
\frac{1}{m} + \frac{1}{m'} \le \frac{2}{n'} + \frac{2}{\ell} < 
\frac{1}{8}\min \{ \epsilon, \delta \} + \frac{1}{8}\min \{ \epsilon, \delta \} 
= \frac{1}{4}\min \{ \epsilon, \delta \}. 
\]

First, we claim that $- E_0 ^2 = -2$. By adjunction of $(X, B + M)$ to $E_0$, we have 
\[
2 + E_0 ^2 a_0 \ge  (1 - a_{-1}) + (1 - a_1). 
\]
Therefore it follows that
\[
(a_{-1} - a_0) + (a_1 - a_0) \ge (- E_0 ^2 - 2) a_0 \ge (- E_0 ^2 - 2) \epsilon, 
\]
and we have $- E^2 _0 = 2$ by ($\star$). 

We have the following two cases: 
\begin{itemize}
\item[(A)] $E_{-1}$ and $E_1$ are the only $f$-exceptional divisors which meet $E_0$. 
\item[(B)] $E_0$ meets three $f$-exceptional divisors $E_{-1}, E_1$ and $E'$. 
\end{itemize}

Suppose (A). Then by adjunction of $(X, B + M)$ to $E_0$, we have 
\[
(a_{-1} - a_0) + (a_1 - a_0) = (f^{-1}_* B + M_Y) \cdot E_0, 
\]
where $M_Y$ is the push-forward of $M'$ on $Y$ (possilby replacing $X'$ with a higher model). 
Here, we note that at least one of $E_{-1}$ and $E_1$ is not a computing divisor by (2). 
Therefore, the left-hand side is positive, and hence 
\[
(a_{-1} - a_0) + (a_1 - a_0) \ge \delta 
\]
by the assumptions $B \ge \delta$ and $M \ge \delta$. 
However, it contradicts ($\star$). 

Suppose (B). Set $a' := a_{E'}(X, B + M)$. 
Then by adjunction of $(X, B + M)$ to $E_0$, we have 
\[
(a_{-1} - a_0) + (a_1 - a_0) = (f^{-1}_* B + M_Y) \cdot E_0 + (1 - a'). 
\]
Hence we have 
\[
1 - a' \le (a_{-1} - a_0) + (a_1 - a_0) < \frac{1}{4}\min \{ \epsilon, \delta \}.
\] 
By adjunction to $E'$, we have
\[
2 - (- E^{'2}) a' \ge 1 - a_0 \ge 0. 
\]
Hence, if $- E^{'2} \ge 3$, we get a contradiction: 
\[
0 \le 2 - (- E^{'2}) a' \le 2(1-a') - a' < \frac{1}{2} \min \{ \epsilon, \delta \} - \epsilon < 0. 
\]
Since $- E^{'2} = 2$, we get 
\[
1 - a_0 \le 2 (1 - a') < \frac{1}{2} \min \{ \epsilon, \delta \} \le \frac{1}{2} \delta \le \frac{1}{3},  
\]
which contradicts Lemma \ref{lem:all-2}. 
The proof of (5) is complete. 
\end{proof}

\begin{thm}\label{thm:ACC_2dim}
For any DCC subset $I \subset [0,+ \infty)$, the set $A_{\textup{gen}}(2,I)$ satisfies the ACC. 
\end{thm}

\begin{proof} 
Let $I \subset [0,+ \infty)$ be a DCC subset. 
We may assume $\frac{2}{3} \in I$. 
Let $\delta$ be the minimum element of $I \cap (0,+ \infty)$. 
Then we have $\delta \le \frac{2}{3}$, which will be used when applying Lemma \ref{lem:I}(5). 
By Lemma \ref{lem:acc>1}, it is sufficient to show that 
$A_{\textup{gen}}(2,I) \cap [\epsilon, 1)$ satisfies the ACC for any $0 < \epsilon < 1$. 
Suppose the contrary that there exist a sequence of generalized lc surfaces 
$\bigl \{ (X_i, B_i + M_i)/Z_i \bigr \} _{i \ge 1}$ and 
closed points $x_i \in X_i$ with the following conditions. 
\begin{itemize}
\item $\epsilon \le a_0^{(i)} < 1$ holds for each $a_0^{(i)} := \operatorname{mld} _{x_i} (X_i, B_i + M_i)$. 
\item $\bigl( a_0^{(i)} \bigr)_{i \ge 1}$ is a strictly increasing sequence. 
\end{itemize}
We fix a resolution $f_i : Y_i \to X_i$ that satisfies the conditions in Lemma \ref{lem:Y}. 
Let $\Gamma^{(i)}$ be the dual graph of the $f_i$-exceptional divisors over $x_i$. 
Let $n^{(i)}$ be the number of the $f_i$-exceptional divisors over $x_i$. 
Then by Lemma \ref{lem:n_bounded}, we may assume 
\begin{itemize}
\item the sequence $\bigl( n^{(i)} \bigr)_{i \ge 1}$ is strictly increasing. 
\end{itemize}
We set $n^{\prime (i)} := \log _3 (2n^{(i)} + 1) - 1$. Then, 
\begin{itemize}
\item the sequence $\bigl( n^{\prime (i)} \bigr)_{i \ge 1}$ is also strictly increasing.
\end{itemize}
Hence, by passing to a subsequence, we may assume that $\min \{ \epsilon, \delta \} > \frac{16}{n^{\prime (i)}}$ for each $i \ge 1$.
Therefore by Lemma \ref{lem:I}, 
for each $i \ge 1$, there exist $f_i$-exceptional divisors $E_0^{(i)}$ and $E_1 ^{(i)}$ over $x_i$ with the following conditions: 
\begin{itemize}
\item $E_0^{(i)}$ is a computing divisor of $\operatorname{mld} _{x_i} (X_i, B_i + M_i)$. 
\item $E_0^{(i)}$ meets $E_1^{(i)}$. 
\item $0 \le a_1 ^{(i)} - a_0 ^{(i)} \le \frac{2}{n^{\prime (i)}}$, where we set $a_1^{(i)} := a_{E_1 ^{(i)}}(X_i, B_i + M_i)$. 
\item If $a_1 ^{(i)} - a_0 ^{(i)} = 0$, then the graph $\Gamma ^{(i)}_{E_0^{(i)}, E_1^{(i)}}$ is a chain. 
\item The number of the vertices of $\Gamma ^{(i)}_{E_0^{(i)}, E_1^{(i)}}$ is at most 
$\frac{16}{\min \{\epsilon, \delta\}} + 1$. 
\item If $E_0^{(i)}$ is a $(-1)$-curve, then $E_0^{(i)}$ is the unique vertex of $\Gamma ^{(i)}_{E_0^{(i)}, E_1^{(i)}}$. 
\end{itemize}
In particular, by the fifth condition and Lemma \ref{lem:mono}(2), 
the weighted dual graph $\Gamma ^{(i)}_{E_0^{(i)}, E_1^{(i)}}$ has finitely many possibilities depending only on $\epsilon$ and $I$. 

Let $Y_i \to \widetilde{Y}_i$ be the contraction of all the curves contained in $\Gamma ^{(i)}_{E_0^{(i)}, E_1^{(i)}}$. 
We denote by $g_i$ and $h_i$ the following induced morphisms:
\[
\xymatrix{
Y_i \ar[r]^{g_i} & \widetilde{Y}_i \ar[r]^{h_i} & X_i. 
}
\]
Then, since $E_1^{(i)}$ is the unique $h_i$-exceptional divisor that contains the center of $E_0^{(i)}$ 
on $\widetilde{Y}_i$, 
we have 
\begin{align*}
a_0^{(i)} 
&= a_{E_0^{(i)}}(X_i, B_i + M_i) \\
&= a_{E_0^{(i)}} \Bigl ( \widetilde{Y}_i, h_{i*} ^{-1} B_i + \bigl( 1 - a_1^{(i)} \bigr) E_1^{(i)} + M_{i, \widetilde{Y}_i} \Bigr ) \\
&= {a'}^{(i)} - c^{(i)} \bigl( 1 - a_1^{(i)} \bigr), 
\end{align*}
where $M_{i, \widetilde{Y}_i}$ is the push-forward of $M'_i$ on $\widetilde{Y}_i$ (possilby replacing $X'_i$ with a higher model), and we set 
\[
{a'}^{(i)} := a_{E_0^{(i)}} \bigl ( \widetilde{Y}_i, h_{i*} ^{-1} B_i + M_{i, \widetilde{Y}_i} \bigr ), 
\qquad c^{(i)} := \operatorname{coeff} _{E_0^{(i)}} \bigl( g_i ^* E_1^{(i)} \bigr). 
\]
Since the weighted dual graph $\Gamma ^{(i)}_{E_0^{(i)}, E_1^{(i)}}$ 
has finitely many possibilities depending only on $\epsilon$ and $I$, 
the number $c^{(i)}$ has also finitely many possibilities. 
Therefore, by passing to a subsequence, 
we may assume the existence of a positive real number $c$ such that $c^{(i)} = c$ for any $i \ge 1$. 

We claim that we may assume $c \le 1$. 
First, we have $c^{(i)} = 1$ when $E_0 ^{(i)}$ is a $(-1)$-curve. 
Indeed, $E_0^{(i)}$ is the unique vertex of $\Gamma ^{(i)}_{E_0^{(i)}, E_1^{(i)}}$ in this case. 
In what follows, we assume that $E_0 ^{(i)}$ is not a $(-1)$-curve, and hence $Y_i$ is the minimal resolution of $X_i$ (cf.\ Lemma \ref{lem:Y}). 
Since $Y_i$ is also the minimal resolution of $\widetilde{Y}_i$, we have ${a'}^{(i)} \le 1$. 
Hence we have 
\[
\frac{2}{n^{\prime (i)}} \ge a_1 ^{(i)} - a_0 ^{(i)} = a_1 ^{(i)} - {a'} ^{(i)} + c \bigl( 1 - a_1 ^{(i)} \bigr) 
\ge (c-1) \bigl( 1 - a_1 ^{(i)} \bigr). 
\]
Suppose that $c > 1$. Then we have 
\[
1- a_0 ^{(i)} = \bigl( 1 - a_1 ^{(i)} \bigr) + \bigl( a_1 ^{(i)} - a_0 ^{(i)} \bigr) \le \frac{c}{c-1} \cdot \frac{2}{n^{\prime (i)}}. 
\]
It contradicts Lemma \ref{lem:all-2} for $i$ sufficiently large. Therefore we may assume $c \le 1$ 
possibly passing to a tail of the sequence. 

Next, we claim that we may assume the following three conditions. 
\begin{itemize}
\item $({a'}^{(i)})_i$ is a non-increasing sequence.
\item $(a_1^{(i)})_i$ is an increasing sequence. 
\item The sequence $\bigl ( a_1 ^{(i)} - a_0 ^{(i)} \bigr )_{i}$ is a non-increasing sequence which converges to $0$. 
\end{itemize} 
First, we may assume the third condition by passing to a subsequence because we have
\[
0 \le a_1 ^{(i)} - a_0 ^{(i)} \le \frac{2}{n^{\prime (i)}}. 
\]
Next, we shall see the first condition. We have 
\[
{a'}^{(i)} = 1 + \operatorname{coeff}_{E_0^{(i)}} \bigl( K_{Y_i/ \widetilde{Y}_{i}} \bigr) 
- \operatorname{ord}_{E_0^{(i)}} \bigl ( h_{i*} ^{-1} B_i + M_{i, \widetilde{Y}_i} \bigr ). 
\]
Since the weighted dual graph $\Gamma ^{(i)}_{E_0^{(i)}, E_1^{(i)}}$ 
has finitely many possibilities depending only on $\epsilon$ and $I$, 
the following hold: 
\begin{itemize}
\item $\operatorname{coeff}_{E_0^{(i)}} \bigl( K_{Y_i/ \widetilde{Y}_{i}} \bigr)$ also has finitely many possibilities, and
\item $\operatorname{ord}_{E_0^{(i)}} \bigl ( h_{i*} ^{-1} B_i + M_{i, \widetilde{Y}_i} \bigr )$ 
is contained in a DCC set depending on $\epsilon$ and $I$. 
\end{itemize}
Therefore ${a'}^{(i)}$ is contained in an ACC set depending only on $\epsilon$ and $I$. 
Hence, by passing to a subsequence, we may assume that $\bigl( {a'}^{(i)} \bigr)_i$ is a non-increasing sequence. 
Furthermore, since we have 
\[
c \bigl( 1- a_1^{(i)} \bigr) =  {a'}^{(i)} - a_0^{(i)} 
\]
and $\bigl( a_0^{(i)} \bigr) _i$ is an increasing sequence, it follows that $\bigl( a_1^{(i)} \bigr)_i$ is also an increasing sequence. 

Recall that 
\[
\bigl( 1 - a_1 ^{(i)} \bigr) (1 - c) + \bigl( a_1 ^{(i)} - a_0 ^{(i)} \bigr) = 1 - {a'}^{(i)}. 
\]
Therefore by the three claims above, 
we have
\[
c = 1, \quad a_1 ^{(i)} - a_0 ^{(i)} = 0, \quad {a'}^{(i)} = 1
\]
for some $i \ge 1$. 
Since $a_1 ^{(i)} - a_0 ^{(i)} = 0$, the graph $\Gamma ^{(i)}_{E_0^{(i)}, E_1^{(i)}}$ is a chain. 
Furthermore, by ${a'}^{(i)} = 1$, all vertices of $\Gamma ^{(i)}_{E_0^{(i)}, E_1^{(i)}}$ are $(-2)$-curves. 
In this case, we have $c^{(i)} = \frac{\ell}{\ell + 1}$ if $\ell$ is the number of the vertices of 
$\Gamma ^{(i)}_{E_0^{(i)}, E_1^{(i)}}$. It contradicts $c=1$. 
\end{proof}

\section{ACC conjecture for varieties with bounded Gorenstein index}\label{section:ACCfixed}
\subsection{ACC conjecture for varieties with bounded Gorenstein index}

Hacon, McKernan, and Xu in \cite{HMX14} prove the ACC for log canonical thresholds. 
Birkar and Zhang in \cite{BZ16} generalize this result to generalized pairs.

\begin{thm}[{\cite[Theorem 1.5]{BZ16}}]\label{thm:LACC}
Let $d$ be a positive integer and let $I \subset [0, + \infty)$ be a DCC subset. 
Then there exists an ACC set $J \subset [0, + \infty)$ with the following condition: 
If $X, B, M, D$ and $N$ satisfy 
\begin{itemize}
\item $(X, B+M)/Z$ is a generalized pair with $\dim X = d$, 
\item $D$ and $N$ are $\mathbb{R}$-divisors on $X$ with the conditions in Definition \ref{defi:defs}(\ref{item:lct}), and 
\item $B, D \in I$ and $M, N \in I$ hold, 
\end{itemize}
then the generalized lc threshold of $D + N$ with respect to $(X, B+M)$ belongs to $J \cup \{ + \infty \}$. 
\end{thm}

\begin{thm}[{\cite[Theorem 1.6]{BZ16}}]\label{thm:GACC}
Let $d$ be a positive integer and let $I \subset [0, + \infty)$ be a DCC subset. 
Then there exists a finite subset $J \subset I$ with the following condition: 
If 
\begin{itemize}
\item $(X, B+M)/Z$ is a generalized lc pair with $\dim Z=0$ and $\dim X = d$,
\item $K_X + B + M \equiv 0$, 
\item $B \in I$, and 
\item $M'$ on $X'$ in Definition \ref{defi:gp} satisfies 
$M' = \sum_k m_k M'_k$ for some nef Cartier divisors $M'_k \not \equiv 0$ and some real numbers $m_k \in I$, 
\end{itemize}
then $B \in J$ and $M \in J$ hold. 
\end{thm}

The following theorem asserts that it is possible to perturb irrational coefficients preserving the log canonicity. 
The assertion for the usual pairs is proved in \cite[Theorem 1.6]{Nak16b} and \cite[Corollary 5.5]{HLS}
as an application of the rationality of
accumulation points of log canonical thresholds \cite[Theorem 1.11]{HMX14}. 
For generalized pairs, the same proof also works due to Theorems \ref{thm:LACC} and \ref{thm:GACC} 
(see \cite{C20} for a detailed discussion). 

\begin{thm}[{\cite[Theorem 3.15]{C20}}]\label{thm:UP}
Fix $d \in \mathbb{Z} _{>0}$. 
Let $r_1, \ldots, r_{\ell}$ be positive real numbers and let $r_0 = 1$. 
Assume that $r_0, r_1,  \ldots, r_{\ell}$ are $\mathbb{Q}$-linearly independent. 
Let $s^{B} _1, \ldots, s^{B} _{c^{B}} : \mathbb{R}^{\ell +1} \to \mathbb{R}$ (resp.\ 
$s^{M} _1, \ldots, s^{M} _{c^{M}} : \mathbb{R}^{\ell +1} \to \mathbb{R}$) be $\mathbb{Q}$-linear functions 
(that is, the extensions of $\mathbb{Q}$-linear functions from $\mathbb{Q} ^{\ell + 1}$ to $\mathbb{Q}$ 
by taking the tensor product $\otimes _{\mathbb{Q}} \mathbb{R}$). 
Assume that $s^{B}_i (r _0, \ldots , r_{\ell}) \ge 0$ and 
$s^{M}_j (r_0, \ldots , r_{\ell}) \ge 0$ for each $i$ and $j$. 
Then there exists a positive real number $\epsilon >0$ such that the following holds:
For any normal variety $X$ of dimension $d$,
a projective morphism $X\to Z$,
a projective birational morphism  $f:X' \to X$,
effective Weil divisors $B_1, \ldots, B_{c^{B}}$ on $X$, and 
nef$/Z$ Cartier divisors $M'_1, \ldots, M'_{c^{M}}$ on $X'$, 
if the generalized pair
\[
\Bigl ( X, \sum_{1 \le i \le c^{B}} s^B _i(r_0, \ldots, r _{\ell}) B_i + 
\sum_{1 \le i \le c^{M}} s^M _i(r_0, \ldots, r_{\ell}) M_i \Bigr )/Z
\] 
is generalized lc, then the pair 
\[
\Bigl ( X, \sum_{1 \le i \le c^{B}} s^B _i(r_0,  \ldots, r_{\ell -1}, t) B_i + 
\sum_{1 \le i \le c^{M}} s^M _i(r_0,  \ldots, r_{\ell -1}, t) M_i \Bigr )/Z
\] 
is also generalized lc for any 
$t \in [r_{\ell} - \epsilon, r_{\ell} + \epsilon]$. 
\end{thm}

The following notation will be used in Theorems \ref{thm:ACCfixed} and \ref{thm:perturbmld}. 
\begin{defi}\label{defi:P}
For $d,r \in \mathbb{Z} _{>0}$ and a finite set $I \subset [0, + \infty)$,  
we define $P(d,r,I)$ as the set of all generalized lc pairs $(X, B + M)/Z$ with the following four conditions: 
\begin{itemize}
\item $\dim X = d$. 
\item $r K_X$ is Cartier. 
\item $B = \sum _{i} b_i B_i$ for some $b_i \in I$ and effective Cartier divisors $B_i$. 
\item $M = \sum _i m_i M_i$ for some $m_i \in I$ and Cartier divisors $M_i$ such that $M_i = f_* M_i '$ for some 
projective birational morphism $f:X' \to X$ and nef$/Z$ Cartier divisors $M'_i$. 
\end{itemize}
\end{defi}

Theorem \ref{thm:ACCfixed} below can be proved by the induction on $\dim _{\mathbb{Q}} \operatorname{Span}_{\mathbb{Q}} (I \cup \{ 1 \})$. 
This idea is originally from Kawakita \cite{Kaw14}. 
We emphasize here that in Definition \ref{defi:P}, we assume a Cartier condition on $B_i$ and $M_i$. 
Because of it, we shall use the different notation $A'_{\textup{gen}}$ for such a special set of mld's. 
\begin{thm}\label{thm:ACCfixed}
Let $d,r \in \mathbb{Z} _{>0}$ and let $I \subset [0, + \infty)$ be a finite set. 
Let $P(d,r,I)$ be the set of generalized lc pairs defined in Definition \ref{defi:P}. 
Then the following set 
\[
B_{\textup{gen}}(d,r,I) := \left\{ a_E (X, B + M) \ \middle | 
\begin{array}{l}
\text{$(X,B+M)/Z \in P(d,r,I)$,} \\
\text{$E$ is a divisor over $X$.}
\end{array}
\right \}
\]
is a discrete subset of $[0,+\infty)$. 

In particular, the set
\[
A' _{\textup{gen}}(d,r,I) := \left\{ \operatorname{mld}_x (X, B + M) \ \middle | 
\begin{array}{l}
\text{$(X,B+M)/Z \in P(d,r,I)$, $x \in X$.}
\end{array}
\hspace{-1.5mm}
\right\}
\] 
is a discrete subset of $[0,+\infty)$. 
\end{thm}

\begin{proof}
The same proof of \cite[Theorem 1.2]{Nak16b} works due to Theorem \ref{thm:UP}. 
For readers' convenience, we give a proof below. 

We may assume $1 \in I$. 
Let $r_0 = 1, r_1 , \ldots , r_c$ be all elements of $I$. 
Set $c' + 1  := \dim _{\mathbb{Q}} \operatorname{Span}_{\mathbb{Q}} \{ 1, r_1, \ldots , r_c \}$. 
Possibly rearranging the indices, 
we may assume that $r_0 , \ldots , r_{c'}$ are $\mathbb{Q}$-linearly independent. 
For each $0 \le i \le c$, we may uniquely write $r_i = \sum_{0 \le j \le c'} q_{ij}r_j$ with $q_{ij} \in \mathbb{Q}$. 

We prove the discreteness of $B_{\textup{gen}}(d,r,I)$ by induction on $c'$. 
If $c' = 0$, we can take $n \in \mathbb{Z} _{>0}$ such that 
$I \subset \frac{1}{n} \mathbb{Z}$ and $\frac{1}{r} \in \frac{1}{n} \mathbb{Z}$. 
Then we have $B_{\textup{gen}}(d,r,I) \subset \frac{1}{n} \mathbb{Z}$ and $B_{\textup{gen}}(d,r,I)$ turns out to be discrete. 

Set $\mathbb{Q}$-linear functions $s_0, \ldots , s_c$ as follows: 
\[s_i : \mathbb{R} ^{c'+1} \to \mathbb{R}; \quad s_i (x_0, \ldots , x_{c'}) = \sum _{0 \le j \le c'} q_{ij}x_j. \]
Take $\epsilon > 0$ as in Theorem \ref{thm:UP}. 
We fix $t^+, t^- \in \mathbb{Q}$ satisfying 
\[
t^+ \in (r_{c'}, r_{c'} + \epsilon] \cap \mathbb{Q}, \quad 
t^- \in [r_{c'} - \epsilon, r_{c'}) \cap \mathbb{Q}. 
\]
We define $r^+ _0, \ldots, r^+ _c$ and $r^- _0, \ldots, r^- _c$ as
\[
r_i ^+ = s_i (r_0, \ldots, r_{c' -1}, t^+), \quad 
r_i ^- = s_i (r_0, \ldots, r_{c' -1}, t^-). 
\]
Furthermore, we set $I' := \{ r^+ _0, \ldots , r^+ _c, r^- _0, \ldots , r^- _c \}$. 
Then we have $\dim _{\mathbb{Q}} \operatorname{Span} _{\mathbb{Q}} (I') = c'$, 
and hence $B_{\textup{gen}}(d,r,I')$ is discrete by induction.

Let $\Bigl ( X, \sum _{0 \le i \le c} r_i B_i + \sum _{0 \le i \le c} r_i M_i \Bigr )/Z \in P(d,r,I)$ 
and let $E$ be a divisor over $X$. 
Since $\Bigl ( X, \sum _{0 \le i \le c} r_i B_i + \sum _{0 \le i \le c} r_i M_i \Bigr )$ is generalized lc, 
$\Bigl ( X, \sum _{0 \le i \le c} r_i ^* B_i  + \sum _{0 \le i \le c} r_i^*  M_i \Bigr )$ is also generalized lc for each $* \in \{ +, - \}$. 
Hence we have 
\begin{align*}
0 &\le a_E \Bigl ( X, \sum _{0 \le i \le c} r_i ^* B_i + \sum _{0 \le i \le c} r_i^*  M_i \Bigr ) \\ 
&= a_E \Bigl (X, \sum _{0 \le i \le c} r_i B_i  + \sum _{0 \le i \le c} r_i  M_i \Bigr ) - 
(t^* - r_{c'}) \sum _{0 \le i \le c} q_{ic'}\operatorname{ord} _E (B_i + M_i). 
\end{align*}
Therefore, either of the following holds:
\begin{itemize}
\item $0 \le \sum _{0 \le i \le c} q_{ic'}\operatorname{ord} _E (B_i + M_i) \le 
\epsilon _+ ^{-1} a_E \Bigl ( X, \sum _{0 \le i \le c} r_i B_i  + \sum _{0 \le i \le c} r_i  M_i \Bigr )$, or
\item $- \epsilon _- ^{-1} a_E \Bigl ( X, \sum _{0 \le i \le c} r_i B_i + \sum _{0 \le i \le c} r_i  M_i \Bigr )
\le \sum _{0 \le i \le c} q_{ic'}\operatorname{ord} _E (B_i + M_i) \le 0$, 
\end{itemize}
where we set $\epsilon _+ := t^+ - r_{c'}$ and $\epsilon _- := r_{c'} - t^-$.

It is sufficient to show the discreteness of $B_{\textup{gen}}(d,r,I) \cap [0,a]$ for any $a \in \mathbb{R} _{>0}$. 
Take $n \in \mathbb{Z} _{>0}$ such that $q_{ic'} \in \frac{1}{n} \mathbb{Z}$ for any $i$. 
Then, it is sufficient to prove that $B_{\textup{gen}}(d,r,I) \cap [0,a]$ is a subset of the following set
\begin{align*}
\Bigl \{ b + \epsilon_+ e  \ & \Big | \  b \in B_{\textup{gen}}(d,r,I'), 
e \in \frac{1}{n}\mathbb{Z} \cap[0, \epsilon _+ ^{-1} a]  \Bigr \} \\
&\cup 
\Bigl \{ b - \epsilon_- e  \ \Big | \  b \in B_{\textup{gen}}(d,r,I'), 
e \in \frac{1}{n}\mathbb{Z} \cap[- \epsilon _- ^{-1} a, 0]  \Bigr \}. 
\end{align*}
In fact, this set is discrete because $B_{\textup{gen}}(d,r,I')$ is discrete, and both
$\frac{1}{n}\mathbb{Z} \cap[0, \epsilon _+ ^{-1} a]$ and $\frac{1}{n}\mathbb{Z} \cap[- \epsilon _- ^{-1} a, 0]$ 
are finite. 

Let $\Bigl ( X, \sum _{0 \le i \le c} r_i B_i + \sum _{0 \le i \le c} r_i  M_i \Bigr )/Z \in P(d,r,I)$ 
and let $E$ be a divisor over $X$. 
Assume $a_E \Bigl ( X, \sum _{0 \le i \le c} r_i B_i  +  \sum _{0 \le i \le c} r_i  M_i \Bigr ) \in [0,a]$. 
Furthermore, suppose $\sum _{0 \le i \le c} q_{ic'}\operatorname{ord} _E (B_i + M_i) \ge 0$ 
(the same proof also works in the other case). 
Then, we have 
\begin{align*}
&a_E \Bigl ( X, \sum _{0 \le i \le c} r_i B_i + \sum _{0 \le i \le c} r_i  M_i \Bigr ) \\
&=
a_E \Bigl ( X, \sum _{0 \le i \le c} r_i^+ B_i  + \sum _{0 \le i \le c} r_i^+  M_i \Bigr )
+
(t^+ -r_{c'}) \sum _{0 \le i \le c} q_{ic'}\operatorname{ord} _E (B_i + M_i). 
\end{align*}
Here, we have 
\begin{itemize}
\item $a_E \Bigl ( X, \sum _{0 \le i \le c} r^+_i B_i  + \sum _{0 \le i \le c} r^+_i  M_i \Bigr ) \in B(d,r,I')$, 
\item $t^+ -r_{c'} = \epsilon _+$, and 
\item $\sum _{0 \le i \le c} q_{ic'}\operatorname{ord} _E (B_i + M_i)\in \frac{1}{n}\mathbb{Z} 
\cap[0, \epsilon _+ ^{-1} a]$. 
\end{itemize}
The proof is complete. 
\end{proof}

\begin{cor}\label{cor:ACC3dim}
Let $I \subset [0, + \infty)$ be a finite subset. 
Then the following set 
\[
A_{\textup{gen.can}}(3,I) := \left\{ \operatorname{mld}_x (X, B + M) \ \middle | 
\begin{array}{l}
\text{$(X,B+M)/Z$ is a generalized } \\
\text{canonical pair of $\dim X = 3$}\\
\text{with $B \in I$, $M \in I$, $x \in |X|_{0}$.}
\end{array}
\right \}
\]
satisfies the ACC. Furthermore, $1$ is the only accumulation point of this set. 
\end{cor}

\begin{proof}
The same proof of \cite[Corollary 1.5]{Nak16b} works due to Theorem \ref{thm:ACCfixed}. 
For readers' convenience, we give a proof below. 

Note that $A_{\textup{gen.can}}(3,I) \subset [1,3]$ holds (cf.\ \cite{Kaw92}, \cite{Mar96}). 
We shall prove that for any $a > 1$, the set
\[
A_{\textup{gen.can}}(3,I) \cap [a, + \infty)
\]
is a finite set. 

By the classification of three-dimensional $\mathbb{Q}$-factorial terminal singularities 
(see \cite{Kaw92} or \cite{Mar96}), 
the minimal log discrepancy of a three-dimensional terminal singularity is equal to $3$ or $1+1/r$ for some $r \in \mathbb{Z}_{>0}$. 
Furthermore, in the case when $\operatorname{mld} _x (X)=3$, the $X$ is known to be smooth at $x$. 
If $\operatorname{mld} _x (X)=1+1/r$, the Gorenstein index of $X$ at $x$ is known to be $r$. 
Furthermore, by \cite[Corollary 5.2]{Kaw88}, if $X$ has the Gorenstein index $r$ at $x \in X$, then 
$rD$ is Cartier at $x$ for any Weil divisor $D$. 

Let $(X, B + M)/Z$ be a three-dimensional generalized canonical pair satisfying 
$B \in I$, $M \in I$ and $\operatorname{mld} _x(X, B + M) \ge a$. 
By Theorem \ref{thm:extraction}, 
there exists a projective morphism $f : Y \to X$ with the following properties: 
\begin{itemize}
\item $Y$ is a $\mathbb{Q}$-factorial terminal variety. 
\item $f^* (K_X + B + M) = K_Y + B_Y + M_Y$ holds, where $B _Y$ is the strict transform of $B$ on $Y$, 
and $M_Y$ is the push-forward of $M'$ (possilby replacing $X'$ with a higher model). 
\end{itemize}
Take a computing divisor $E$ of $\operatorname{mld} _x(X, B + M)$. 

Suppose $\dim c_Y (E) = 0$. 
Then $\operatorname{mld} _x(X, B + M) = \operatorname{mld} _y (Y, B _Y + M_Y)$ holds, 
where $\{ y \} := c_Y (E)$. 
Since $\operatorname{mld} _y (Y) \ge \operatorname{mld} _y (Y, B _Y + M_Y) \ge a$ holds, 
the Gorenstein index of $Y$ at $y$ is at most $\bigl \lfloor \frac{1}{a-1} \bigr \rfloor$. 
Let $\ell$ be the Gorenstein index of $Y$ at $y$. 
Since $\ell D$ is Cartier at $y$ for any Weil divisor $D$ on $Y$, 
it follows that 
\[
\operatorname{mld} _y (Y, B _Y + M_Y) \in A' _{\textup{gen}} \Bigl ( 3, \ell, \frac{1}{\ell}I \Bigr ), 
\]
where we set $\frac{1}{\ell}I := \{ f \ell^{-1} \mid f \in I \}$. 
Therefore we have 
\[
\operatorname{mld} _x (X, B + M) \in 
\bigcup _{\ell \le \lfloor \frac{1}{a-1} \rfloor} A' _{\textup{gen}} \Bigl ( 3, \ell, \frac{1}{\ell}I \Bigr ), 
\]
and the right-hand side is a finite set by Theorem \ref{thm:ACCfixed}.

Suppose $\dim c_Y (E) = 1$. 
Then by Lemma \ref{lem:ambro}, we have
\[
\operatorname{mld} _y (Y, B _Y + M_Y) = 1 + \operatorname{mld} _x(X, B + M)
\]
for some closed point $y \in c_Y (E)$. 
Since $\operatorname{mld} _y (Y) \ge 1 + a > 2$, it follows that $Y$ is smooth at $y$. 
Hence, 
\[
\operatorname{mld} _y (Y, B _Y + M_Y) \in A' _{\textup{gen}}(3,1,I). 
\]
Therefore, we have
\[
\operatorname{mld} _x (X, B + M) \in -1 + A' _{\textup{gen}}(3,1,I), 
\]
and the right-hand side is a finite set by Theorem \ref{thm:ACCfixed}. 

Suppose $\dim c_Y (E) = 2$. 
Then $E$ is a divisor on $Y$, and we have
\[
\operatorname{mld} _x (X, B + M) = 1 - \operatorname{coeff} _E B _Y = 1. 
\]
It contradicts $\operatorname{mld} _x (X, B + M) \ge a > 1$. 
The proof is complete. 
\end{proof}

\subsection{Uniform perturbation of the $\epsilon$-log canonicity}\label{subsectionperturb}
In this subsection, we prove Theorem \ref{thm:perturbmld} as an application of Theorem \ref{thm:UP}. 

\begin{defi}
Let $S \subset \mathbb{R}$ be a $\mathbb{Q}$-linear subspace and 
let $f : S \to  \mathbb{R}$ be a $\mathbb{Q}$-linear function. 
Let $B$ be an $\mathbb{R}$-divisor on a variety $X$. 
When $B \in S$, we define an $\mathbb{R}$-divisor $f(B)$ on $X$ as follows: 
\[
f(B) := \sum _{D} f(\operatorname{coeff}_D B) D, 
\]
where $D$ is taken over all prime divisors $D$ on $X$.

\end{defi}

\begin{lem}\label{lem:f}
Let $S \subset \mathbb{R}$ be a $\mathbb{Q}$-linear subspace and 
let $f : S \to  \mathbb{R}$ be a $\mathbb{Q}$-linear function. 
Let $B$ be an $\mathbb{R}$-divisor on a variety $X$. 
\begin{enumerate}
\item Suppose that $B = \sum _{i = 1} ^{\ell} a_i D_i$ holds for some $a_i \in S$ and $\mathbb{Q}$-divisors $D_i$. 
Then it follows that $f(B) = \sum _{i = 1} ^{\ell} f(a_i) D_i$. 

\item Suppose that $B \in S$ holds and $B$ is $\mathbb{R}$-Cartier. 
Then $f(B)$ is also $\mathbb{R}$-Cartier. 

\item Suppose that $B \in S$ holds and $B$ is $\mathbb{R}$-Cartier. 
Then $\varphi ^* B \in S$ holds for any morphism $\varphi : Y \to X$. 
Furthermore, it follows that $\varphi^* (f(B)) = f( \varphi ^* B)$. 
\end{enumerate}
\end{lem}
\begin{proof}
(1) follows from the $\mathbb{Q}$-linearlity of $f$. 

Suppose that $B$ is an $\mathbb{R}$-Cartier $\mathbb{R}$-divisor with $B \in S$. 
Then we may write $B = \sum _{i=1} ^{\ell} a_i D_i$
with some $a_i \in S$ and $\mathbb{Q}$-divisors $D_i$. 
Moreover we may assume that $a_1, \ldots , a_{\ell}$ are $\mathbb{Q}$-linearly independent. 
Then by Lemma \ref{lem:RtoQ} below, each $D_i$ turns out to be $\mathbb{Q}$-Cartier. 
Therefore (2) and (3) follow from (1). 
\end{proof}

\begin{lem}\label{lem:RtoQ}
Let $D$ be an $\mathbb{R}$-divisor of the form 
$D = \sum _{i=1} ^{\ell} a_i D_i$ with 
$a_i \in \mathbb{R}$ and $\mathbb{Q}$-divisors $D_i$ on a variety $X$. 
If $D$ is $\mathbb{R}$-Cartier and $a_1, \ldots , a_{\ell}$ are $\mathbb{Q}$-linearly independent, 
then each $D_i$ is $\mathbb{Q}$-Cartier.
\end{lem}
\begin{proof}
Suppse that $D$ is $\mathbb{R}$-Cartier. Then we may write $D = \sum _{i=1} ^{s} b_i E_i$ 
with some $b_i \in \mathbb{R}$ and some $\mathbb{Q}$-Cartier $\mathbb{Q}$-divisors $E_i$. 
We may assume that $s \ge \ell$ and $a_i = b_i$ for $1 \le i \le \ell$ and 
further that $b_1 , \ldots , b_s$ are $\mathbb{Q}$-linear independent. 
Here, we allow $E_i = 0$ in this expression. 
Since $\sum _{i=1} ^{\ell} a_i D_i = \sum _{i=1} ^{s} b_i E_i$, 
it follows from the $\mathbb{Q}$-linear independence of $b_i$'s that $D_i = E_i$ holds for each $1 \le i \le \ell$. 
Therefore $D_i$'s turn out to be $\mathbb{Q}$-Cartier. 
\end{proof}

\begin{rmk}
Let $S \subset \mathbb{R}$ be a $\mathbb{Q}$-linear subspace and 
let $f : S \to  \mathbb{R}$ be a $\mathbb{Q}$-linear function. 
Let $(X, B + M)/Z$ be a generalized pair with $M \in S$. 
In this remark, we see that $f(M)$ is well-defined. 
Let $\varphi : X' \to X$ be the birational morphism 
and $M'$ the $\mathbb{R}$-Cartier divisor such that $\varphi _* M' = M$ in Definition \ref{defi:gp}. 
By definition of ``$M \in S$" (Definition \ref{defi:defs}(1)), we may write 
$M' = \sum _i r_i M'_i$ with $r_i \in S$ and Cartier divisors $M'_i$. 
Then we can define $f(M')$ and define $f(M) := \varphi _* \bigl( f(M') \bigr)$. 
By Lemma \ref{lem:f}, this definition does not depend on the choice of $X'$ and the expression $M' = \sum _i r_i M'_i$. 
\end{rmk}

\begin{lem}\label{lem:f(gen)}
Let $S \subset \mathbb{R}$ be a $\mathbb{Q}$-linear subspace containing $\mathbb{Q}$ and 
let $f : S \to  \mathbb{R}$ be a $\mathbb{Q}$-linear function fixing $\mathbb{Q}$. 
For a generalized pair $(X, B+M)/Z$ with $B, M \in S$, the following hold. 
\begin{enumerate}
\item $\big( X, f(B) + f(M) \big)$ is also a generalized pair if $f(B), f(M) \in \mathbb{R}_{\ge 0}$. 
\item $a_E (X, B + M) \in S$ holds for any divisor $E$ over $X$. 
\item $f\big( a_E (X, B + M) \big) = a_E \big( X,f(B) + f(M) \big)$ holds for any divisor $E$ over $X$. 
\end{enumerate}
\end{lem}
\begin{proof}
We have $f(K_X+B+M) = K_X + f(B) + f(M)$ since $f$ fixes $\mathbb{Q}$. 
Furthermore, it is $\mathbb{R}$-Cartier by Lemma \ref{lem:f}(2). We have proved (1). 

(2) follows from a similar argument as in the proof of Lemma \ref{lem:f}. 
(3) follows from Lemma \ref{lem:f}(3). 
\end{proof}

The following theorem is an immediate consequence of Theorem \ref{thm:UP}. 

\begin{thm}\label{thm:perturblc}
Let $d$ be a positive integer and let $I \subset [0,+\infty)$ be a finite subset. 
Then, there exists a positive real number $\delta$ depending only on $d$ and $I$ such that the following holds: 
If 
\begin{itemize}
\item $(X,B+M)/Z$ is a generalized lc pair with $\dim X = d$ and $B , M \in I$, and 
\item $g : \operatorname{Span}_{\mathbb{Q}} (I\cup \{1\}) \rightarrow \mathbb{R}$ is a $\mathbb{Q}$-linear function 
fixing $\mathbb{Q}$ and satisfies $|g(a)-a|\leq\delta$ for every $a \in I$, 
\end{itemize}
then $\bigl( X, g(B) +g(M) \bigr)$ is also generalized lc.
\end{thm}
\begin{proof}
Set $S = \operatorname{Span}_{\mathbb{Q}} (I\cup \{1\})$. 
We fix a $\mathbb{Q}$-linear basis $r_0 = 1, r_1, \ldots , r_{\ell}$ of $S$ that satisfies $r_1, \ldots, r_{\ell} \in I$.
We shall prove the stronger statement that is obtained by replacing the assumption on $g$ with the following weaker condition. 
\begin{itemize}
\item
$g : \operatorname{Span}_{\mathbb{Q}} (I\cup \{1\}) \rightarrow \mathbb{R}$ is a $\mathbb{Q}$-linear function 
fixing $\mathbb{Q}$ and satisfies $|g(r_i)-r_i| \leq \delta$ for each $r_i$. 
\end{itemize}
We prove the assertion by induction on $\dim_{\mathbb{Q}}S$. 
If $\dim_{\mathbb{Q}} S = 1$, then the function $g$ in the statement should be the identity map 
and there is nothing to show in this case. 

Suppose $\dim_{\mathbb{Q}}S >1$. 
For each element $a_j \in I$, we may write $a_j= s _j (r_0, \ldots, r _{\ell})$ 
for some $\mathbb{Q}$-linear function $s_j:\mathbb{R}^{\ell+1}\to\mathbb{R}$. 
By Theorem \ref{thm:UP}, there exists a positive real number $\epsilon$ such that the following holds: 
\begin{itemize}
\item[($\heartsuit$)] 
For any generalized lc pair of the form 
\[
\Bigl( X, \sum_j s _j(r_0, \ldots,r _{\ell})B_j + \sum_j s _j(r_0, \ldots, r _{\ell}) M_j \Bigr)/Z, 
\]
where $X$ is a normal variety of dimension $d$, 
$B_j$'s are effective Weil divisors and $M_j '$'s are nef$/Z$ Cartier divisors on $X'$, 
the pair 
\[
\Bigl( X, \sum_j s _j(r_0, \ldots,r _{\ell -1}, t)B_j + \sum_j s _j(r_0, \ldots, r _{\ell -1}, t) M_j \Bigr)
\]
is also generalized lc for any $t \in [r_{\ell} - \epsilon, r_{\ell} + \epsilon]$. 
\end{itemize}

We fix $t_1 \in [r _{\ell}-\epsilon, r _{\ell}) \cap \mathbb{Q}$
and $t_2 \in (r _{\ell}, r _{\ell} + \epsilon] \cap \mathbb{Q}$.  
We set
\[
I' := \{ s_j(r_0, \ldots, r_{\ell -1}, t_i) \mid i,j \}, \quad
S' := \operatorname{Span}_{\mathbb{Q}} (I' \cup \{ 1 \}) =  \operatorname{Span}_{\mathbb{Q}} \{ r_0, \ldots, r_{\ell -1} \}. 
\]
Then by the induction hypothesis, there exists a positive real number $\delta '$ satisfying the assertion for $I'$. 

In what follows, we shall prove that $\delta := \min \{ \delta', r_{\ell} - t_1, t_2 - r_{\ell} \}$ satisfies the assertion for $I$. 
Let $(X, B+M)/Z$ be a generalized pair and $g$ a function satisfying the conditions in the statement. 
We denote by $g_1$ and $g_2$ the $\mathbb{Q}$-linear functions $S \to S'$ defined by 
\[
g_i(r_j)=
\begin{cases}
r_j & \text{($0 \le j \le \ell -1$)}, \\
t_i & \text{($j = \ell$)}.
\end{cases}
\]
Then by ($\heartsuit$), it follows that $\bigl( X, g_1(B) + g_1(M) \bigr)$ and $\bigl( X, g_2(B) + g_2(M) \bigr)$ are generalized lc. 
Note that $g_i(B), g_i(M) \in I'$ holds for each $i = 1, 2$. 
Then by the choice of $\delta'$, it follows that 
$\bigl( X, g \circ g_1(B) + g \circ g_1(M) \bigr)$ and $\bigl( X, g \circ g_2(B) + g \circ g_2(M) \bigr)$ are generalized lc. 

Note that the real numbers $u_1$ and $u_2$ defined by
\[
u_1 := \frac{t_2-g(r_{\ell})}{t_2-t_1}, \qquad u_2 := \frac{g(r_{\ell})-t_1}{t_2-t_1}
\]
satisfy $u_1 + u_2 = 1$ and $g = u_1 (g \circ g_1) +u_2 (g \circ g_2)$. 
Indeed, it is easy to see that 
the function $u_1 (g \circ g_1) +u_2 (g \circ g_2)$ is a $\mathbb{Q}$-linear function satisfying 
$\bigl( u_1 (g \circ g_1) +u_2 (g \circ g_2) \bigr) (r_i) = g(r_i)$ for each $i$. 
Furthermore, $u_1, u_2 \ge 0$ follows from the inequality 
$\delta \le \min \{ r_{\ell} - t_1, t_2 - r_{\ell} \}$ and the assumption on $g$. 
Therefore $\bigl( X, g (B) + g (M) \bigr)$ is generalized lc. 
\end{proof}

\begin{lem}\label{lem:ext}
Let $\delta$ be a positive real number. 
Let $I \subset \mathbb{R}$ be a finite subset and 
let $I' \subset \operatorname{Span}_{\mathbb{Q}} I$ be a finite subset. 
Then there exists a positive real number $\delta'$ with the following condition: 
If $g: \operatorname{Span}_{\mathbb{Q}} I \to \mathbb{R}$ is a $\mathbb{Q}$-linear function satisfying 
$|g(a) - a| \le \delta'$ for any $a \in I$, then $|g(a) - a| \le \delta$ holds for any $a \in I'$. 
\end{lem}

\begin{proof}
The proof is straightforward. 
\end{proof}

\begin{lem}\label{lem:partition}
Let $I \subset \mathbb{R}$ be a finite subset and $\delta$ a positive real number. 
Then there exist finitely many positive real numbers $a_1, \ldots , a_k \in (0,1]$ and 
$\mathbb{Q}$-linear functions $f_1, \ldots, f_k: \operatorname{Span}_{\mathbb{Q}}(I \cup \{ 1 \}) \to \mathbb{Q}$ 
with the following conditions: 
\begin{itemize}
\item Each $f_i$ fixes $\mathbb{Q}$, 
\item $\sum_{i=1} ^k a_i f_i = \rm{id}$ 
(i.e.\ $\sum_{i=1} ^k a_i f_i (a) = a$ holds for any $a \in \operatorname{Span}_{\mathbb{Q}}(I \cup \{ 1 \})$), and
\item $| f_i(a) - a| \leq \delta$ holds for every $a \in I$ and $i$. 
\end{itemize}
\end{lem}
\begin{proof}
Set $S := \operatorname{Span}_{\mathbb{Q}}(I \cup \{ 1 \})$. 
Let $1, r_1, \ldots, r_n$ be a $\mathbb{Q}$-linear basis of $S$ satisfying $r_1 , \ldots , r_n \in I$. 
We may assume that $I = \{ r_1, \ldots, r_n \}$ by Lemma \ref{lem:ext}. 

For each $1 \le i \le n$, we take $q_{i1} \in [r_i - \delta, r_i) \cap \mathbb{Q}$ and $q_{i2} \in (r_i, r_i + \delta] \cap \mathbb{Q}$, and 
we set 
\[
u_{i1} := \frac{q_{i2} - r_i}{q_{i2} - q_{i1}}, \qquad 
u_{i2} :=\frac{r_i - q_{i1}}{q_{i2} - q_{i1}}. 
\]
Then we have
\[
u_{i1}, u_{i2} > 0, \qquad u_{i1} + u_{i2} = 1, \qquad u_{i1}q_{i1} + u_{i2}q_{i2} = r_i
\]
for each $1 \le i \le n$. 
We denote by $\{ 1, 2 \}^{\{ 1, \ldots , n \}}$ the set of maps $\{ 1, \ldots , n \} \to \{ 1, 2 \}$. 
For each $\sigma \in \{ 1, 2 \}^{\{ 1, \ldots , n \}}$, we define $a_{\sigma} \in (0,1]$ and 
$\mathbb{Q}$-linear functions $f_{\sigma}: S \to \mathbb{Q}$ as follows: 
\[
a_{\sigma} = \prod _{1 \le i \le n} u_{i \sigma(i)}, \qquad 
f_{\sigma}(1)=1, \quad f_{\sigma}(r_i) = q_{i \sigma(i)}. 
\]
Then it is easy to see that $f_{\sigma}$'s satisfy the first and the third conditions. 
Furthermore, we have 
\[
\Bigl( \sum _{\sigma \in \{ 1, 2 \}^{\{ 1, \ldots , n \}}} a_{\sigma} f_{\sigma} \Bigr) (1)
= \sum _{\sigma \in \{ 1, 2 \}^{\{ 1, \ldots , n \}}} \prod _{1 \le i \le n} u_{i \sigma (i)}
= \prod _{1 \le i \le n} (u_{i1} + u_{i2}) = 1, 
\]
and 
\begin{align*}
\Bigl( \sum _{\sigma \in \{ 1, 2 \}^{\{ 1, \ldots , n \}}} a_{\sigma} f_{\sigma} \Bigr) (r_j)
&= \sum _{\sigma \in \{ 1, 2 \}^{\{ 1, \ldots , n \}}} \Bigl( \prod _{1 \le i \le n} u_{i \sigma (i)} \Bigr) q_{j\sigma(j)} \\
&= \sum _{\sigma' \in \{ 1, 2 \}^{\{ 1, \ldots , n \} \setminus \{ j \}}} 
	(u_{j1}q_{j1} + u_{j2}q_{j2}) \Bigl( \prod _{i \not = j} u_{i \sigma' (i)} \Bigr) \\
&= (u_{j1}q_{j1} + u_{j2}q_{j2}) \prod _{i \not = j} (u_{i1} + u_{i2}) = r_j
\end{align*}
for each $1 \le j \le n$. 
Therefore $a_{\sigma}$'s and $f_{\sigma}$'s satisfy the second condition. 
The proof is complete. 
\end{proof}

\begin{rmk}
The positive real numbers $a_1, \ldots , a_k$ in Lemma \ref{lem:partition} satisfy $\sum _{i=1} ^k a_i = 1$
due to the first and the second conditions.
\end{rmk}

Now, we prove Theorem \ref{thm:perturbmld} as an application of Theorem \ref{thm:ACCfixed}. 
This gives a generalization of \cite[Lemma 3.4]{CH21} by removing the boundedness criterion on the minimal log discrepancy, 
which is assumed in \cite[Lemma 3.4]{CH21}.

\begin{thm}\label{thm:perturbmld}
Let $d,r \in \mathbb{Z} _{>0}$ and $\epsilon \in \mathbb{R}_{\ge 0}$, and let $I \subset [0, + \infty)$ be a finite set.
Let $P(d,r,I)$ be the set of generalized lc pairs defined in Definition \ref{defi:P}. 
Suppose $\epsilon \in \operatorname{Span}_{\mathbb{Q}} (I \cup \{ 1 \} )$. 
Then there exists a positive real number $\delta$ depending only on $d$, $r$, $\epsilon$ and $I$ such that the following holds: 
If
\begin{itemize}
 \item $(X, B + M)/Z \in P(d,r,I)$ and $x \in X$ is a scheme-theoretic point, 
 \item $\operatorname{mld}_x\bigl( X, B + M \bigr)\geq \epsilon$, 
 \item $f:\operatorname{Span}_{\mathbb{Q}} (I\cup \{1\})\rightarrow \mathbb{R}$ is a $\mathbb{Q}$-linear function fixing $\mathbb{Q}$, and
 \item $|f(a)-a|\leq\delta$ holds for every $a \in I$, 
\end{itemize}
then $\operatorname{mld}_x\bigl( X, f(B) + f(M) \bigr)\geq f(\epsilon)$. 
\end{thm}
\begin{proof}
We set $S := \operatorname{Span}_{\mathbb{Q}} (I \cup \{1\})$. 
By Lemma \ref{lem:ext}, we may replace the fourth condition in the statement with the following stronger condition: 
\begin{itemize}
\item[($\heartsuit$)] $|f(a)-a| \leq \delta$ holds for every $a \in I \cup \{ \epsilon \}$. 
\end{itemize}

Let $B_{\rm{gen}}(d,r,I)$ be the discrete set of log discrepancies obtained from $P(d,r,I)$ defined in Theorem \ref{thm:ACCfixed}. 
We set 
\[
\lambda := \min \bigl\{ a - \epsilon \ \big| \  a \in B_{\rm{gen}}(d,r,I) ,\ a > \epsilon \bigr\} > 0. 
\]
Let $\delta'$ be the $\delta$ given in Theorem \ref{thm:perturblc}. 
Then we define $\delta$ by
\[
\delta := \min \left\{ \delta ',  \frac{\delta' \lambda}{\epsilon + \lambda+\delta'} \right\} > 0. 
\]
In what follows, we shall see that this $\delta$ satisfies the statement. 

Let $(X,B+M)/Z$, $x$, and $f$ be as in the assumptions and ($\heartsuit$). 
Then the pair $\bigl( X, f(B) + f(M) \bigr)$ is also generalized lc by the choice of $\delta '$. 
By Lemma \ref{lem:f(gen)}, we have 
\[
a_E \bigl( X, f(B)+ f(M) \bigr) = f \bigl( a_E(X, B+M) \bigr) 
\]
for any divisor $E$ over $X$. 
Let $F$ be a divisor that satisfies $a_F(X, B+M) > \epsilon$ and $c_X(F) = \overline{\{ x \}}$. 
Then it is sufficient to show that 
\[
a_F \bigl( X, f(B) + f(M) \bigr) \ge f (\epsilon). 
\]

By the definition of $\lambda$, we have $a_F(X, B+M) \ge \epsilon + \lambda$. 
We define a $\mathbb{Q}$-linear function $g : S \to \mathbb{R}$ by $g := \frac{\delta '}{\delta} (f - \rm{id}) + \rm{id}$. 
Then $g$ fixes $\mathbb{Q}$ and satisfies $| g(a) - a | \le \delta '$ for any $a \in I$. 
Therefore the pair $\bigl( X, g(B) + g(M) \bigr)$ is generalized lc by the choice of $\delta '$, 
and hence $a_F \bigl( X, g(B) + g(M) \bigr) \ge 0$. 
Therefore we have 
\begin{align*}
a_F \bigl( X, f(B) + f(M) \bigr) 
&= a_F \Bigl( X, \frac{\delta}{\delta'}\bigl( g(B) + g(M) \bigr) + \Bigl( 1-\frac{\delta}{\delta'} \Bigr)(B+M) \Bigr) \\
&= \frac{\delta}{\delta'} a_F \bigl( X, g(B) + g(M) \bigr) + 
	\Bigl( 1-\frac{\delta}{\delta'} \Bigr) a_F(X, B+M) \\
&\ge \Bigl( 1-\frac{\delta}{\delta'} \Bigr)(\epsilon + \lambda) \\
&\ge \Bigl( 1-\frac{1}{\delta'} \cdot \frac{\delta' \lambda}{\epsilon + \lambda+\delta'} \Bigr)(\epsilon + \lambda) \\
&= \epsilon + \frac{\delta ' \lambda}{\epsilon + \lambda + \delta '} \\
&\ge \epsilon +  \delta \\
&\ge f(\epsilon). 
\end{align*}
Here, we used the assumption $(\heartsuit)$ for the last inequality. 
This completes the proof. 
\end{proof}

\subsection{Application to complement}\label{subsection:compl}
Based on Birkar's result (Theorem \ref{thm:Bircomp}), 
Filipazzi and Moraga in \cite[Theorem 1.2]{FM} show the existence of bounded strong complements 
for generalized pairs with coefficients in a closed rational DCC set. 
Moreover, in \cite[Theorem 1.3]{FM}, they show the existence of bounded $\epsilon$-complements for such pairs and any $\epsilon>0$ 
when $\dim Z=0$ and $M'=0$, using the BAB theorem.
Using Theorem \ref{thm:UP}, G.\ Chen and Q.\ Xue removed the rationality condition on the coefficient set 
if the coefficient set is bounded ({\cite[Theorem 1.3]{C20}} and {\cite[Theorem 1.3]{CX20}}). 

In this subsection, 
we first prove Theorem \ref{m0}, which allows us to reduce the problem to the case when the coefficient set is bounded.
As a corollary (Corollary \ref{cor1}), we remove the rationality condition from the statement of \cite[Theorem 1.2]{FM}. 
We also give a proof of {\cite[Theorem 1.3]{CX20}} in Corollary \ref{cor2}.

We begin with the definition of the complements for generalized pairs.
\begin{defi}[{\cite[2.18]{Bir19a}}, {\cite[Definition 3.17]{HLS}}]\label{comp}
Let $(X,B+M)/Z$ be a generalized pair with $B \in [0,1]$. 
\begin{enumerate}
\item 
Let $n$ be a positive integer.
An \textit{(generalized) $n$-complement} of $K_X+B+M$ over a point $z\in Z$ is 
of the form $K_X+B^++M$ with the following conditions over a neighborhood of $z$: 
\begin{itemize}
\item $(X,B^++M)$ is generalized lc,
\item $n(K_X+B^++M)\sim 0$ and $nM'$ is Cartier, and
\item $n B^+ \geq n \lfloor B \rfloor + \left \lfloor(n+1) (B - \lfloor B \rfloor) \right \rfloor$. 
\end{itemize}
For a non-negative real number $\epsilon$, 
an $n$-complement $K_X+B^+ + M$ is called \textit{$\epsilon$-lc} (resp.\ \textit{klt})
if the corresponding generalized pair $(X,B^+ + M)$ is generalized $\epsilon$-lc (resp.\ klt). 

\item 
An \textit{(generalized) $\mathbb{R}$-complement} 
of $K_X+B+M$ over a point $z \in Z$ is of the form $K_X+B^++M^+$ with the following conditions over a neighborhood of $z$:
\begin{itemize}
\item $(X,B^++M^+)$ is generalized lc,
\item $K_X+B^++M^+\sim_{\mathbb{R}}0$, 
\item $B^+\geq B$ and $M^{+ \prime}-M'$ is nef over $Z$.
\end{itemize}
We also define an \textit{$\epsilon$-lc $\mathbb{R}$-complement} similarly. 

\item
An \textit{$n$-complement} (resp.\ \textit{$\mathbb{R}$-complement}) of $K_X+B+M$ is of the form $K_X + B^+ + M^+$ 
such that $K_X + B^+ + M^+$ is an $n$-complement (resp.\ $\mathbb{R}$-complement) of $K_X+B+M$ over $z$ for every $z \in Z$.
We say that $(X, B+M)$ is \textit{$\mathbb{R}$-complementary} if there exits an $\mathbb{R}$-complement of $K_X+B+M$.
We say that $(X, B+M)$ is \textit{$(\epsilon,\mathbb{R})$-complementary} if there exits an $\epsilon$-lc $\mathbb{R}$-complement of $K_X+B+M$.
\end{enumerate}
\end{defi}

\begin{rmk}\label{rmk:compl}
We take over the notations in Definition \ref{comp}.
\begin{enumerate}
\item An $n$-complement is called \textit{strong} if $B^+ \geq B$ holds. 
In this subsection, we will consider only strong complements. 
We remark that if $n(K_X+B^++M) \sim 0$ and $B^+ \geq B$ hold, and $nM'$ is Cartier, 
then the condition $n B^+ \geq n \lfloor B \rfloor + \left \lfloor(n+1) (B - \lfloor B \rfloor) \right \rfloor$ 
automatically holds (see {\cite [6.1(1)]{Bir19a}}).

\item 
Let $(X_1, B_1 + M_1)$ and $(X_2, B_2 + M_2)$ be generalized pairs with 
a birational map $X_1 \dasharrow X_2$. 
Let $\varphi _1: X' \to X_1$ and $\varphi _2 : X' \to X_2$ be the projective birational morphisms 
in Definition \ref{defi:gp} (possilby replacing $X'$ with a higher model). 
Suppose that 
\[
\varphi_1 ^*(K_{X_1} + B_1 + M_1)+P= \varphi_2^* (K_{X_2} + B_2 + M_2)
\]
holds for some $P \geq 0$. 
In this case, if $K_{X_2} + B_2 + M_2$ has a strong $\epsilon$-lc $n$-complement (resp.\ $\mathbb{R}$-complement), 
then so does $K_{X_1} + B_1 + M_1$ (see \cite[6.1(2)(3)]{Bir19a}).
\end{enumerate}
\end{rmk}

\noindent
Birkar proved in \cite{Bir19a} the boundedness of complements for generalized pairs. 
\begin{thm}[{\cite[Theorem 1.10]{Bir19a}}]\label{thm:Bircomp}
Let $d$ and $p$ be positive integers. 
Let $I \subset [0,1] \cap \mathbb{Q}$ be a finite subset. 
Then there exists a positive integer $n$ depending only on $d$, $p$ and $I$ such that 
if 
\begin{itemize}
\item $(X,B+M)/Z$ is a generalized lc pair with $\dim X=d$, 
\item $\dim Z=0$,
\item $B \in I$ holds and $pM'$ is Cartier,
\item $X$ is of Fano type, and
\item $-(K_X+B+M)$ is nef, 
\end{itemize}
then there exists a strong $n$-complement $K_X+B^++M$ of $K_X+B+M$.
\end{thm}

\noindent
We recall the following theorems by Filipazzi and Moraga. 

\begin{thm}[{\cite [Theorem 1.2]{FM}}]\label{thm:FM18c}
Let $d$ and $p$ be positive integers. 
Let $I \subset [0,1] \cap \mathbb{Q}$ be a finite subset. 
Then there exists a positive integer $n$ depending only on $d$, $p$ and $I$ 
such that if 
\begin{itemize}
\item $(X,B+M)/Z$ is a generalized lc pair with $\dim X=d$, 
\item $B \in I$ holds and $pM'$ is Cartier, 
\item $X$ is of Fano type over $Z$, and 
\item $-(K_X+B+M)$ is nef$/Z$, 
\end{itemize}
then for every point $z \in Z$, there exists a strong $n$-complement $K_X+B^++M$ of $K_X+B+M$ over $z$.
\end{thm}

\begin{thm}[{\cite [Theorem 1.3]{FM}}]\label{thm:FM18b}
Let $d$ and $p$ be positive integers, and let $\epsilon$ be a positive real number. 
Let $I \subset [0,1] \cap \mathbb{Q}$ be a finite subset. 
Then there exists a positive integer $n$ depending only on $d$, $p$, $\epsilon$ and $I$ 
such that if 
\begin{itemize}
\item $(X,B+M)/Z$ is a generalized $\epsilon$-lc pair with $\dim X = d$, 
\item $\dim Z=0$, 
\item $B \in I$ holds and $pM'$ is Cartier, 
\item $X$ is of Fano type, and 
\item $-(K_X+B+M)$ is nef, 
\end{itemize}
then there exists a strong $\epsilon$-lc $n$-complement $K_X+B^++M$ of $K_X+B+M$.
\end{thm}

\begin{thm}[{cf.\ \cite [Lemma 5.13, Theorem 5.15] {HLS}}]\label{thm:peff}
Fix $d \in \mathbb{Z} _{>0}$. 
Let $r_1, \ldots, r_{\ell}$ be positive real numbers and let $r_0 = 1$. 
Assume that $r_0, r_1,  \ldots, r_{\ell}$ are $\mathbb{Q}$-linearly independent. 
Let $s^{B} _1, \ldots, s^{B} _{c^{B}} : \mathbb{R}^{\ell +1} \to \mathbb{R}$ (resp.\ 
$s^{M} _1, \ldots, s^{M} _{c^{M}} : \mathbb{R}^{\ell +1} \to \mathbb{R}$) be $\mathbb{Q}$-linear functions.
Assume that $s^{B}_i (r _0, \ldots , r_{\ell}) \ge 0$ and 
$s^{M}_j (r_0, \ldots , r_{\ell}) \ge 0$ for each $i$ and $j$. 
Then there exists a positive real number $\epsilon >0$ such that the following holds:
If $(X,B+M)/Z$ is a generalized pair with $\dim X=d$ such that
\begin{itemize}
\item $B=\sum_{1 \le i \le c^{B}} s^B _i(r_0, \ldots, r _{\ell}) B_i$ for some Weil divisors $B_i$ on $X$,
\item $M'=\sum_{1 \le j \le c^{M}} s^M_j(r_0, \ldots, r_{\ell}) M'_j$ for some nef$/Z$ Cartier divisors $M'_j$ on $X'$, and
\item $(X, B+M)$ is $\mathbb{R}$-complementary,
\end{itemize}
then for every $t\in [r_{\ell}-\epsilon, r_{\ell}+\epsilon]$, 
$\bigl( X,B(t)+M(t) \bigr)$ is generalized lc and $- \bigl( K_X+B(t)+M(t) \bigr)$ is pseudo-effective over $Z$, where we set
\[
B(t) := \sum_{1 \le i \le c^{B}} s^B _i(r_0,  \ldots, r_{\ell -1}, t) B_i, \quad 
M'(t) := \sum_{1 \le j \le c^{M}} s^M _j(r_0,  \ldots, r_{\ell -1}, t) M'_j.
\]
\end{thm}
\begin{proof}
Enlarging the set $\bigl \{s^B_i \ \big| \  i \bigr \}$, we may assume that $s^B _i(r_0, \ldots, r _{\ell})=r_0=1$ for some $i$.
We may also remove the zero functions from $\bigl \{s^B_i , s^M_j \ \big| \  i,j \bigr \}$.
Let $\epsilon'$ be the $\epsilon$ appearing in Theorem \ref{thm:UP} and let $u_0$ be a rational number in $[r_{\ell}-\epsilon', r_{\ell}+\epsilon']$.

Suppose the contrary that for each $k \ge 1$,
there exist a real number $t_k \in \bigl[ r_{\ell}-\frac{\epsilon'}{k}, r_{\ell}+\frac{\epsilon'}{k} \bigr]$ 
and a generalized pair $( X_k, B_k + M_k )/Z_k$ with $\dim X_k =d$ such that
\begin{itemize}
\item $B_k(t) = \sum_{1 \le i \le c^{B}} s^B _i(r_0, \ldots, r _{\ell-1},t) B_{k,i}$ for some Weil divisors $B_{k,i}$ on $X_k$,
\item $M'_k(t) = \sum_{1 \le i \le c^{M}} s^M_i(r_0, \ldots, r_{\ell -1},t) M'_{k,i}$ for some nef$/Z_k$ Cartier divisors $M'_{k,i}$ on $X'_k$, 
\item $B_k = B_k(r_{\ell})$ and $M_k = M_k(r_{\ell})$, 
\item $( X_k, B_k + M_k )$ is $\mathbb{R}$-complementary, and
\item $- \bigl( K_{X_k} + B_k( t_k) + M_k( t_k) \bigr)$ is not pseudo-effective over $Z_k$.
\end{itemize}
Moreover, restricting $X_k$ over a neighborhood of a point $z \in Z_k$ at which 
$- \bigl( K_{X_k} + B_k( t_k) + M_k( t_k) \bigr)$ is not pseudo-effective over $Z_k$, 
we may assume that $K_{X_k} + B_k + M_k$ has an $\mathbb{R}$-complement $K_{X_k}+B_k^++M_k^+$ satisfying $K_{X_k}+B_k^++M_k^+\sim_{\mathbb{R}} 0$.

In the following paragraphs, we fix $k \in \mathbb{Z}_{> 0}$ and denote $(X, B+M)/Z=(X_k,B_k+M_k)/Z_k$ for simplicity. 

\vspace{2mm}

\noindent
\underline{\textbf{STEP 1:}}\ \ 
Let $K_X + B^+ + M^+ \sim_{\mathbb{R}} 0$ be an $\mathbb{R}$-complement of $K_X+B+M$. 
We set 
\[
D = B^+-B, \qquad N' = M^{+ \prime} - M'. 
\]
Let $N$ be the push-forward of $N'$ on $X$. 
This step shows that we may assume the following conditions: 
\begin{itemize}
\item $(X,B^++M^+)/Z$ is a $\mathbb{Q}$-factorial generalized dlt pair. 
\item $\operatorname{Supp}D$ and $\operatorname{Supp}\lfloor B^+ \rfloor$ have no common divisors. 
\item $\operatorname{Supp}(B - B(t_k))$ and $\operatorname{Supp}\lfloor B^+ \rfloor$ have no common divisors.
\end{itemize} 

Let $f: Y \to X$ be a $\mathbb{Q}$-factorial generalized dlt model of $(X,B^++M^+)$ 
(Definition \ref{defi:defs}(\ref{item:dltmodel})). 
Let $T$ be the sum of the $f$-exceptional divisors, 
let $B_Y (t)$ and $D_Y$ be the strict transforms of $B (t)$ and $D$ on $Y$, and
let $M_Y (t)$ and $N_Y$ be the push-forwards of $M' (t)$ and $N'$ on $Y$ (possibly replacing $X'$ by a higher model).
Then 
\[
\bigl( Y,  (T + B_Y ( r _{\ell})+D_Y )+ (M_Y ( r _{\ell})+N_Y) \bigr)/Z
\]
is a $\mathbb{Q}$-factorial generalized dlt pair, and we have 
\begin{align*}
K_Y &+ \bigl( T + B_Y ( r _{\ell})+D_Y \bigr) + \bigl( M_Y ( r _{\ell})+N_Y \bigr) \\
& = f^* \bigl( K_X + (B+D) + (M+N) \bigr) \sim_{\mathbb{R}} 0. 
\end{align*}

We may write 
\[
B_Y(t) = \sum s^B _i(r_0, \ldots, r _{\ell-1},t) B_{Y,i},
\]
where $B_{Y,i}$ is the strict transform of $B_i$ on $Y$. 
We define a Weil divisor $\overline{B}_i$ on $Y$ as 
\[
\operatorname{coeff}_G \overline{B}_i = 
\begin{cases}
0 & \text{if $G \le \lfloor B_Y ( r _{\ell})+D_Y \rfloor$}, \\
\operatorname{coeff}_G B_{Y,i} & \text{otherwise}, 
\end{cases}
\]
where $G$ is a prime divisor on $Y$. 
Then we set 
\[
B'_Y(t) := \sum s^B _i(r_0, \ldots, r _{\ell-1},t) \overline{B}_i. 
\]
Similarly, write $D=\sum_i d_iD_i$, where $D_i$'s are distinct prime divisors, 
and let $D'_Y := \sum d_k \overline{D}_k$, 
where $\overline{D}_k$ is the strict transform of $D_k$ on $Y$ and the sum is taken over all the $k$ 
satisfying $\overline{D}_k \not \leq \lfloor B_Y ( r _{\ell})+D_Y \rfloor$. 
Then the following hold by the construction: 
\begin{itemize}
\item $\lfloor B_Y ( r _{\ell})+D_Y \rfloor+B'_Y ( r _{\ell})+D'_Y=B_Y ( r _{\ell})+D_Y$.

\item $\operatorname{Supp} D'_Y$ and 
$\operatorname{Supp} \bigl( T+\lfloor B_Y ( r _{\ell})+D_Y \rfloor \bigr)$ have no common divisors. 

\item $\operatorname{Supp} \bigl( B'_Y(r_{\ell}) - B'_Y(t_k) \bigr)$ and 
$\operatorname{Supp} \bigl( T+\lfloor B_Y ( r _{\ell})+D_Y \rfloor \bigr)$ have no common divisors. 
\end{itemize}
We shall see that we may assume 
\begin{itemize}
\item $- \bigl( K_Y +T+\lfloor B_Y ( r _{\ell})+D_Y \rfloor+B'_Y(t_k)+M_Y(t_k) \bigr)$ is not pseudo-effective over $Z$. 
\end{itemize}
Note that $- \bigl( K_X + B(t_k)+ M(t_k) \bigr)$ is not pseudo-effective over $Z$ by the assumption. 
Since 
\[
f_*\bigl( K_Y +T+B_Y(t_k)+M_Y(t_k) \bigr) = K_X + B(t_k)+ M(t_k), 
\]
$- \bigl( K_Y +T+B_Y(t_k)+M_Y(t_k) \bigr)$ is not pseudo-effective over $Z$. 
On the other hand, 
for every $i$, if $s^B _i(r_0, \ldots, r _{\ell})=1$, 
then $s^B _i(r_0,  \ldots, r_{\ell -1}, t)$ should be constantly one by the linear independence of $r_0, r_1,  \ldots, r_{\ell}$. 
Therefore, passing to a tail of the sequence, we may assume that $s^B _i(r_0, \ldots, r _{\ell-1}, t_k) \leq 1$. 
Hence we have
\[
\lfloor B_Y ( r _{\ell})+D_Y \rfloor+B'_Y(t_k)-B_Y(t_k)\geq 0, 
\]
and therefore $- \bigl( K_Y +T+\lfloor B_Y ( r _{\ell})+D_Y \rfloor+B'_Y(t_k)+M_Y(t_k) \bigr)$ is not pseudo-effective over $Z$. 

Therefore, we may replace 
\begin{itemize}
\item $\bigl( X,B(t)+M(t) \bigr)$ with $\bigl( Y,(T+\lfloor B_Y ( r _{\ell})+D_Y \rfloor+B'_Y(t))+M_Y(t) \bigr)$,
\item $D$ with $D'_Y$, $N$ with $N_Y$, and 
\item $B^++M^+$ with $(T +  B_Y ( r _{\ell})+D_Y ) + (M_Y ( r _{\ell})+N_Y)$, 
\end{itemize}
and the new pair satisifies the desired three conditions.

\vspace{2mm}

\noindent
\underline{\textbf{STEP 2:}}\ \ 
By the linear independence of $r_0, r_1,  \ldots, r_{\ell}$, we have 
\[
s^B _i(r_0, \ldots, r _{\ell}) > 0, \quad 
s^M _j(r_0, \ldots, r _{\ell}) > 0.  
\]
Hence by STEP 1, we may take $\delta_k > 0$ such that
\[
\bigl( X, \bigl( B^++\delta_k D+\delta_k(B-B(t_k)) \bigr) + \bigl( M^+ + \delta_k N + \delta_k (M- M(t_k)) \bigr) \bigr)
\]
is generalized dlt. 
Since
\begin{align*}
& - \delta_k \bigl( K_X + B( t_k) + M( t_k) \bigr) \\
&\sim_{\mathbb{R}} \delta_k \bigl( (B-B(t_k)+D)+(M-M(t_k)+N) \bigr)\\
&\sim_{\mathbb{R}}K_X + \bigl( B^++\delta_k D+\delta_k(B-B(t_k)) \bigr) 
+ \bigl( M^+ + \delta_k N + \delta_k(M-M(t_k)) \bigr),
\end{align*}
there exists a $- \bigl( K_X + B( t_k) + M( t_k) \bigr)$-MMP over $Z$ (Theorem \ref{thm:MFS}) and it ends with 
a Mori fiber space $\overline{X} \to W$ such that 
$K_{\overline{X} }+ B_{\overline{X}}( t_k) + M_{\overline{X}}( t_k)$ is ample over $W$, 
where $B_{\overline{X}}(t)$ is the strict transform of $B(t)$ on $\overline{X}$, and 
$M_{\overline{X}}(t)$ is the push-forward of $M'(t)$ on $\overline{X}$ (possibly replacing $X'$ with a higher model).
On the other hand, since $-(K_X+B+M)$ is pseudo-effective over $Z$, 
$- \bigl( K_{\overline{X}}+B_{\overline{X}}(r_{\ell}) + M_{\overline{X}}(r_{\ell}) \bigr)$ is nef over $W$.
Since $(X,B+M)$ is $\mathbb{R}$-complementary, 
$\bigl( \overline{X}, B_{\overline{X}}(r_{\ell}) + M_{\overline{X}}(r_{\ell}) \bigr)$ is generalized lc.
Therefore, $\bigl( \overline{X},B_{\overline{X}}(t_k)+M_{\overline{X}}(t_k) \bigr)$ is also generalized lc 
since $t_k \in [r_{\ell}-\epsilon', r_{\ell}+\epsilon']$.

Let $F$ be a general fiber of ${\overline{X}}\to W$.
We set 
\[
B_F(t) = B_{\overline{X}}(t)|_F, \qquad M_F(t)=M_{\overline{X}}(t)|_F.
\]
Then $\bigl( F,B_F(r_{\ell})+M_F(r_{\ell}) \bigr)$ and $\bigl( F,B_F(t_k)+M_F(t_k) \bigr)$ are generalized lc, 
$- \bigl( K_F+B_F(r_{\ell})+M_F(r_{\ell}) \bigr)$ is nef and $K_F+B_F(t_k)+M_F(t_k)$ is ample.
Therefore, we may find a real number $u_k$ between $t_k$ and 
$r_{\ell}$ satisfying $K_F + B_{F}(u_k)+M_{F}(u_k) \equiv 0$. 
Note that $\bigl( F,B_F(u_k)+M_F(u_k) \bigr)$ is also generalized lc. 

We may write 
\begin{align*}
s^B _i(r_0, \ldots, r_{\ell -1}, t) &= s^B _i(r_0, \ldots, r_{\ell -1}, u_0)+q_i (t-u_0), \\
s^M _j(r_0, \ldots, r_{\ell -1}, t) &= s^M _j(r_0, \ldots, r_{\ell -1}, u_0)+p_j (t-u_0)
\end{align*}
with $q_i, p_j \in \mathbb{Q}$.
Let 
\[
I := \bigl \{s^B _i(r_0, \ldots, r_{\ell -1}, u_0),s^M _j(r_0, \ldots, r_{\ell -1}, u_0) \ \big| \  i, j \bigr\}.
\]
Take a positive integer $m$ such that $mq_i, mp_j \in \mathbb{Z}$ holds for every $i$ and $j$.
Since $K_{F}+B_{F}(t_k)+M_{F}(t_k) \not \equiv 0$ and $\lim_{k\to\infty}u_k=r_{\ell}$, 
we have
\[
\frac{r_{\ell}-u_0}{m} = \lim_{k \to \infty} \frac{u_k -u_0}{m} \in
\operatorname{Span} _{\mathbb{Q}} \bigl( I \cup \{ 1 \} \bigr) \subset \operatorname{Span} _{\mathbb{Q}} (r_0, \ldots, r_{\ell -1})
\]
by {\cite[Theorem 3.6]{C20}}. 
This contradicts the $\mathbb{Q}$-linearly independence of $r_0, \ldots , r_{\ell}$.
\end{proof}

\begin{thm}\label{thm:perturbpeff}
Let $d$ be a positive integer and let $I \subset [0,+\infty)$ be a finite subset. 
Then, there exists a positive real number $\delta$ depending only on $d$ and $I$ such that the following holds: 
If 
\begin{itemize}
\item $(X,B+M)/Z$ is an $\mathbb{R}$-complementary generalized pair of dimension $d$ with $B , M \in I$, and 
\item $g : \operatorname{Span}_{\mathbb{Q}} (I\cup \{1\}) \rightarrow \mathbb{R}$ is a $\mathbb{Q}$-linear function 
fixing $\mathbb{Q}$ and satisfies $|g(a)-a|\leq\delta$ for every $a \in I$, 
\end{itemize}
then $\bigl( X, g(B) +g(M) \bigr)$ is generalized lc and $- \bigl( K_X+g(B)+g(M) \bigr)$ is pseudo-effective over $Z$.
\end{thm}
\begin{proof}
The same proof as in Theorem \ref{thm:perturblc} works due to Theorem \ref{thm:peff}. 
\end{proof}

The following lemma follows from the BAB Theorem ({\cite [Theorem 1.1]{Bir16b}}).

\begin{lem}\label{lem:bab}
Let $d$ be a positive integer and let $\epsilon$ be a positive real number.
Then the projective $\mathbb{Q}$-factorial varieties $X$ of dimension $d$ such that 
\begin{itemize}
\item $(X,B+M)/Z$ is generalized $\epsilon$-lc for some generalized boundary $B+M$ with
\item $\dim Z=0$,
\item $K_X + B + M \sim_{\mathbb{R}} 0$, and
\item $X$ is of Fano type
\end{itemize}
form a bounded family.
\end{lem}
\begin{proof}
Let $(X,B+M)$ be as in the assumption.
By {\cite [Corollary 1.4]{Bir16b}}, it is enough to show that $(X,\Delta)$ is $\epsilon$-lc and 
$K_X+\Delta\sim_{\mathbb{R}}0$ for some (usual) boundary $\Delta$.

Let $X \dasharrow Y$ be the end result of a $(-K_X)$-MMP. 
Since $-K_Y$ is semi-ample and $Y$ is $\epsilon$-lc, there exists a boundary $\Delta _Y$ such that 
$(Y, \Delta _Y)$ is $\epsilon$-lc and $K_Y + \Delta _Y \sim _{\mathbb{R}} 0$. 
Therefore, there exists a boundary $\Delta$ on $X$ such that $(X, \Delta)$ is $\epsilon$-lc and $K_X+\Delta\sim_{\mathbb{R}}0$. 
The proof is complete. 
\end{proof}

We then define the $(n, \Gamma)$-decomposable complements for generalized pairs, modifying {\cite[Definition 1.9]{HLS}}.

\begin{defi} [cf.\ {\cite[Definition 1.9]{HLS}}]
Let $(X,B+M)/Z$ be a generalized pair. 
Let $\Gamma = (a_1, \ldots , a_k)$ be a finite partition of one 
(i.e.\ a finite sequence $a_1, \ldots , a_k$ of non-negative real nunmbers 
with $\sum _{i =1} ^k a_i = 1$). 
\textit{An (generalized) $(n,\Gamma)$-decomposable complement} of $K_X + B + M$
is of the form $K_X+B^++M^+$ such that
\begin{itemize}
\item $(X,B^++M^+)$ is an $\mathbb{R}$-complement of $K_X+B+M$, 
\item $B^+ = \sum_{i=1}^k a_i B^+_i$ and $M^{+ \prime} = \sum_{i=1}^k a_i M^{+ \prime}_i$ hold for some generalized boundaries $B^+_i+M^+_i$ of $X$, and 
\item $K_X+B^+_i+M^+_i$ is an $n$-complement of $K_X+B^+_i + M^+_i$ itself for each $i=1,2,\ldots,k$.
\end{itemize}

For a non-negative real number $\epsilon$, an $(n,\Gamma)$-decomposable complement $K_X+B^++M^+$ is called \textit{$(\epsilon, n, \Gamma)$-decomposable complement} 
if $(X, B^++M^+)$ is generalized $\epsilon$-lc. 
\end{defi}

We shall show the existence of bounded strong complements for generalized pairs with coefficients in a finite set. 
This is a special case of {\cite[Theorem 1.3]{C20}} and {\cite[Theorem 1.5]{CX20}}. 
We give a proof using the perturbation technique introduced in Subsection \ref{subsectionperturb}. 

\begin{thm}\label{thm:bdd-e-comple}
Let $d$ be a positive integer, $\epsilon$ a non-negative real number, 
and $I \subset [0, + \infty)$ a finite subset. 
Then there exist a positive integer $n$ and a partition $\Gamma = (a_1, \ldots , a_k)$ of one depending only on $d$, $\epsilon$ and $I$ 
such that the following condition holds: If 
\begin{itemize}
\item $(X,B+M)/Z$ is a generalized $\epsilon$-lc pair of dimension $d$, 
\item $\dim Z=0$ if $\epsilon > 0$,
\item $B, M \in I$ holds, 
\item $X$ is of Fano type over $Z$, and 
\item $(X,B+M)$ is $(\epsilon,\mathbb{R})$-complementary over $z$ for every $z \in Z$, 
\end{itemize}
then for every $z \in Z$, there exists a strong generalized $(\epsilon,n,\Gamma)$-decomposable complement $K_X+B^++M$ of $K_X+B+M$ over $z$. 
\end{thm}
\begin{proof}
Replacing $I$ with $I \cup \{ \epsilon \}$, we may assume that $\epsilon \in I$. 
The proof is divided into the cases where $\epsilon >0$ and where $\epsilon =0$. 

\noindent
\underline{\textbf{Case 1}}\ \ Suppose $\epsilon > 0$. 
In this case, $Z$ in consideration satisfies $\dim Z=0$ by assumption.

First, we claim that there exists a positive integer $r$ depending only on $d$ and $\epsilon$ such that 
\begin{itemize}
\item If $(Y , D + N)/Z$ is a $\mathbb{Q}$-factorial generalized pair with the following conditions: 
\begin{itemize}
\item $\dim Z = 0$ and $\dim Y = d$, 
\item $D, N \in I$, 
\item $Y$ is of Fano type, and
\item $(Y, D+N)$ has an $\epsilon$-lc $\mathbb{R}$-complement, 
\end{itemize}
then it follows that $(Y, D + N) \in P \bigl( d,r,\frac{1}{r}I \bigr)$, 
where the notion $P \bigl( d,r,\frac{1}{r}I \bigr)$ is defined in Definition \ref{defi:P}. 
\end{itemize}
Indeed, by Lemma \ref{lem:bab}, such $Y$ belongs to a bounded family depending on $d$ and $\epsilon$. 
Hence, there exists a positive integer $r$ depending only on $d$ and $\epsilon$ such that $r F$ is Cartier for any Weil divisor $F$ on $Y$ 
(cf.\ \cite[Lemma 1.12]{Kaw88}). 
Therefore we have $(Y, D + N) \in P \bigl( d,r,\frac{1}{r}I \bigr)$. 

We set $\delta := \min \{ \delta_1, \delta _2 \}$, where 
$\delta _1$ is the $\delta$ given in Theorem \ref{thm:perturbmld} applied to $d$, $r$, $\epsilon$ and $I' := \frac{1}{r}I$, 
and $\delta _2$ is the $\delta$ given in Theorem \ref{thm:perturbpeff}. 
By Lemma \ref{lem:partition}, 
there exist a partition $\Gamma = (a_1, \ldots, a_k)$ of one and 
finitely many $\mathbb{Q}$-linear functions 
$f_1, \ldots , f_k : \operatorname{Span}_{\mathbb{Q}} (I \cup \{ 1 \} ) \to \mathbb{Q}$ fixing $\mathbb{Q}$ such that 
\begin{itemize}
\item $\sum_{i=1}^ka_i f_i = \rm{id}$ and 
\item $|f_i(a)-a| \leq \delta$ holds for every $a \in I$ and $i \in \{ 1, \ldots , k \}$. 
\end{itemize}

Let $(X,B+M)$ be a generalized pair satisfying the conditions in the statement. 
Note that strong complements are preserved by taking push-forwards of small birational maps. 
Therefore we may assume that $X$ is $\mathbb{Q}$-factorial by replacing $X$ with its small $\mathbb{Q}$-factorialization.
Then by the choice of $\delta _2$, it follows that $- \bigl( K_X + f_i(B) + f_i(M) \bigr)$ is pseudo-effective. 
Let $X \dasharrow Y_i$ be the end result of a $- \bigl( K_X + f_i(B) + f_i(M) \bigr)$-MMP. 
We denote by $B_i$ and $M_i$ the pushforward of $B$ and $M$ on $Y_i$, respectively. 
Since $(X,B+M)$ is $(\epsilon,\mathbb{R})$-complementary by assumption, so is $(Y_i, B_i + M_i)$.
In particular, $(Y_i, B_i + M_i)$ is generalized $\epsilon$-lc and $(Y_i, B_i + M_i) \in P(d,r,I')$. 
Hence by the choice of $\delta _1$, $\bigl( Y_i, f_i(B_i) + f_i(M_i) \bigr)$ is generalized $f_i(\epsilon)$-lc. 
Therefore by Theorem \ref{thm:FM18b} and Remark \ref{rmk:compl}(2), 
there exists a positive integer $n$ depending only on $d$, $\epsilon$ and $I$ such that 
there is a generalized strong $f_i (\epsilon)$-lc 
$n$-complement $K_{X} + B_i ^+ + f_i(M)$ of $K_X + f_i(B) + f_i(M)$ for each $i$. 
We set $B^+ = \sum_{i=1}^ka_iB_i^+$. 
Then we have 
\begin{align*}
\operatorname{mld} ( X, B^+ + M ) 
&= \operatorname{mld} \Bigl( X, \sum_{i=1}^ka_iB_i^+ +\sum_{i=1}^k a_i f_i(M) \Bigr) \\
&\geq  \sum_{i=1}^ka_i \operatorname{mld} \bigl( X, B_i^+ +f_i( M) \bigr)\\
&\ge \sum_{i=1}^k a_i f_i (\epsilon) = \epsilon. 
\end{align*}
Therefore $K_X + B^+ +M$ is a strong generalized $(\epsilon,n,\Gamma)$-decomposable complement of $K_X+B+M$. 

\vspace{2mm}

\noindent
\underline{\textbf{Case 2}}\ \ Suppose $\epsilon = 0$. 

The same argument as in Case 1 works by replacing Theorem \ref{thm:perturbmld} with Theorem \ref{thm:perturblc}, 
and Theorem \ref{thm:FM18b} with Theorem \ref{thm:FM18c} in the argument. 
\end{proof}

\begin{thm}\label{m0}
Let $d$ be a positive integer and $\Lambda \subset [0, + \infty)$ a DCC set. 
Then there exists a positive real number $m'$ depending only on $d$ and $\Lambda$ such that 
the following condition holds: If 
\begin{itemize}
\item $(X,B+M)/Z$ is a generalized pair of dimension $d$, 
\item $X$ is of Fano type over $Z$,
\item $(X,B+M)$ is $\mathbb{R}$-complementary, 
\item $B \in \Lambda$, and
\item $M' = \sum_j m_j M'_j$ holds for some $m_j \in \Lambda$ and some nef$/Z$ Cartier divisors $M'_j \not \equiv _{Z} 0$, 
\end{itemize}
then $m_j\leq m'$ holds for each $j$.
\end{thm}
\begin{proof}
Replacing $X$ with its small $\mathbb{Q}$-factorial modification, 
we may assume that $X$ is $\mathbb{Q}$-factorial. 

First, we shall prove the boundedness of $m_j$ for $j$ satisfying $M'_j \not =\varphi^*(M_j)$. 
Let $\Gamma$ be the ACC set $J$ appearing in Theorem \ref{thm:LACC} when applying it to $I = \Lambda$. 
For each $j$ with $M'_j \not =\varphi^*(M_j)$, 
the generalized lc threshold $t$ of $M_j$ with respect to 
$\bigl( X , B + (M - m_jM_j) \bigr)$ satisfies $m_j \leq t < + \infty$ and $t \in \Gamma$. 
In particular, $m_j$ is bounded from above by the maximum element $\max \Gamma$ of the ACC set $\Gamma$.

Next, we shall prove the boundedness of $m_j$ for $j$ satisfying $M'_j = \varphi^*(M_j)$. 
Suppose the contrary that for each $i \ge 1$, there exists a generalized lc pair $(X_i, B_i+M_i)/Z_i$ of dimension $d$ with the following conditions: 
\begin{itemize}
\item $(X_i, B_i + M_i)/Z_i$ satisfies the assumptions in the statement. 

\item $M'_{i,0} = \varphi^*(M_{i,0})$ and $m_{i,0} \ge i$ hold 
when we write $M'_i = \sum_{j \ge 0} m_{i,j} M'_{i,j}$ as in the assumption. 
\end{itemize}
Since $M_{i,0} \not \equiv _{Z_i} 0$ by the assumption, 
$-(K_{X_i}+B_i+M_i+t_iM_{i,0})$ is not pseudo-effective over $Z_i$ for some $t_i>0$. 
Since $X_i$ is of Fano type over $Z_i$, there exists a $-(K_{X_i} + B_i + M_i + t_i M_{i,0})$-MMP over $Z_i$ and 
it ends with a Mori fiber space $Y_i \to W_i$. 
Then 
\[
K_{Y_i}+B_{i,Y_i}+M_{i,Y_i}+t_iM_{i,0,Y_i}
\] 
is ample over $W_i$, where $B_{i,Y_i}$, $M_{i,Y_i}$, and $M_{i,0,Y_i}$ are 
the strict transforms of $B_i$, $M_i$, and $M_{i,0}$, respectively.
Since $(X_i, B_i+M_i)$ is $\mathbb{R}$-complementary, $(Y_i, B_{i,Y_i}+M_{i,Y_i})$ is generalized lc. 
Possibly replacing $X'_i$ with a higher model, we may assume that $\psi: X'_i \to Y_i$ is a morphism. 
Then for $i > \max\Gamma$, we have $\psi^*(M_{i,0,Y_i})=M'_{i,0}$ by the same argument in the previous paragraph. 
Hence $(Y_i,B_{i,Y_i}+M_{i,Y_i}+t_iM_{i,0,Y_i})$ is also generalized lc.
Let $F_i$ be a general fiber of $Y_i \to W_i$.
We set 
\[
B_{F_i} := B_{i,Y_i}|_{F_i}, \quad M_{F_i} := M_{i,Y_i}|_{F_i}, \quad M_{i,0,F_i} := M_{i,0,Y_i}|_{F_i}.
\]
Then $(F_i,B_{F_i} + M_{F_i} + t_iM_{i,0,F_i})$ is generalized lc, 
$-(K_{F_i}+B_{F_i}+M_{F_i})$ is nef and $K_{F_i}+B_{F_i}+M_{F_i}+t_iM_{i,0,F_i}$ is ample. 
Therefore, there exists $\lambda _i \in [0, t_i)$ such that 
\begin{itemize}
\item $K_{F_i} + B_{F_i} + M_{F_i} + \lambda_i M_{i,0,F_i} \equiv 0$, and 
\item $(F_i, B_{F_i} + M_{F_i} + \lambda_i M_{i,0,F_i})$ is generalized lc. 
\end{itemize}
Note that $B_{F_i} \in \Lambda$ and $M_{F_i} + \lambda_i M_{i,0,F_i} \in \Lambda \cup \{ m_{i,0} + \lambda _i \mid i \}$, 
and that $m_{i,0} + \lambda _i \ge i$. 
Therefore, we get a contradiction by Theorem \ref{thm:GACC}. 
\end{proof}

\begin{cor}\label{cor1}
Fix a positive integer $d$ and a DCC subset $\Lambda \subset [0, + \infty)$ of real numbers. 
Then there exist a positive integer $n$ and a finite partition $\Gamma$ of one depending only on $d$ and $\Lambda$
such that the following condition holds: If 
\begin{itemize}
\item $(X,B+M)/Z$ is a generalized lc pair of dimension $d$, 
\item $B, M \in \Lambda$ holds, 
\item $X$ is of Fano type over $Z$, and
\item $(X,B+M)$ is $\mathbb{R}$-complementary over $z$ for every $z \in Z$, 
\end{itemize}
then for every point $z\in Z$, there exists a strong $(n,\Gamma)$-decomposable complement $K_X+B^++M^+$ of $K_X+B+M$ over $z$.
\end{cor}
\begin{proof}
Replacing $X$ with a small $\mathbb{Q}$-factorial modification, we may assume that $X$ is $\mathbb{Q}$-factorial. 
Since $M \in \Lambda$, we may write 
\[
M' = \sum _j m_j M'_j + \sum_h n_h N'_h
\]
with $m_j, n_h \in \Lambda$ and some nef$/Z$ Cartier divisors $M'_j$ and $N'_h$ on $X'$, 
where each $M'_j$ satisfies $M'_j \not \equiv _Z 0$ and each $N'_h$ satisfies $N'_h \equiv _Z 0$.
Let $M_j$ and $N_h$ be the push-forwards of $M'_j$ and $N'_h$ on $X$, respectively.
By Theorem \ref{m0}, there exists a positive real number $m'$ such that $m_j \leq m'$ holds for every $j$.
On the other hand, by the negativity lemma, $N'_h = \varphi^*(N_h)$ holds for every $h$. 
Since $N_h$ is a numerically trivial$/Z$ $\mathbb{Q}$-Cartier divisor on a Fano type variety $X$ over $Z$, 
it follows that $N_h \sim _{\mathbb{Q}, Z} 0$ by the base point free theorem (cf.\ \cite[Theorem 3.1.1]{KMM}). 
In particular, for every point $z \in Z$, $N_h \sim _{\mathbb{Q}} 0$ holds over a neighborhood of $z$. 
By assumption, for every $z$, there exists an $\mathbb{R}$-complement $K_X+B^++M^+$ of $K_X+B+M$ over $z$. 
Therefore, $K_X+B^++M^++\sum_h(\lceil n_h\rceil-n_h) N_h$ is an $\mathbb{R}$-complement of 
$K_X+B+M+\sum_h(\lceil n_h\rceil-n_h) N_h$ over $z$. 
Hence, $\bigl( X,B+\sum _j m_j M_j +\sum_h\lceil n_h\rceil N_h \bigr)$ is $\mathbb{R}$-complementary over $z$. 
Then, the assertion follows from the case proved in {\cite[Theorem 1.3]{C20}} where $\Lambda$ is bounded.
\end{proof}

When $M'=0$, the boundedness of complements are shown in \cite[Theorem 1.5]{CX20}. 
We give a brief proof here using Theorem \ref{thm:bdd-e-comple}. 

\begin{cor}[{\cite[Theorem 1.5]{CX20}}]\label{cor2}
Fix a positive integer $d$, a DCC subset $\Lambda \subset [0,1]$ of real numbers, and a non-negative real number $\epsilon$.
Then there exist a positive integer $n$ and a finite partition $\Gamma$ of one depending only on $d$, $\Lambda$ and $\epsilon$ 
such that the following condition holds: If 
\begin{itemize}
\item $(X,B)$ is a projective pair of dimension $d$, 
\item $B \in \Lambda$ holds, 
\item $X$ is of Fano type, and
\item $(X,B)$ is $(\epsilon, \mathbb{R})$-complementary, 
\end{itemize}
then there exists a strong $(\epsilon , n , \Gamma)$-decomposable complement $K_X+B^+$ of $K_X+B$.
\end{cor}
\begin{proof}
When $\epsilon=0$, the assertion is the special case (i.e.\ $M'=0$) of Corollary \ref{cor1}.
Therefore, we may assume $\epsilon>0$. 
Replacing $\Lambda$ with its closure, we may assume that $\Lambda$ is closed.
Replacing $X$ with its small $\mathbb{Q}$-factorial modification, we may assume that $X$ is $\mathbb{Q}$-factorial.

By the proof of {\cite[Theorem 1.3]{FM}}, there exists a finite subset $\Lambda_0 \subset \Lambda$ 
depending only on $d$, $\Lambda$ and $\epsilon$ with the following condition: 
\begin{itemize}
\item[$(\spadesuit)$]
If $(X,B)$ is a pair satisfying the assumptions in the statement, 
then there exists a boundary $B^+$ of $X$ such that 
$B^+\geq B$, $B^+\in \Lambda_0$ and $(X, B^+)$ is $(\epsilon,\mathbb{R})$-complementary.
\end{itemize}
\noindent
Then the assertion follows from Theorem \ref{thm:bdd-e-comple}. 
Note that in {\cite[Theorem 1.3]{FM}}, the rationality condition $\overline{\Lambda} \subset \mathbb{Q}$ is assumed.
However, for proving $(\spadesuit)$, their proof works even for our setting $\Lambda = \overline{\Lambda} \subset \mathbb{R}$.
\end{proof}

\begin{rmk}\label{rmk:overlap}
\begin{enumerate}
\item Some results on the boundedness of complements are proved in {\cite{CH21}} for surfaces (not necessarily of Fano type).
\item Corollary \ref{cor1} is proved in {\cite[Theorem 1.3]{C20}} for $\Lambda \subset [0,1]$.
\item Corollary \ref{cor2} is proved in \cite[Theorem 1.5]{CX20}. Corollary \ref{cor2} is still open for generalized pairs in general.
\end{enumerate}
\end{rmk}


\begin{bibdiv}
\begin{biblist*}

\bib{Ale93}{article}{
   author={Alexeev, Valery},
   title={Two two-dimensional terminations},
   journal={Duke Math. J.},
   volume={69},
   date={1993},
   number={3},
   pages={527--545},
}


\bib{Amb99}{article}{
   author={Ambro, Florin},
   title={On minimal log discrepancies},
   journal={Math. Res. Lett.},
   volume={6},
   date={1999},
   number={5-6},
   pages={573--580},
}

\bib{Amb06}{article}{
   author={Ambro, Florin},
   title={The set of toric minimal log discrepancies},
   journal={Cent. Eur. J. Math.},
   volume={4},
   date={2006},
   number={3},
   pages={358--370 (electronic)},
}

\bib{BCHM}
{article}{
   author={Birkar, Caucher},
   author={Cascini,Paolo},
   author={Hacon, Christopher D.},
   author={McKernan, James},
   title={Existence of minimal models for varieties of log general type},
   journal={J. Amer. Math. Soc.},
   volume={23},
   date={2010},
   number={2},
   pages={405--468},
}


\bib{Bir19a}{article}{
   author={Birkar, Caucher},
   title={Anti-pluricanonical systems on Fano varieties},
   journal={Ann. of Math. (2)},
   volume={190},
   date={2019},
   number={2},
   pages={345--463},
}

\bib{Bir16b}{article}{
   author={Birkar, Caucher},
   title={Singularities of linear systems and boundedness of Fano varieties},
   journal={Ann. of Math. (2)},
   volume={193},
   date={2021},
   number={2},
   pages={347--405},
}

\bib{Bir20}{article}{
   author={Birkar, Caucher},
   title={Generalised pairs in birational geometry},
   journal={EMS Surv. Math. Sci.},
   volume={8},
   date={2021},
   number={1-2},
   pages={5--24},
}

\bib{BZ16}{article}{
   author={Birkar, Caucher},
   author={Zhang, De-Qi},
   title={Effectivity of Iitaka fibrations and pluricanonical systems of
   polarized pairs},
   journal={Publ. Math. Inst. Hautes \'{E}tudes Sci.},
   volume={123},
   date={2016},
   pages={283--331},
}


\bib{C20}{article}{
   author={Chen, Guodu},
   title={Boundedness of $n$-complements for generalized pairs},
   eprint={arXiv:2003.04237}
}

\bib{CH21}{article}{
   author={Chen, Guodu},
   author={Han, Jingjun},
   title={Boundedness of $(\epsilon, n)$-complements for surfaces},
   journal={Adv. Math.},
   volume={383},
   date={2021},
   pages={Paper No. 107703, 40},
}


\bib{CX20}{article}{
   author={Chen, Guodu},
   author={Xue, Qingyuan},
   title={Boundedness of $(\epsilon,n)$-complements for projective
   generalized pairs of Fano type},
   journal={J. Pure Appl. Algebra},
   volume={226},
   date={2022},
   number={7},
   pages={Paper No. 106988, 11},
}

\bib{EM04}{article}{
   author={Ein, Lawrence},
   author={Musta{\c{t}}{\v{a}}, Mircea},
   title={Inversion of adjunction for local complete intersection varieties},
   journal={Amer. J. Math.},
   volume={126},
   date={2004},
   number={6},
   pages={1355--1365},
}

\bib{EMY03}{article}{
   author={Ein, Lawrence},
   author={Musta{\c{t}}{\u{a}}, Mircea},
   author={Yasuda, Takehiko},
   title={Jet schemes, log discrepancies and inversion of adjunction},
   journal={Invent. Math.},
   volume={153},
   date={2003},
   number={3},
   pages={519--535},
}

\bib{FM}{article}{
   author={Filipazzi, Stefano},
   author={Moraga, Joaqu\'{\i}n},
   title={Strong $(\delta,n)$-complements for semi-stable morphisms},
   journal={Doc. Math.},
   volume={25},
   date={2020},
   pages={1953--1996},
}

\bib{Ful98}{book}{
   author={Fulton, William},
   title={Intersection theory},
   series={Ergebnisse der Mathematik und ihrer Grenzgebiete. 3. Folge. A
   Series of Modern Surveys in Mathematics [Results in Mathematics and
   Related Areas. 3rd Series. A Series of Modern Surveys in Mathematics]},
   volume={2},
   edition={2},
   publisher={Springer-Verlag, Berlin},
   date={1998},
   pages={xiv+470},
}

\bib{GW10}{book}{
   author={G{\"o}rtz, Ulrich},
   author={Wedhorn, Torsten},
   title={Algebraic geometry I},
   series={Advanced Lectures in Mathematics},
   publisher={Vieweg + Teubner, Wiesbaden},
   date={2010},
}



\bib{HMX14}{article}{
   author={Hacon, Christopher D.},
   author={McKernan, James},
   author={Xu, Chenyang},
   title={ACC for log canonical thresholds},
   journal={Ann. of Math. (2)},
   volume={180},
   date={2014},
   number={2},
   pages={523--571},
}


\bib{HMo}{article}{
   author={Hacon, Christopher},
   author={Moraga, Joaqu\'{\i}n},
   title={On weak Zariski decompositions and termination of flips},
   journal={Math. Res. Lett.},
   volume={27},
   date={2020},
   number={5},
   pages={1393--1421},
}



\bib{HL22}{article}{
   author={Han, Jingjun},
   author={Li, Zhan},
   title={Weak Zariski decompositions and log terminal models for
   generalized pairs},
   journal={Math. Z.},
   volume={302},
   date={2022},
   number={2},
   pages={707--741},
}


\bib{HLS}{article}{
   author={Han, Jingjun},
   author={Liu, Jihao},
   author={Shokurov, V. V.},
   title={ACC for minimal log discrepancies of exceptional singularities},
   eprint={arXiv:1903.04338v2}
}


\bib{HanLuo}{article}{
   author={Han, Jingjun},
   author={Luo, Yujie},
   title={On boundedness of divisors computing minimal log discrepancies for surfaces},
   eprint={arXiv:2005.09626v2}
      date={2020},
}



\bib{Kaw14}{article}{
   author={Kawakita, Masayuki},
   title={Discreteness of log discrepancies over log canonical triples on a
   fixed pair},
   journal={J. Algebraic Geom.},
   volume={23},
   date={2014},
   number={4},
   pages={765--774},
}

\bib{Kaw15}{article}{
   author={Kawakita, Masayuki},
   title={A connectedness theorem over the spectrum of a formal power series
   ring},
   journal={Internat. J. Math.},
   volume={26},
   date={2015},
   number={11},
   pages={1550088, 27},
}

\bib{Kaw}{article}{
   author={Kawakita, Masayuki},
   title={On equivalent conjectures for minimal log discrepancies on smooth
   threefolds},
   journal={J. Algebraic Geom.},
   volume={30},
   date={2021},
   number={1},
   pages={97--149},
}

\bib{Kaw88}{article}{
   author={Kawamata, Yujiro},
   title={Crepant blowing-up of $3$-dimensional canonical singularities and
   its application to degenerations of surfaces},
   journal={Ann. of Math. (2)},
   volume={127},
   date={1988},
   number={1},
   pages={93--163},
}

\bib{Kaw92}{article}{
   author={Kawamata, Yujiro},
   title={The minimal discrepancy coefficients of terminal singularities in dimension 3},
   journal={Appendix to V. V. Shokurov, 
   \textit{Three-dimensional log perestroikas}, Izv. Ross. Akad. Nauk Ser. Mat.},
   volume={56},
   date={1992},
   number={1},
   pages={105--203},
   issn={0373-2436},
}


\bib{KMM}{article}{
   author={Kawamata, Yujiro},
   author={Matsuda, Katsumi},
   author={Matsuki, Kenji},
   title={Introduction to the minimal model problem},
   journal={Avd. Stud. Pure Math.},
   volume={10},
   date={1987},
   pages={283--360},
}


\bib{Kol96}{book}{
   author={Koll{\'a}r, J{\'a}nos},
   title={Rational curves on algebraic varieties},
   series={Ergebnisse der Mathematik und ihrer Grenzgebiete. 3. Folge. A
   Series of Modern Surveys in Mathematics}, 
   volume={32},
   publisher={Springer-Verlag, Berlin},
   date={1996},
}

\bib{Kol13}{book}{
   author={Koll{\'a}r, J{\'a}nos},
   title={Singularities of the minimal model program},
   series={Cambridge Tracts in Mathematics},
   volume={200},
   note={With a collaboration of S\'andor Kov\'acs},
   publisher={Cambridge University Press, Cambridge},
   date={2013},
}

\bib{KM98}{book}{
   author={Koll{\'a}r, J{\'a}nos},
   author={Mori, Shigefumi},
   title={Birational geometry of algebraic varieties},
   series={Cambridge Tracts in Mathematics},
   volume={134},
   publisher={Cambridge University Press, Cambridge},
   date={1998},
}

\bib{Mar96}{article}{
   author={Markushevich, Dimitri},
   title={Minimal discrepancy for a terminal cDV singularity is $1$},
   journal={J. Math. Sci. Univ. Tokyo},
   volume={3},
   date={1996},
   number={2},
   pages={445--456},
}

\bib{Mor}{article}{
   author={Moraga, Joaqu\'{\i}n},
   title={Termination of pseudo-effective 4-fold flips},
   eprint={arXiv:1802.10202v2}
}

\bib{Mor2}{article}{
   author={Moraga, Joaqu\'{\i}n},
   title={On termination of flips and fundamental groups},
   eprint={arXiv:2109.05608v1}
}

\bib{Mor21}{article}{
   author={Moraga, Joaqu\'{\i}n},
   title={On minimal log discrepancies and koll\'{a}r components},
   journal={Proc. Edinb. Math. Soc. (2)},
   volume={64},
   date={2021},
   number={4},
   pages={982--1001},
}


\bib{Mor3}{article}{
   author={Moraga, Joaqu\'{\i}n},
   title={Minimal log discrepancies of regularity one},
   journal={Int. Math. Res. Not. IMRN},
   date={2023},
}





\bib{MN18}{article}{
   author={Musta\c{t}\u{a}, Mircea},
   author={Nakamura, Yusuke},
   title={A boundedness conjecture for minimal log discrepancies on a fixed
   germ},
   conference={
      title={Local and global methods in algebraic geometry},
   },
   book={
      series={Contemp. Math.},
      volume={712},
      publisher={Amer. Math. Soc., Providence, RI},
   },
   date={2018},
   pages={287--306},
}

\bib{Nak16a}{article}{
   author={Nakamura, Yusuke},
   title={On semi-continuity problems for minimal log discrepancies},
   journal={J. Reine Angew. Math.},
   volume={711},
   date={2016},
   pages={167--187},
}


\bib{Nak16b}{article}{
   author={Nakamura, Yusuke},
   title={On minimal log discrepancies on varieties with fixed Gorenstein
   index},
   journal={Michigan Math. J.},
   volume={65},
   date={2016},
   number={1},
   pages={165--187},
}


\bib{NS}{article}{
   author={Nakamura, Yusuke},
   author={Shibata, Kohsuke},
   title={Inversion of adjunction for quotient singularities},
   journal={Algebr. Geom.},
   volume={9},
   date={2022},
   number={2},
   pages={214--251},
}

\bib{NS2}{article}{
   author={Nakamura, Yusuke},
   author={Shibata, Kohsuke},
   title={Inversion of adjunction for quotient singularities II: Non-linear actions},
   eprint={arXiv:2112.09502v1}
}

\bib{Sho}{article}{
   author={Shokurov, V. V.},
   title={A.c.c. in codimension 2},
   date={1994, preprint},
}

\bib{Sho93}{article}{
   author={Shokurov, V. V.},
   title={$3$-fold log flips, With an appendix by Yujiro Kawamata},
   journal={Russian Acad. Sci. Izv. Math.},
   volume={40},
   date={1993},
   number={1},
   pages={95--202},
}

\bib{Sho96}{article}{
   author={Shokurov, V. V.},
   title={$3$-fold log models},
   note={Algebraic geometry, 4},
   journal={J. Math. Sci.},
   volume={81},
   date={1996},
   number={3},
   pages={2667--2699},
}

\bib{Sho04}{article}{
   author={Shokurov, V. V.},
   title={Letters of a bi-rationalist. V. Minimal log discrepancies and
   termination of log flips},
   journal={Tr. Mat. Inst. Steklova},
   volume={246},
   date={2004},
   number={Algebr. Geom. Metody, Svyazi i Prilozh.},
   pages={328--351},
   translation={
      journal={Proc. Steklov Inst. Math.},
      date={2004},
      number={3 (246)},
      pages={315--336},
   },
}

\end{biblist*}
\end{bibdiv}
\end{document}